\documentclass[12pt]{article}

\usepackage[dvips]{graphicx}
\usepackage{amsfonts}
\usepackage{amsmath}
\usepackage{mathtools}
\usepackage{amssymb}
\usepackage{amscd}
\usepackage{color}
\usepackage{amsthm}
\usepackage{indentfirst}
\usepackage[hmargin=3cm,vmargin=3cm]{geometry}

\usepackage{enumitem}

\newtheorem{thm}{Theorem}[section]
\newtheorem{thm*}{Theorem}
\newtheorem{prop}[thm]{Proposition}
\newtheorem{lma}[thm]{Lemma}

\newtheorem{clm}[thm]{Claim}

\theoremstyle{definition}

\newtheorem{df}[thm]{Definition} 

\theoremstyle{remark}
\newtheorem*{pf}{Proof}

\newtheorem{rmk}[thm]{Remark} 
\newtheorem{example}[thm]{Example}

\newtheorem{conj}[thm]{Conjecture}
\newtheorem{constraint}[thm]{Constraint}
\newtheorem{problem}[thm]{Question}

\newcommand{\R}{{\mathbb{R}}}
\newcommand{\Z}{{\mathbb{Z}}}
\newcommand{\C}{{\mathbb{C}}}
\newcommand{\Q}{{\mathbb{Q}}}

\newcommand{\bs}{\bigskip}

\newcommand{\del}{\partial}

\newcommand{\sm}[1]{C^\infty(#1)}

\newcommand{\vareps}[1]{\varepsilon_{#1}}

\newcommand\sign{\operatorname{sign}}

\newcommand{\Sum}{\Sigma}
\newcommand{\G}{\mathcal{G}}

\newcommand{\A}{\mathcal{A}}
\newcommand{\K}{\mathcal{K}}

\newcommand{\cL}{\mathcal{L}}

\newcommand{\til}[1]{\widetilde{#1}}

\newcommand{\om}{\omega}
\newcommand{\al}{\alpha}
\newcommand{\la}{\lambda}
\newcommand{\Om}{\Omega}
\newcommand{\ga}{\gamma}
\newcommand{\eps}{\epsilon}

\newcommand{\oa}{a}
\newcommand{\ob}{b}
\newcommand{\oc}{c}

\newcommand{\cA}{\mathcal{A}}
\newcommand{\cB}{\mathcal{B}}
\newcommand{\cC}{\mathcal{C}}
\newcommand{\cD}{\mathcal{D}}

\newcommand{\cF}{\mathcal{F}}
\newcommand{\cG}{\mathcal{G}}
\newcommand{\cH}{\mathcal{H}}
\newcommand{\cI}{\mathcal{I}}
\newcommand{\cJ}{\mathcal{J}}

\newcommand{\cS}{\mathcal{S}}

\newcommand{\cP}{\mathcal{P}}

\newcommand{\cN}{\mathcal{N}}

\newcommand{\rP}{\mathrm{P}}
\newcommand{\rR}{\mathrm{R}}
\newcommand{\rQ}{\mathrm{Q}}
\newcommand{\rT}{\mathrm{T}}
\newcommand{\rS}{\mathrm{S}}

\newcommand{\Foabc}{{\mathrm{Free}\langle \oa,\ob, \oc \rangle}}

\DeclareMathOperator{\trace}{\mathrm{trace}}

\DeclareMathOperator{\Ham}{\mathrm{Ham}}
\DeclareMathOperator{\Hom}{\mathrm{Hom}}
\DeclareMathOperator{\Symp}{\mathrm{Symp}}

\DeclareMathOperator{\Osc}{{\mathcal{E}}}
\DeclareMathOperator{\Max}{{\mathcal{E}^{+}}}
\DeclareMathOperator{\Min}{{\mathcal{E}^{-}}}
\DeclareMathOperator{\conjugation}{\mathrm{conj}}

\def\Hk{H^{(k)}}
\def\Fk{F^{(k)}}
\def\Gk{G^{(k)}}
\def\H2{H^{(2)}}

\DeclareMathOperator{\Aut}{\mathrm{Aut}}

\DeclareMathOperator{\aut}{\mathrm{aut}}

\DeclareMathOperator{\pow}{\mathrm{powers}}
\DeclareMathOperator{\Pow}{\mathrm{Powers}}

\newcommand{\id}{{\bf{1}}}
\newcommand{\N}{{\mathbb{N}}}

\usepackage{hyperref}

\begin{document}

\title{Autonomous Hamiltonian flows, Hofer's geometry and persistence modules}

\renewcommand{\thefootnote}{\alph{footnote}}

\author{\textsc Leonid Polterovich$^{a}$   and Egor Shelukhin}

\footnotetext[1]{Partially supported by the European Research Council Advanced grant 338809.}

\date{\today}

\maketitle

\abstract{We find robust obstructions to representing a Hamiltonian diffeomorphism as a full $k$-th power, $k \geq 2,$ and in particular, to including it into a one-parameter subgroup. The robustness is understood in the sense of Hofer's metric. Our approach is based on the theory of persistence modules applied in the context of filtered Floer homology. We present applications to geometry and dynamics of Hamiltonian diffeomorphisms.}

\tableofcontents

\section{Introduction and main results}

This paper deals with robust obstructions to representing a Hamiltonian diffeomorphism as a full $k$-th power, $k \geq 2,$ and in particular, to including it into a one-parameter subgroup. The robustness is understood in the sense of Hofer's metric. These obstructions yield applications to geometry and dynamics of Hamiltonian diffeomorphisms. On the geometric side, we prove that for certain symplectic manifolds the complement of the set of Hamiltonian diffeomorphisms admitting a root of order $k$ (and {\it a fortiori}, of autonomous Hamiltonian diffeomorphisms) contains an arbitrarily large ball. We also establish a result of a dynamical flavor providing a symplectic
take on Palis'  dictum ``Vector fields generate few diffeomorphisms": for symplectically aspherical manifolds  the subset of non-autonomous Hamiltonian diffeomorphisms contains a $C^\infty$-dense Hofer-open subset.

Our approach is based on the theory of persistence modules
applied in the context of filtered Floer homology enhanced with a special periodic automorphism.
The latter is induced by the natural action (by conjugation) of a Hamiltonian diffeomorphism $\phi$ on the
Floer homology of its power $\phi^k$.

\subsection{Distance to autonomous diffeomorphisms and $k$-th powers.}
Let $d$ be the Hofer metric on the group $\Ham(M,\omega)$ of $C^\infty$-smooth Hamiltonian diffeomorphisms
of a closed symplectic manifold $(M,\omega)$. We recall that $d(f,g)$ for $f,g \in \Ham$ is defined as the infimum over all Hamiltonian paths $\{\phi_t\}_{t \in [0,1]}$ with $\phi_0 =f,\; \phi_1 = g,$ of the quantity \[\int_{0}^{1} (\max_M H_t - \min_M H_t) dt, \] where $H_t(x) = H(t,x)$ is the time-dependent Hamiltonian that generates the path $\{\phi_t\}.$ Recall that a Hamiltonian diffeomorphism is called {\it autonomous} if it is generated by a time independent Hamiltonian function,
or, in other words, it can be included into a one parameter subgroup of $\Ham$. Denote by $\Aut \subset \Ham$
the set of all autonomous Hamiltonian diffeomorphisms. Define the quantity
$$\aut(M,\omega):= \sup_{\phi \in \Ham} d(\phi, \Aut)\;.$$


\medskip
\noindent
\begin{conj}\label{conj-main} $\aut(M,\omega)=+\infty$ for all closed symplectic manifolds $(M,\omega)$.
\end{conj}

\medskip
\noindent
In the present paper we make a first step towards this conjecture. Recall that $(M,\omega)$ is called symplectically
aspherical if the class of the symplectic form and the first Chern class vanish on $\pi_2(M)$.

\medskip
\noindent
\begin{thm}\label{thm-main}
Let $\Sigma$ be a closed oriented surface of genus $\geq 4$ equipped with an area form $\sigma$.
Then for every closed symplectically aspherical manifold $(M,\omega)$
$$\aut(\Sigma \times M, \sigma \oplus \omega)= +\infty\;.$$
\end{thm}

\medskip
\noindent In this theorem $M$ is allowed to be the point.

\medskip
\noindent
Our main result is the following refinement of Theorem \ref{thm-main}. Let $k \geq 2$ be an integer.
       Write $\Pow_k =  \{\phi=\psi^k|\; \psi \in \Ham\}$ for the set of Hamiltonian diffeomorphisms admitting a
         root of order $k$ and denote
\[\pow_k := \sup_{\phi \in \Ham} d(\phi, \Pow_k)\;,\].

\medskip
\noindent
\begin{thm}\label{thm-main-squares}
Let $\Sigma$ be a closed oriented surface of genus $\geq 4$ equipped with an area form $\sigma,$ and $k \geq 2$ an integer.
Then for every closed symplectically aspherical manifold $(M,\omega)$
$$\pow_k(\Sigma \times M, \sigma \oplus \omega)= +\infty\;.$$
\end{thm}

Since for $p$ dividing $k,$ we have $\Pow_p \supset \Pow_k,$ it suffices to prove Theorem \ref{thm-main-squares} in the special case when $k$ is a prime.



\medskip
\noindent
A few remarks are in order. Since every autonomous diffeomorphism admits a root of any order, Theorem \ref{thm-main} is an immediate consequence of Theorem \ref{thm-main-squares}. Nevertheless, we state it separately since it provides a symplectic take on an important phenomenon in dynamical systems which will be discussed in the next section. Furthermore, it admits a (somewhat simpler) independent proof in the course of which we introduce a new invariant of Hamiltonian diffeomorphisms. The split form $\Sigma \times M$ of the symplectic manifold is crucial for producing specific examples of Hamiltonian diffeomorphisms which lie arbitrarily far from $\Aut$ and $\Pow_p$. More comments on the proof of both theorems can be found in Sections \ref{subsec-constraints} and \ref{subsec-fullsquares} of the Introduction.

\subsection{``Vector fields generate few diffeomorphisms"} The phenomenon described in the title of this subsection,
which is valid for various classes of dynamical systems, has been studied for more than four decades starting from Palis's seminal work \cite{Palis}. Nowadays it is known, for instance, that non-autonomous $C^1$-smooth symplectomorphisms (not necessarily Hamiltonian) form a $C^1$-open and dense set in the group of all $C^1$-symplectomorphisms, see Arnaud, Bonatti, and  Crovisier \cite{ABC} and Bonatti, Crovisier, Vago, and Wilkinson \cite[p. 929]{BCVW}.  It is an easy consequence of a work by Ginzburg and G\"{u}rel \cite{GG} that a similar statement holds true for $C^{\infty}$-smooth Hamiltonian diffeomorphisms of certain
symplectic manifolds, for instance of symplectically aspherical ones. The method developed in the present paper yields the following result.

\medskip
\noindent
\begin{thm}\label{thm-generic}
For a closed symplectically aspherical manifold, the set $\Ham \setminus \Aut$ contains
a $C^\infty$-dense subset which is open in the topology induced by Hofer's metric.
\end{thm}

\medskip
\noindent
In this context, Conjecture \ref{conj-main} states that the set $\Ham \setminus \Aut$
contains a Hofer ball of an arbitrary large radius.

\medskip
\noindent
\begin{problem}\label{prob-2}{\rm
It sounds likely that, by a slight modification of the tools developed in this paper, one can show that $\Ham \setminus \Pow_k,$ $k \geq 2,$ contains a $C^\infty$-dense Hofer open subset.}
\end{problem}

\subsection{Further motivation}
There is a couple of additional circumstances which triggered our interest in Hofer's geometry of the set $\Ham \setminus \Aut$. For certain closed symplectic manifolds $(M,\omega)$ one can show that a generic
smooth function $F \in C^\infty(M)$ generates a Hamiltonian flow $\{f_t\}, t \in \R$ with
$$ c|s-t| < d(f_t,f_s) \leq c^{-1}|s-t|$$ for some $c=c(F)>0$, that is a generic one-parameter subgroup of
$\Ham$ is a quasi-geodesic. This holds true, for instance, for symplectically
aspherical manifolds and for the sphere $S^2$ (see e.g. Section 6.3.1 in \cite{PR-book}).
For split manifolds of the form $\Sigma \times M,$ where $\Sigma$ is a surface of genus $\geq 4$ and $M$ is
symplectically aspherical,  Theorem \ref{thm-main} rules out existence of a constant $r >0$ such that the group $\Ham$ lies in the union of tubes of radius $r$ around these quasi-geodesics. In contrast to this, in the case $M=S^2$ (where Conjecture \ref{conj-main} is still open) we even cannot exclude existence of such a tube around {\it a specific} quasi-geodesic one-parameter subgroup constructed in \cite{P-diameter}.\footnote{This problem has been formulated by Misha Kapovich and L.P. at an Oberwolfach meeting
in 2006.}

\medskip

Autonomous Hamiltonian diffeomorphisms give rise to an interesting biinvariant metric on the group $\Ham(M,\omega)$
called the {\it autonomous metric}. It has been studied by Gambaudo and Ghys \cite{GaGh}, Brandenbursky and Kedra \cite{BK} and Brandenbursky and Shelukhin \cite{BS}. Observe that the set $\Aut \subset \Ham$ is conjugation invariant. Since the group $\Ham$ is simple \cite{Banyaga}, the normal subgroup generated by $\Aut$ coincides with $\Ham$. In other words, every Hamiltonian diffeomorphism $f$ can be decomposed into a product of $k$ autonomous ones. The autonomous distance $d_{aut}(\id, f)$ from the identity to $f$ is defined as the minimal number $k$
of terms in such a decomposition. The above-mentioned papers prove unboundedness of $\Ham$ with respect to
$d_{aut}$ on most surfaces. The comparison between the autonomous and the Hofer metrics is far from being
understood (see \cite{BK} for a discussion). For instance, let us assume that $M$ is a closed
oriented surface of genus $\geq 4$. Let $S(k) \subset Ham$ be the sphere of radius  $k \in \N$ with respect
to the autonomous metric. Theorem \ref{thm-main} shows that $d(f,S(1))$ can be made arbitrarily large,
and in fact, as we shall see in the course of the proof, such $f$'s can be chosen from $S(2)$.
However we cannot prove whether $d(f, S(2))$ can be made arbitrarily large, and at the moment this problem
sounds out of reach.

Let us mention finally that for most surfaces, an analogue of the quantity $\aut(\Sigma)$ for right-invariant "hydrodynamical" metrics on $\Ham(\Sigma)$ is infinite, see \cite{BS}.

\subsection{Constraints on autonomous diffeomorphisms}\label{subsec-constraints}
Even though Theorem \ref{thm-main} on autonomous diffeomorphisms is a formal consequence of Theorem \ref{thm-main-squares} on full $p$-th powers, we present an independent proof. One of the reasons is that
it uses less sophisticated tools. Here one can avoid the language of persistence modules
even though it provides a useful intuition. Additionally, Floer homology with $\Z_2$-coefficients
does the job and hence one can ignore orientation issues in Floer theory. Another reason for presenting an
independent argument is that in the cause of the proof we introduce a new invariant of Hamiltonian diffeomorphisms,
the so-called spectral spread, which is useful in its own right. In contrast to this, the proof of Theorem \ref{thm-main-squares} involves the theory of persistence modules in an essential way.

Let us outline our approach to Theorem \ref{thm-main}. Let us mention that its proof
simplifies significantly for the case of surfaces. It is instructive to discuss both the two- and the higher-dimensional cases.

We start with preliminaries on Hamiltonian flows and diffeomorphisms on a closed symplectic manifold $(M,\omega)$. Let $F_t(x)$, $t\in \R$ be a $1$-periodic Hamiltonian function generating a Hamiltonian flow $\{f_t\}$ with the time one map $\phi=f_1 \in \Ham(M,\omega)$. The $1$-periodicity of the Hamiltonian yields
\begin{equation}\label{eq-vsp-per-1}
f_{t+1}=f_tf_1 \;\;\forall t \in \R\;.
\end{equation}
The $1$-periodic orbits $x(t)=f_tx$ of the flow correspond to fixed points $x=x(0)$ of $\phi$.

A well known (and non-trivial) fact is that the free homotopy class $\alpha=\alpha(\phi,x)$ of the orbit $x(t)$ depends only on the fixed point $x$ of $\phi$, but not on the specific Hamiltonian flow $\{f_t\}$ generating $\phi$. The class $\alpha$ is called {\it primitive} if it cannot be represented by a multiply-covered loop, i.e., by a map $$S^1 \to S^1 \to M\;,$$ where the left arrow is a non-trivial cover.

Let $\alpha$ be a primitive free homotopy class on a closed orientable surface $\Sigma$ of genus $\geq 2$ equipped with an area form.   We call $\alpha$ {\it simple} if it can be represented by an embedded closed curve. Note that the non-constant periodic orbits of any autonomous flow on $\Sigma$ are either multiply covered or have no self-intersections by the uniqueness theorem for ODEs.

\medskip
\noindent
\begin{constraint} \label{constraint-1}   If a Hamiltonian diffeomorphism $\phi$ of a closed
symplectic surface $\Sigma$ possesses a fixed point $x$ such that the free homotopy class
$\alpha(\phi,x)$ is primitive but non-simple, then $\phi$ is not autonomous.
\end{constraint}

\medskip
\noindent
Let us mention that this constraint is purely two dimensional as all free homotopy classes of loops are simple in higher dimensions.

\medskip
\noindent Let us return to the case of general closed symplectic manifolds.
For an integer $k \geq 2$ the Hamiltonian
\begin{equation}\label{eq-Fk}
F^{(k)}= kF_{kt}
\end{equation}
generates the flow $\{f_{kt}\}$ with the time one map $\phi^k$. The $1$-periodic orbits of this flow
have the form $x(t)=f_{kt}x,$ where $\phi^kx=x$. Next comes the following observation:

\medskip
\noindent
{\bf ($\spadesuit$)} For every $i \in \N$,  the loop $x(t+i/k)$
is again a closed orbit of $\{f_{kt}\}$ corresponding to the fixed point $\phi^ix$ of $\phi^k$.
\medskip
\noindent

The fixed point $x$ of $\phi^k$  is called {\it primitive}\footnote{We remark that elsewhere in the literature (cf. \cite{Ginzburg,GG}) such fixed points are called {\it simple}, but this terminology would cause unnecessary confusion in the context of this paper.} if all $\phi^ix$ are pair-wise distinct for
$i=0,...,k-1$. For instance, if the free homotopy class $\alpha(\phi^k,x)$ is primitive,
the point $x$ is primitive as well (but, in general, not vice versa!).

The fixed point $x$ of $\phi^k$ is called {\it isolated} if
there are no other fixed points of $\phi^k$ in a sufficiently small neighborhood of $x$.
For instance, if $x$ is {\it non-degenerate}, i.e., the differential of $\phi^k$ at $x$ does not have $1$ as an eigenvalue, then $x$  is isolated.

If the Hamiltonian $F$ is time independent, primitive fixed points of $\phi^k$
with $k \geq 2$ are never isolated. They necessarily appear in $S^1$-families $\{f_tx\}$,
$t \in S^1= \R/(k\Z)$.

\medskip
\noindent
\begin{constraint} \label{constraint-2} If a $k$-th power ($k \geq 2$) $\phi^k$ of a
Hamiltonian diffeomorphism $\phi$ possesses an isolated primitive fixed point, $\phi$ is not
autonomous.
\end{constraint}

\medskip

Our next task is to refine Constraints \ref{constraint-1} and \ref{constraint-2} so that
they become robust with respect to (not necessarily small) $C^0$-perturbations of the Hamiltonian $F_t$ generating $\phi$. To this end we use the machine of filtered Floer homology.

For a free homotopy class $\alpha$ on $M$, denote by $\cL_\alpha M$ the space of loops $S^1 \to M$ representing $\alpha$. Under certain assumptions on $M$ and $\alpha$, which will be stated precisely later, every Hamiltonian $F$ as above determines {\it an action functional} $\A_{F}: \cL_\alpha M \to \R$,
$$z(t) \mapsto \int_0^1 F_t(z(t)) dt -\int_{\overline{z}} \omega\;.$$
Here $\overline{z}$ is (any) annulus connecting $z(t)$ with the loop playing the role of a base
point in  $\cL_\alpha M$, chosen once and forever. The critical points of $\A_{F}$ correspond to $1$-periodic orbits of $f_t$ in the class $\alpha$. {\it Filtered Floer homology} $HF^{(a,b)}(\phi)_\alpha$, $b >a$, is, roughly speaking, the Morse homology of the space $$\{\A_F < b\}/\{\A_F < a\}\;.$$ In particular,
$HF^{(a,b)}(\phi)_\alpha \neq 0$ yields the existence of closed orbits of $\{f_t\}$ in the class $\alpha$. An important feature of this existence mechanism for closed orbits is its robustness with respect to $C^0$-perturbations
of the Hamiltonian $F$. In particular, if the map $HF^{(a,b)}(\phi)_\alpha \to HF^{(a+c,b+c)}(\phi)_\alpha$ induced by inclusions of sublevel sets is non-zero, then every Hamiltonian flow $g_t$ generated by a Hamiltonian $G_t$ with
$$\max_{t \in [0,1],z\in M} |F_t(z) -G_t(z)| < c/2$$  possesses a $1$-periodic trajectory in the class $\alpha$. This readily yields that every Hamiltonian flow generating a diffeomorphism $\psi$ with $d(\phi,\psi) < c/2$
must have a closed orbit in the class $\alpha$.

In the context of Constraint \ref{constraint-1} above,  we produce a sequence of Hamiltonian diffeomorphisms
$\phi_i$, $i \to \infty$ of a surface $\Sigma$ and primitive non-simple free homotopy classes $\alpha_i$ on $\Sigma$ so that the filtered Floer homology of $\phi_i$ does not vanish in some window of width $c_i \to \infty$. In this way we conclude that $d(\phi_i,\Aut) \to \infty$, and hence $\aut(\Sigma)=+\infty$.

\medskip

In order to put Constraint \ref{constraint-2} into the framework of filtered Floer homology,
let us note that the main feature of the action functional  associated with a time-independent
Hamiltonian  is that it is invariant under the canonical circle action
\begin{equation}\label{eq-s1action}
z(t) \mapsto z(t+s),\;\; s \in S^1\;,
\end{equation}
on the loop space $\cL_\alpha M$.  Nowadays a number of tools for tackling this action are available,
such as equivariant Floer homology  and the Batalin-Vilkovisky operator\footnote{The original unsuccessful attempt of the authors was to use the BV operator.}
(see e.g. \cite{BO}). The approach of the present paper is based on a trick\footnote{It was communicated to us by Paul Seidel.} which can be described as follows. For any (not necessarily autonomous) Hamiltonian $F$ create an artificial $\Z_k$-symmetry and then confront it with the $S^1$-symmetry inherent to autonomous Hamiltonians. More precisely,
fix an integer $k \geq 2$ and take the Hamiltonian $F^{(k)}$ from Equation
\eqref{eq-Fk} generating $\phi^k$. Look at the $S^1$-action \eqref{eq-s1action} and define the {\it the loop-rotation operator}
$$\rR_k: \cL_\alpha M \to \cL_\alpha M,\; z(t) \to z(t+1/k)$$
generating a cyclic subgroup $\Z_k \subset S^1$. A straightforward calculation shows that the action functional $\A_{F^{(k)}}$ is invariant under $\rR_k$. In particular, $\rR_k$ induces a filtration-preserving morphism $[\rR_k]$ of the filtered Floer homology of $F^{(k)}$.

For an illustration, assume that $\phi^k$ has a primitive isolated fixed point $x$. The closed orbits of
the flow $\{f_{kt}\}$ corresponding to the fixed points $\phi^ix$, $i=0,...,k-1,$ have the same action, say $a$, and represent the same  free homotopy class, say, $\alpha$. Assume that some action window $(a-c,a+c)$ does not contain any other critical values of $\A_{F^{(k)}}$. Then each $\phi^ix$, $i=0,...,k-1,$ defines an
element $\gamma_i:= [\phi^ix]$ in $HF^{(a-c,a+c)}(\phi^k)_{\alpha}$.
Furthermore, these elements are pairwise distinct. Observation {\bf ($\spadesuit$)} above readily yields
that $[\rR_k]$ cyclically permutes $\gamma_i$'s, i.e.,
$[\rR_k](\gamma_i)=\gamma_{i+1}$ (we put here $\gamma_k=\gamma_0$).

On the other hand, if $F$ is autonomous, the action functional $\A_{F^{(k)}}$ is invariant under circle action
\eqref{eq-s1action}, and hence it is invariant under the homotopy
$$\cL_\alpha M \to \cL_\alpha M, \;z(t) \mapsto z(t+s/k),\; s \in [0,1]$$
joining $\rR_k$ with the identity. Thus, in this case $[\rR_k] = \id$.
In the next sections we shall extract a lower bound for the distance between $\phi$ and $\Aut$
from, roughly speaking, the widths of the windows where $[\rR_k] \neq \id$.
A precise realization of this strategy occupies Sections \ref{Section: loop rotation}-\ref{Subsection: spread} below.

\medskip

\subsection{Constraints on full $p$-th powers}\label{subsec-fullsquares}
In the spirit of Milnor's obstruction \cite{MilnorCycles} for a diffeomorphism to have a square root that was applied in the Hamiltonian setting by Albers-Frauenfelder \cite{AlbersFrSquares}, we have the following constraint:

\noindent

\begin{constraint} \label{constraint-2'} Assume that the $k$-th power (where $k=p$ is a prime) $\phi^k$ of a
Hamiltonian diffeomorphism $\phi$ possesses only isolated fixed points in a primitive class $\al$.
Look at all non-parameterized (i.e., considered up to a cyclic shift) $k$-periodic orbits
$$S= \{x,\phi (x),\ldots, \phi^{k-1} (x) \}$$ of $\phi$, where $x$ is such a fixed point. If $\phi$ is a $p$-th power, the number $ \# \{S\}$ of these orbits is divisible by $p$.
\end{constraint}

This statement has the following linear-algebraic proof that generalizes best to the setting of Floer homology.

\begin{proof}(Constraint \ref{constraint-2'})
Let $\K =\Q_p$ be the cyclotomic field obtained from $\Q$ by adjoining a primitive root of unity $\zeta_p$ of order $p.$ For each set $$S= \{x,\phi (x),\ldots, \phi^{k-1} (x) \}\;,$$ $x$ a fixed point of $\phi^k$ in class $\al,$ consider the vector space $V_S$ over $\K$ freely generated by $S.$ In other words $V_S$ is dual to the $\K$-vector space of $\K$-valued functions on $S.$ Let $V = \bigoplus_{S} V_S$ be the $\K$-vector space freely generated by all fixed points of $\phi^k$ in class $\al.$ The map $x \mapsto \phi(x)$ preserves each set $S,$ and induces a linear map $A:V\to V$ that satisfies $A^k=1,$ and moreover $A = \bigoplus_S A_S$ with respect to the decomposition $V = \bigoplus_{S} V_S,$ where $A_S: V_S \to V_S$ is a linear map satisfying $A_S^k=1$. It is easy to see that for each $S,$ $A_S$ is diagonalizeable, and the $\zeta_p$-eigenspace $(V_S)_{\zeta_p}$ of $A_S$ is one-dimensional. Therefore $\dim V_{\zeta_p} = \sum_S \dim (V_S)_{\zeta_p} = \# \{S\}.$ If $\phi$ admits a root $\psi$ of order $p,$ that is $\phi=\psi^p,$ then $\psi$ induces a linear map $B:V \to V$ such that $B^p = A.$ Note that $B$ commutes with $A$ and hence preserves $V_{\zeta_p}.$ Moreover the restriction $B'$ of $B$ to $V_{\zeta_p}$ satisfies $(B')^p = \zeta_p \cdot \id.$ Therefore by algebraic Lemmas \ref{Lemma: dimension divisible by m} and \ref{Lemma: equation has no solution} below the number $ \# \{S\} = \dim V_{\zeta_p}$ is divisible by $p.$ This finishes the proof.
\end{proof}

\begin{rmk}\label{Remark: k=pn}
Constraint \ref{constraint-2'} generalizes to the case when $k$ is any integer divisible by $p.$ Both the proof from \cite{MilnorCycles} and the above proof generalize to this case. As a matter of curiosity we present another equivalent short proof of this result. Consider the set $Y$ of all fixed points of $\phi^k$ in class $\al.$ The map $x \mapsto \phi(x)$ defines a {\em free} $\Z_k$-action on $Y.$ Note that $\{S\} \cong Y/\Z_k.$ Any $p$-th root $\psi$ of $\phi$ determines a $G=\Z_{pk}$-action whose restriction to $\Z_k \cong p\Z_{pk} \subset \Z_{pk}$ agrees with the above $\Z_k$-action. Take $x \in Y.$ Consider the stabilizer $H = \mathrm{Stab}_G (x) \subset G$ in $G$ of $x.$ Clearly $H \cap p\Z_{pk} = \{0\},$ as the $\Z_k$-action is free. Note that $p\cdot H \subset H \cap p\Z_{pk} = \{0\},$ whence $H \subset \ker(\Z_{pk} \displaystyle\xrightarrow{p\cdot} p \Z_{pk}) = k \Z_{pk}.$ Whenever $k$ is divisible by $p,$ we have $k\Z_{pk} \subset p \Z_{pk},$ yielding $H = 0.$ Hence the $G$-action is free. Thererefore, since $\Z_k \cong p\Z_{pk} \subsetneq G,$ the induced $G/p\Z_{pk} \cong \Z_p$-action on $Y/\Z_k$ is free, and hence
$ \# \{S\} = \# (Y/\Z_k)$ is divisible by $p$. This concludes the proof.
\end{rmk}

\subsection{Hamiltonian diffeomorphisms and persistence modules}
The Floer homological version of Constraint \ref{constraint-2'} is based on the theory of
one-parametric persistence modules (see \cite{Ghrist,CarlssonTopologyData} for a survey). Let us sketch it very briefly leaving details for Section \ref{Subsection: enriched persistence modules}.

A {\it barcode} $\cB =\{(I_j,m_j)\}_{1 \leq j \leq N}$ is a finite collection of intervals (or bars) $I_j=(a_j,b_j]$, $a_j \in \R$, $b_j \in \R \cup +\infty$
with multiplicities $m_j \in \N$. We say that two barcodes $\cB$ and $\cC$ are $\delta$-matched, $\delta > 0$, if after erasing some bars of lengths $< 2\delta$ in $\cB$ and $\cC$, the remaining ones can be matched bijectively so that the endpoints of the corresponding intervals lie at a distance $< \delta$ from one another.
The {\it bottleneck distance} $d_{bottle}(\cB,\cC)$ is defined as the infimum of such $\delta$. Note that if
$\cB$ and $\cC$ have a different number of infinite rays, the bottleneck distance between them is infinite.

A {\it persistence module} $V$ is a collection of finite-dimensional vector spaces $V_t$, $t \in \R$ over $\K$
 equipped with morphisms $\pi_{st}:V_s \to V_t$, $s < t$ which satisfy $\pi_{sr} = \pi_{tr}\circ\pi_{st}$ for all $s<t<r$. It is assumed in addition that that $V_t =0$ for $t \ll 0$ and that the morphisms $\pi_{st}$
satisfy certain regularity assumptions. According to the  {\it structure theorem}, for every persistence
module $V$ there exists unique barcode $\cB(V)= \{(I_j,m_j )\}$ so that $V \cong \oplus_{j=1}^N (Q(I_j))^{m_j}$.
Here the building block $Q(I) = (\{Q(I)_t\}, \theta)$ is given by $(Q(I))_t = \K$ for $t \in I$ and $(Q(I))_t = 0$ otherwise, while the morphisms $\theta$ are the identity maps within $I$ and zeroes otherwise.

Given a (sufficiently non-degenerate) Hamiltonian diffeomorphism $\phi$ of a closed aspherical
(or, if required, atoroidal) symplectic manifold $M$, one can associate to it a number of canonical persistence
modules $V(\phi)$ coming from Floer theory. Our guiding principle is that the resulting mapping
\begin{equation}\label{eq-guiding}
(\Ham(M), d) \to (\text{Barcodes}, d_{bottle}),\;\; \phi \mapsto \cB(V(\phi))
\end{equation}
is Lipschitz with respect to the Hofer and the bottleneck distances. Composing this map
with a real-valued Lipschitz function $\varsigma$ on the space of barcodes, one gets a numerical invariant
$\varsigma^\sharp(\phi):= \varsigma(\cB(V(\phi)))$ of Hamiltonian diffeomorphisms which is robust in Hofer's metric. Varying persistence modules $V(\phi)$ and functions $\varsigma$ yields a wealth of such invariants. \footnote{In order to prove that the map \eqref{eq-guiding} is Lipschitz, one uses a deep {\it isometry
theorem} between the interleaving distance on persistence modules and the bottleneck distance on barcodes, see \cite{BauerLesnick}.}

As a warm up, take $V(\phi):= HF^{(-\infty,a)}(\phi)_\alpha$, where $\alpha$ is the class of the point.
For a barcode $\cB$ let $c_1 \leq \dots \leq c_N$ be the endpoints of infinite rays sorted in the increasing order. One readily checks
that for $\varsigma(\cB):= c_k$, the corresponding value $\varsigma^\sharp(\phi)$ is {\it a spectral invariant}
of the Hamiltonian diffeomorphism $\phi$ (see Schwarz, \cite{SchwarzAspherical}).
Alternatively, taking $\varsigma$ to be the maximal length of a finite bar in $\cB$, we recover {\it the boundary depth} of $\phi$ as defined by Usher in \cite{UsherBoundaryHofer}. \footnote{Our understanding of this picture appeared in discussions with Michael Usher and Jun Zhang. For its extension to general symplectic manifolds we refer
the reader to a forthcoming paper by Usher and Zhang.}

In order to produce a Floer homological version of Constraint \ref{constraint-2'}, we  work (again) over the cyclotomic field $\K = \Q(\zeta_p)$, where $p$ is a prime, and look at the $\Z_p$-action of the loop rotation operator $A:=[\rR_p]$ on  $HF^{(-\infty,a)}(\phi^p)_\alpha$. Here $\alpha$ is a primitive class of free loops. The persistence module $L(\phi)$ of interest is the $\zeta_p$-eigenspace of this action. The Lipschitz function $\varsigma$ on the space of barcodes is defined, roughly speaking, as the length of the maximal interval in $\cB$ whose multiplicity is {\bf not} divisible by $p$ "in a stable way". Put $\varsigma^\sharp = \mu_p$. 

Now, observe that for the full $p$-th power $\psi^p$, the loop rotation operator $[\rR_{p^2}]$ associated to  $\psi$ induces a morphism $B$ of $HF^{(-\infty,a)}(\phi^p)_\alpha$ with $B^p=A$. Its restriction $B'$  to
$L(\psi^p)$ satisfies $(B')^p = \zeta_p \cdot \id$. Arguing as in the proof of Constraint  \ref{constraint-2'} above, we conclude that the dimension of $L(\psi^p)_t$ is divisible by $p$ for every $t$. Thus
$\varsigma (\cB(L(\psi^p)))=\mu_p(\psi^p)=0$. The vanishing of $\mu_p(\psi^p)$ is the desired Floer theoretical version of Constraint \ref{constraint-2'}.
 
Combining these facts together we get that for any $\phi \in \Ham(M,\omega)$
$$\mu_p(\phi)= |\mu_p(\phi) -\mu_p(\psi^p)| \leq \mathrm{const} \cdot d(\phi,\psi^p)\;.$$
This yields a lower bound on Hofer's distance from $\phi$ to $\Pow_p$ which we use for the proof
of Theorem \ref{thm-main-squares}.

The details and precise formulations are presented in Section \ref{Example: pmi from a square}.
For reader's convenience, we include a primer on persistence modules in Section \ref{Subsection: enriched persistence modules}.

\subsection{A Hamiltonian egg-beater map}
Our final task is to present specific examples of Hamiltonian diffeomorphisms $\phi_\lambda$, $\lambda >0$ for which
the distances $d(\phi_\lambda, \Aut)$ and $d(\phi_\lambda,\Pow_p)$ become arbitrarily large as $\lambda \to +\infty$. To this end, we use intuition coming from the transition to chaos in Hamiltonian
dynamics. Observe that in dimension $2$, autonomous Hamiltonian flows provide the simplest examples of integrable systems of classical mechanics. In particular, they exhibit  deterministic
dynamical behavior. This suggests that one should look for $\phi_\lambda$ in the ``opposite" class of chaotic Hamiltonian diffeomorphisms. With this in mind, we choose $\phi_{\lambda}$ to be a (slightly modified) {\it egg-beater map} (see \cite{FrOt}), a cousin of well studied linked twist maps \cite{SOW}. We start with a pair of intersecting annuli (see Figure \ref{fig:two-annuli} in Section \ref{Section: example} below) each of which carries a shear flow with the profile given by a tent-like function whose graph is sketched on Figure \ref{Figure: graph of u_0}. The map $\phi_\lambda$ is the composition of time-$\lambda$ maps of these flows. It is known
(at least at the numerical level) that this map exhibits chaotic behavior  as $\lambda \to +\infty$ \cite{MMSO} and
possesses rich symbolic dynamics.  Next, we embed the union of the annuli into the sphere $S^2$, insert handles into each connected component of the complement of the annuli and extend $\phi_\lambda$ by the identity to the obtained surface $\Sigma$ of genus $\geq 4$. Even though the egg-beater map has a wealth of periodic orbits, their number in a specially chosen free homotopy class of loops on $\Sigma$ (here the handles enter the play)
becomes independent of $\lambda$. This enables us to perform the Floer-homological analysis of the egg-beater map based on persistence modules and loop rotation
operators and eventually to end up with the desired lower bounds on $d(\phi_{\lambda}, \Aut)$. Incidentally, the same construction yields bounds for $d(\phi_\lambda,\Pow_p)$.
Moreover these bounds survive stabilization: they remain valid for $\phi_{\lambda} \times \id$ on $\Sigma \times M$.
In this way we finish off the proof of Theorems \ref{thm-main} and \ref{thm-main-squares}, see Section \ref{Section: example} for details. Note that various versions of Hamiltonian egg-beater maps appeared in the context of
algebra and geometry of Hamiltonian diffeomorphisms in the works by M.~Kapovich \cite{Kapovich}, Brandenbursky and Kedra \cite{BK-egg}, Kim and Koberda \cite{KimKoberda}, and Khanevsky \cite{Khanevsky}.

\medskip

We conclude the introduction by mentioning that some other aspects of symplectomorphisms admitting a square root have been recently studied
by means of ``hard" symplectic topology in \cite{Seidel-new,AF-new}.


%



\medskip
\noindent
{\sc Organization of the paper:} In Section \ref{Section:HF} we set the stage and present the
necessary background from Floer theory.

\
 In Section \ref{Section: loop rotation} we introduce loop rotations
operators coming from the natural circle action on the loop space of a symplectic manifold and relate it to the
action by conjugation of a diffeomorphism $\phi$ on its power $\phi^k$.

In Sections \ref{Subsection: spread} and \ref{Example: pmi from a square} we define new invariants of
Hamiltonian diffeomorphisms, the so called spectral spread and its ramifications.
Ultimately, our construction involves the theory of one-parametric persistence modules. A primer on this theory is presented in Section \ref{Subsection: enriched persistence modules}.
In Section \ref{subsec-persist-invol} we focus on persistence modules enhanced with a $\Z_p$-action and translate the geometric "distance to $p$-th powers" problem
appearing in Theorem \ref{thm-main-squares} into algebraic language.

In Section \ref{Section: example} we design a Hamiltonian ``horseshoe" map and use
the spectral spread and the results on persistence modules for proving Theorems \ref{thm-main} on autonomous diffeomorphisms and Theorem \ref{thm-main-squares} on full $p$-th powers, respectively.

In Section \ref{sec-2D} we present a simple Floer-homological argument proving Theorem \ref{thm-main}
for surfaces.

In Section \ref{Section:Conley} we prove Theorem \ref{thm-generic} stating that for symplectically aspherical manifolds  the subset of non-autonomous Hamiltonian diffeomorphisms contains a $C^\infty$-dense Hofer-open subset.

Finally, in Section \ref{Section: discussion} we outline a generalization of our results to monotone symplectic manifolds and discuss open problems.


\section{Floer homology in a non-contractible class of orbits}\label{Section:HF}

We start with a description of the basic set-up of this paper. Consider a symplectically aspherical manifold $(M,\om),$ such that the class $\alpha \in \pi_0(\cL M)$ is symplectically atoroidal. Namely, put $\cL_\alpha M = p_{\pi_0}^{-1}({\al})$ for the preimage of $\al$ under the natural projection $p_{\pi_0}: \cL M \to \pi_0(\cL M).$  We require that for a loop $\rho$ in $\cL_\al M,$ considered as a map $\rho: T^2 \to M$ from the two-torus, \[\int_{T^2} \rho^* \om = 0,\] and  \[\int_{T^2} \rho^* c_1 = 0,\] where $c_1=c_1(TM,\om)$ denotes the first Chern class of $(M,\om),$ and similar conditions hold for loops in the class $pt_M \in \pi_0(\cL M)$ of contractible loops.

In such manifolds, as far as Floer theory is concerned (see below), a capped periodic orbit $(z,\overline{z})$ of the Hamiltonian flow in class $\al$ can be identified with its starting point $x = z(0).$ Indeed, by the two vanishing conditions, one sees that neither its action nor its index depend on the choice of capping $\overline{z}.$

Denote $\Max(H):=\int_0^1 \max_M H(t,\cdot)\,dt,$ $\Min(H):=\int_0^1 \min_M H(t,\cdot)\,dt,$ and $\Osc(H) = \Max(H) - \Min(H).$ For a segment $I=(a,b)$ and $d \in \R$ denote $I+d:=(a+d,b+d).$

We present the following general definition that organizes certain properties of Floer homology in action windows. For convenience, we shall use the language of (two-parametric) persistence modules.
Let us emphasise that the genuine applications of persistence modules to our story (cf. the title of this paper) appear later on in Section \ref{Section: invariant}, where we
deal with a much more developed theory of one-parametric persistence modules.

Consider the partially ordered set $\mathcal{I}$ of open intervals $(a,b)$ where $a \in \{-\infty\} \cup \R,$ $b \in \R \cup \{+\infty\},$ $a<b,$ with the partial order $I_1=(a_1,b_1) \leq I_2=(a_2,b_2)$ if $a_1 \leq a_2$ and $b_1 \leq b_2.$ Turn this partially ordered set into a category in the natural way, wherein $\Hom_\cI(I_1,I_2)$ has $1$ element if $I_1 \leq I_2$ and is empty otherwise. For a subset $S \subset \R$ we denote by $\cI^S$ the full subcategory of $\cI$ whose objects are intervals $(a,b)$ with $a,b \notin S.$

Given a base field $\K,$ denote by $Vect_\K$ the category of finite-dimensional graded vector spaces over $\K.$

\bs
\begin{df}\label{Definition: r2p persistence module}
We call a {\em restricted two-parametric ($r2p$) persistence module} of graded vector spaces over $\K$ a pair $(S,V)$ consisting of a compact subset $S \subset \R$ with empty interior, {\em the spectrum} of the module, and a functor $V:\cI^S \to Vect_\K:$ a collection of vector spaces $V^{(a,b)}$ for each open interval $(a,b),\;a,b \notin S$ and linear {\em comparison maps} \[j_{(a_1,b_1),(a_2,b_2)}: V^{(a_1,b_1)} \to V^{(a_2,b_2)},\] for each two intervals $(a_1,b_1),(a_2,b_2)$ such that $a_1 \leq a_2,\; b_1 \leq b_2,$ that satisfy \[j_{I_2,I_3} \circ j_{I_1,I_2} = j_{I_1,I_3},\] for $I_1 \leq I_2 \leq I_3.$ Moreover, for each $a < b < c$ outside $S,$ the data includes a prescribed long exact sequence \[V^{(a,b)} \xrightarrow{j_{(a,b),(a,c)}} V^{(a,c)} \xrightarrow{j_{(a,c),(b,c)}} V^{(b,c)} \xrightarrow{\delta_{a,b,c}} V^{(a,b)}[1],\]
where $V^{(a,b)}[1]$ stands for the shift of the grading of $V^{(a,b)}$ by $1$.

We further require $V$ to satisfy the following property: 
\medskip
\begin{itemize}
\item  if $[a,b] \subset \R\setminus S,$ then $V^{(a,b)} = 0.$
%
\end{itemize}

These data and properties imply that

\begin{itemize}
\item  $j_{(a_1,b_1),(a_2,b_2)} = 0$ whenever $(a_1,b_1) \leq (a_2,b_2)$ are disjoint.
\end{itemize}

Restricted two-parametric persistence modules form a particular case of two-dimensional\footnote{The term ``two-dimensional" would cause confusion in our setting.} persistence modules as defined in \cite{CarlssonMultidimensionalPersistence}.
\end{df}

\bs
\begin{rmk}\label{Remark: category of $r2p$ persistence modules}

Note that $r2p$ persistence modules with a given spectrum $S$ form a category $\mathcal{D}^S$ where a morphism $\Phi$ between any two $r2p$ persistence modules $V,W \in \mathcal{D}^S$ is a collection of maps $\Phi^{(a,b)}: V^{(a,b)} \to W^{(a,b)}$ for each interval $(a,b),\;a,b \notin S$ that commutes with the comparison maps, i.e. $\Phi$ is a natural transformation of the corresponding functors.

\end{rmk}

\bs
\begin{rmk}
For a number $\sigma \in \R,$ and a $r2p$ persistence module $(S,V)$ we can form its {\em shift} $(S+\sigma,V^{\bullet + \sigma})$ by $\sigma,$ which is a $r2p$ persistence module with spectrum $S+\sigma$ and defined on objects as $(a,b) \mapsto V^{(a,b)+\sigma},$ and on morphisms as $j^{V^{\bullet + \sigma}}_{(a_1,b_1),(a_2,b_2)}= j^V_{(a_1,b_1)+\sigma,(a_2,b_2)+\sigma}.$ 
\end{rmk}

\bs
\begin{example}
We give a sketch of two examples from Morse homology, and continue to discuss a similar situation for Floer homology.

\begin{enumerate}
\item

A basic example of such an $r2p$ persistence module is associated to a Morse function $f$ on a closed manifold $X.$ Indeed put $V^{(a,b)} = H(\{f<b\},\{f \leq a \}),$ a relative homology group of sublevel sets of $f.$ In fact this group is isomorphic to $HM^{(a,b)}(f),$ the homology of the Morse complex of $f,$ generated by critical points of $f$ with critical values in $(a,b),$ and $S=Spec(f) = f(Crit(f))$ the set of all critical values of $f.$ The comparison maps are given by natural morphisms between the corresponding complexes induced by inclusions, or alternatively by certain Morse continuation maps.


\item A more subtle example can be constructed from {\em any} smooth function $f$ on a closed manifold. Put $S=Spec(f)=f(Crit(f)),$ and let $(a,b), \;a,b \notin S$ be an interval. Consider the set $\cF(f)$ of Morse functions $f'$ sufficiently close to $f$ so that $a,b \notin Spec(f'),$ and all continuation maps $C(f',f''): HM^{(a,b)}(f') \to HM^{(a,b)}(f'')$ for $f',f'' \in \cF(f)$ are isomorphisms. Then define $HM^{(a,b)}(f)$ as \footnote{This definition is a specific representative of the isomorphism class of {\em limits} of the {\em indiscrete groupoid}, namely a category with exactly one morphism between any two objects, formed by $\{HM^{(a,b)}(f')\}_{f' \in \cF(f)}$ and the continuation maps, rendering each two of these vector spaces canonically isomorphic. Note that this representative of the limit of this diagram is canonically isomorphic by a unique isomorphism to a similar representative of the limit of any of its full subdiagrams, namely subdiagrams with the same morphism sets as in the diagram between any two of their objects (since all the continuation maps are isomorphisms). This observation is useful in showing that this definition satisfies the properties of a $r2p$ persistence module.} the vector space of collections $\{x_{f'} \in HM^{(a,b)}(f')\}_{f' \in \cF(f)}$ satisfying $C(f',f'')(x_{f'}) = x_{f''}.$ One shows that $(S,\{HM^{(a,b)}(f)\}_{(a,b) \in \cI^S})$ forms a $r2p$ persistence module.


For example the constant function $f \equiv 0$ on a closed manifold $X$ gives the $r2p$ persistence module with spectrum $S=\{0\}$ and for $a,b \notin S,$ $HM^{(a,b)}(f) = 0$ if $0 \notin (a,b)$ and $HM^{(a,b)}(f) \cong H(X),$ the homology of $X,$ if $0 \in (a,b),$ with the obvious comparison maps.
\end{enumerate}
\end{example}

Let $(M,\om)$ be a closed connected symplectic manifold, and let $\al \in \pi_0(\cL M)$ be a free homotopy class of loops. Choose a reference path $\eta_\al \in \cL_{\al} M.$ Assuming that $(M,\om)$ is symplectically $\al$-atoroidal,  and symplectically aspherical, we describe a construction of a Hamiltonian Floer homology for $(M,\om)$ in the class $\al.$

For a Hamiltonian $H \in \cH:= \sm{S^1 \times M, \R}_0,$ (where $0$ stands for mean-normalized) the time-dependent Hamiltonian vector field $X_{H}$ is defined as \[\iota_{X_H(t)}\om = - d H_t,\] where $H_t(-)=H(t,-).$ Denote by $\mathcal{P}^\al(H)$ the set of $1$-periodic orbits $x(t)$ in the class $\al$ of the Hamiltonian flow of $X_H,$ namely solutions to \begin{eqnarray}
\dot{x}(t) = X_H(t,x(t))\; \forall t \in S^1, \\
\,[ x(t) ] = \al \in \pi_0(\cL M).
\end{eqnarray}

Fix a base field $\K$ (we pick $\K = \Z/(2)$ for the proof of Theorems \ref{thm-main} and \ref{thm-generic}, and $\K = \Q_p,$ the splitting field of $x^p-1 \in \Q[x]$ over $\Q,$ for the proof of Theorem \ref{thm-main-squares}). Consider the vector space $CF(H)_\al$ over $\K$ freely generated by the set $\mathcal{P}^\al(H).$ Put $\{\phi_H^t\}_{t \in [0,1]}$ for the flow generated by $H.$ The elements of $\mathcal{P}^\al(H)$ correspond to fixed points $x$ of $\phi_H:=\phi_H^1,$ such that the loop $\{\phi_H^t \cdot x\}_{t \in S^1 = [0,1]/\{0,1\}}$ lies in the class $\al.$

For an $\al$-non-degenerate Hamiltonian $H,$ namely one for which the linearization $D\phi_H(x)$ of $\phi_H$ at every fixed point $x$ corresponding to an element in $\mathcal{P}^\al(H)$ has no eigenvectors with eigenvalue $1,$ the set $\mathcal{P}^\al(H)$ is finite.

We define the action functional $\A_H:\cL_\al M \to \R$ by choosing a path in $\cL_\al$ between $\eta_\al$ and a point $x \in \cL_\al M,$ considering it as a cylinder $\overline{x}:S^1 \times [0,1] \to M,$ and computing \[\A_H(x,\overline{x}) = \int_0^1 H(t,x(t)) dt - \int_{\overline{x}}\om.\]

Since our manifold is $\al$-atoroidal, the action functional $\A_H$ does not depend on the choice of "capping" $\overline{x},$ and is hence well defined as a functional $\A_H: \cL_\al M \to \R.$ Its critical points are exactly the periodic orbits $\cP^\al(H)$ of the Hamiltonian flow of $H$ in class $\al.$ Put $Spec_\al(H):=\A_H(\cP^\al(H))$ for the spectrum of $\A_H$ in the class $\al.$ By \cite{OhChainLevel} $Spec_\al(H) \subset \R$ is a measure-zero subset, and hence has empty interior.

Since $H$ is $\al$-non-degenerate, $\A_H$ has isolated critical points in $\cL_\al$ - indeed there is a bijective correspondence between a $1$-periodic orbit of a given flow and its initial point. As the manifold $M$ is compact, we conclude that $CF(H)_\al$ is a finite-dimensional $\K$-vector space. We grade it as follows by the Conley-Zehnder index \cite{RobbinSalamon}, with the normalization that for a $C^2$-small autonomous Morse Hamiltonian, the Conley-Zehnder index of each of its critical points as a contractible periodic orbit of the Hamiltonian flow equals to the Morse index of this critical point. We choose non-canonically a trivialization $\Phi_\al$ of the symplectic vector bundle $\eta_\al ^* TM$ over $S^1.$ Then any choice of a homotopy $\overline{x}$ from $\eta_\al$ to a $1$-periodic orbit $x$ of the flow $\{\phi_H^t\}$ in the class $\al,$ defines a homotopically canonical trivialization of $x^* TM.$ We then compute the index of the path of symplectic matrices obtained from $\{D\phi_t^H(x(0))\}_{t \in S^1}$ by the trivialization. Since our manifold is $\al$-atoroidal, this number does not depend on the choice of homotopy $\overline{x}.$

Choosing a generic $\om$-compatible almost complex structure \[J \in \cJ:=\sm{S^1,\cJ(M,\om)}\] depending on $t \in S^1,$ so that transversality is achieved (cf. \cite{FloerHoferSalamon}), we define the matrix coefficients of the Floer differential $\del_{J,H}: CF(H)_\al \to CF(H)_\al[-1]$ by counting the number of (the dimension zero component of isolated) solutions $u$ of the Floer equation \[\del_s u + J(t,u)(\del_t u - X_H(t,u))=0.\]

By standard arguments (see e.g. \cite{Floer, FukayaOno, HoferSalamon, LiuTian}) $\del_{J,H} \circ \del_{J,H} = 0,$ and hence $(CF(H)_\al, \del_{J,H})$ is a chain complex. Moreover, $\A_H$ defines a function on the generators of $CF(H)_\al,$ which extends to $CF(H)_\al$ as a valuation, that is $\A_H(0):=-\infty,$ and for a chain $0 \neq c = \Sum_j a_j x_j,$ where $x_j \in \cP^\al(H),$
\[\A_H(c) := \max \{\A_H(x_j)| a_j \neq 0\}.\] In particular $\A_H(c_1 + c_2) \leq \A_H(c_1) + \A_H(c_2),$ for all $c_1,c_2 \in CF(H)_\al.$ By a standard action-energy estimate $\A_H(\del(c)) < \A_H(c).$ Hence $CF^{(-\infty,a)}(H)_\al := \{c \in CF(H)_\al|\;\A_H(c)<a\}$ is a subcomplex of $(CF(H)_\al,\del_{J,H}),$ and as $a$ runs through $\R \setminus Spec_\al(H),$ defines a filtration on $CF^{(-\infty,a)}(H)_\al.$ We obtain a filtered complex which we denote by $(CF(H)_\al,\del_{J,H},\A_H).$

For $a \notin Spec_\al(H),$ put $CF^{(-\infty,a)}(H)_\al := \{c \in CF(H)_\al|\;\A_H(c) < a\}.$ This is a subcomplex of $(CF(H)_\al,\del_{J,H}).$ For a window $(a,b)$ with $a,b \notin Spec_\al(H),$ define $CF^{(a,b)}(H)_\al$ as the quotient complex \[ CF^{(a,b)}(H)_\al := CF^{(-\infty,b)}(H)_\al/ CF^{(-\infty,a)}(H)_\al,\] with the induced differential, which we denote $\del_{J,H},$ by a slight abuse of notation. Put \[HF^{(a,b)}(H,J)_\al:= H(CF^{(a,b)}(H)_\al, \del_{J,H}),\] for the homology of this quotient complex. We have the following invariance statement.

\medskip
\begin{prop}\label{Proposition: $r2p$ persistence module of a Hamiltonian from Floer theory}
The assignment $(a,b) \mapsto HF^{(a,b)}(H,J)_\al$ defines a $r2p$ persistence module with spectrum $S=Spec_\al(H).$ Moreover, there is a canonical isomorphism between this $r2p$ persistence module and each other one obtained from a different choice of $J$ and $H,$ as long as the path $\{\phi^t_H\}_{t \in [0,1]}$ remains in a fixed class in the universal cover  $\til{\G}$ of the group of Hamiltonian diffeomorphisms.
\end{prop}
This leads us to the following definition.
\medskip

\begin{df}\label{Definition: $r2p$ persistence module of a Hamiltonian diffeomorphism}
Given a non-degenerate element $\til{\phi}$ in $\til{\G},$ we define its $r2p$ persistence module \[(a,b) \mapsto HF^{(a,b)}(\til{\phi})\] as\footnote{This definition is a canonical representative of the isomorphism class in $\cD^S$ of limits in $\cD^S$ of the $\cD^S$-valued diagram defined by Proposition \ref{Proposition: $r2p$ persistence module of a Hamiltonian from Floer theory}.} the $r2p$ persistence module with spectrum $S=Spec_\al(H)$ which to an open interval $(a,b)$ associates the vector space of collections \[\{x_{(a,b),H,J} \in HF^{(a,b)}(H,J)_\al\}_{(H,J)}\] indexed by all pairs of $\al$-regular pairs $(H,J)$ such that the path $\{\phi^t_H\}_{t \in [0,1]}$ lies in class $\til{\phi},$ satisfying the condition \[\Psi^{(a,b)}_{(H,J),(H',J')}(x_{(a,b),H,J}) = x_{(a,b),H',J'},\] for the canonical isomorphism \[\Psi^{(a,b)}_{(H,J),(H',J')}: HF^{(a,b)}(H,J)_\al \to HF^{(a,b)}(H',J')_\al.\]
\end{df}

Proposition \ref{Proposition: $r2p$ persistence module of a Hamiltonian from Floer theory} follows for example from \cite[Proposition 5.2]{UsherBoundaryHofer}. Roughly speaking, and this can be made rigorous, the isomorphism between the $r2p$ persistence modules for different choices of $J$ is constructed by counting solutions to a Floer continuation map, and the isomorphism between the $r2p$ persistence modules for two different choices $\{\phi^t_H\}_{t \in [0,1]}, \{\phi^t_{H'}\}_{t \in [0,1]}$ of the Hamiltonian path in a fixed class in the universal cover is obtained by a diffeomorphism of $\cL_\al$ given by the action of the contractible loop $\{\gamma_t=\phi^t_{H'}(\phi^t_{H})^{-1}\}_{t \in [0,1]}$ in $\G,$ that is the difference between these paths.\footnote{Here and below we deal with certain transformations of loop spaces which naturally act on action functionals and on the Riemannian metrics on $\cL_\al M$  coming from loops of almost complex structures on $M$, thus inducing morphisms in Floer homology which are useful for our purposes. We call them diffeomorphisms since this
way of thinking provides a right intuition for guessing and manipulating these Floer homological constructions.
Incidentally, these transformations are genuine diffeomorphisms if understood in the sense of diffeology \cite{Iglesias}.} This is the diffeomorphism \[\rP(\gamma): \cL_\al \to \cL_\al,\]  \[z(t) \mapsto \gamma_t(z(t)).\]

Moreover (cf. \cite[Section 13.1]{P-book}) we have the diagram

\[(\cL_\al M, \A_H) \xrightarrow{\rP(\ga)} (\cL_\al M, \A_{H'})\]

of spaces with functionals. That is \begin{equation}\label{Equation: loop preserves action} \A_H = \rP(\gamma)^*\A_{H'}.\end{equation}

\bs

In fact it is useful to have a definition of Floer homology for degenerate elements $\til{\phi} \in \til{\G}$ as well, and many of the arguments that follow are based, at least intuitively, on the following extended definition.

\bs
\begin{df}\label{Definition: Floer homology for degenerate paths}
Given any $\til{\phi} \in \til{\G},$ represent it by a path $\{\phi^t_H\}$ with Hamiltonian $H \in \sm{S^1 \times M, \R}.$ Let $(a,b), \; a<b, \; a,b \notin Spec_\al(H)$ be a fixed window. Then for any non-degenerate $C^2$-perturbation $H'$ of $H$ that is sufficiently small\footnote{We say that a perturbation $H'$ of $H$ is $C^2$-small if $H'-H$ is a $C^2$-small function}, $a, b \notin Spec_\al(H')$ still. Moreover, decreasing if necessary the threshold of smallness, interpolation continuation establishes functorial isomorphisms between each pair of Floer homology groups\footnote{In other words, these vector spaces and isomorphism maps form an {\em indiscrete groupoid} in $Vect_\K$} in $\{HF^{(a,b)}(H',J')_\al\}$ where $H'$ runs over the set $\cH^{reg}(H)$ of such $C^2$-small perturbations, and $J' \in \cJ^{reg}(H')$ is an almost complex structure such that the pair $(H',J')$ is regular. We define $HF^{(a,b)}(H)_\al$ as the vector space of collections\footnote{That is - we are considering a specific representative of the limit of the corresponding indiscrete groupoid.} \[\{x_{{(a,b)},{H',J'}} \in HF^{(a,b)}(H',J')_\al\}_{\cH^{reg}(H)}\] such that  \[\Psi^{(a,b)}_{(H'_1,J'_1),(H'_2,J'_2)}(x_{{(a,b)},{H'_1,J'_1}}) = x_{{(a,b)},{H'_2,J'_2}}\] for the canonical isomorphism \[\Psi^{(a,b)}_{(H'_1,J'_1),(H'_2,J'_2)}: HF^{(a,b)}(H'_1,J'_1)_\al \to HF^{(a,b)}(H'_2,J'_2)_\al.\] We note that the same construction with the indexing set any non-empty subset of the above indexing set yields a canonically isomorphic vector space, and that this implies  that $HF^{(a,b)}(H)_\al$ defines a $r2p$ persistence module and that as in Proposition \ref{Proposition: $r2p$ persistence module of a Hamiltonian from Floer theory} this $r2p$ persistence module does not depend on the representative $\{\phi^t_H\}$ of $\til{\phi}.$ This leads as in Definition \ref{Definition: $r2p$ persistence module of a Hamiltonian diffeomorphism} above to a definition of $HF^{(a,b)}(\til{\phi})_\al$ for any class $\til{\phi} \in \til{\G}.$
\end{df}

\begin{rmk}
It is straightforward to check that Definition \ref{Definition: Floer homology for degenerate paths} agrees with definition \ref{Definition: $r2p$ persistence module of a Hamiltonian diffeomorphism}. Namely, that the resulting $r2p$ persistence modules are canonically isomorphic.
\end{rmk}

\begin{rmk}\label{Remark: FH descends to Ham}
In fact, Proposition \ref{Proposition: $r2p$ persistence module of a Hamiltonian from Floer theory}, and hence Definition \ref{Definition: $r2p$ persistence module of a Hamiltonian diffeomorphism}, and consequently Definition \ref{Definition: Floer homology for degenerate paths} can be upgraded in our setting of a closed symplectically aspherical, $\al$-atoroidal symplectic manifold to define a $r2p$ persistence module \[(a,b) \mapsto HF^{(a,b)}(\phi)\] of a Hamiltonian diffeomorphism itself. Indeed, given two paths $\{ \phi_H^t \}$ and $\{\phi_{H'}^t\}$ with endpoint $\phi$ we consider the action $\rP(\gamma)$ of the difference loop $\{\gamma_t=\phi^t_{H'}(\phi^t_{H})^{-1}\}_{t \in [0,1]}$ on the loop space $\cL_\al M.$

It is easy to see that for $x \in \cP^\al(H),$ its image $x'(t)=\rP(\gamma)(x)(t) = \gamma_t(x(t)) \in \cP^\al(H')$ satisfies \[\A_{H'}(x') = \A_H(x) + A(\gamma,\eta_\al),\] where $A(\gamma,\eta_\al) = \A_G(\rP(\gamma)(\eta_\al)),$ where the action is normalized by the reference loop $\eta_\al,$ and $G \in \cH$  is the Hamiltonian generating $\gamma.$ Namely, for $\eta \in \cL_\al M$ we put \[A(\gamma,\eta) = \int_0^1 G(t,\gamma_t(\eta(t))) dt - \int_{\overline{\rP(\gamma)}(\eta)} \om,\] where $\overline{\rP(\gamma)}(\eta)$ is a cylinder in $M$ defined by any path in $\cL_\al M$ between $\eta_\al$ and $\rP(\gamma)(\eta).$ A little differential homotopy computation shows that $A(\gamma,\eta)$ depends only on the class $[\gamma]$ of $\gamma$ in $\pi_1(\Ham(M))$ and the free homotopy class $\al$ of $\eta.$ In particular $A(\gamma,\eta_\al)$ is invariant under reparametrizations of $\gamma$ and $\eta_\al.$ Reparametrizing $\gamma$ and $\eta_\al$ to be non-constant in disjoint time subintervals, we see that $A(\gamma, \eta_\al) = A(\gamma,pt_M)$ and $A(\gamma,pt_M) = 0$ by \cite[Theorem 1.1, Corollary 4.15]{SchwarzAspherical}.

Similarly, it is easy to see that the Conley-Zehnder index of $x \in \cP^\al(H)$ and of its image $x'(t)=\rP(\gamma)(x)(t) \in \cP^\al(H')$ satisfy \[\mathrm{CZ}_{H'}(x') = \mathrm{CZ}_H(x) + I(\gamma,\eta_\al),\] where for $\eta \in \cL_\al$ we define $I(\gamma,\eta)$ as the Maslov index of the following loop of symplectic matrices. Make a non-canonical choice of a cylinder $\overline{\eta}$ from $\eta_\al$ to $\eta.$ This cylinder and the trivialization $\Phi_\al$ of $\eta_\al^* TM$ define homotopically-canonically a trivalization $\Phi=\overline{\eta}_* \Phi_\al$ of $\eta^* TM,$ that is an isomorphism $\Phi: \eta^* TM \to S^1 \times \R^{2n}$ of symplectic vector bundles. Similarly, choosing a cylinder $w$ from $\eta$ to $P(\gamma)\eta$ defines a trivialization $w_* \Phi$ of $P(\gamma)\eta^* TM.$ Since the differential $D\gamma \circ \eta: \eta^* TM \to P(\gamma)\eta^* TM$ is an isomorphism of symplectic vector bundles, we obtain another trivialization $P(\gamma)^* w_* \Phi$ of $\eta^* TM.$ The loop of symplectic matrices we consider is the difference loop $P(\gamma)^* w_* \Phi \circ \Phi^{-1}$ of the two trivializations. Since our manifold is $\al$-atoroidal, the Maslov index of this loop does not depend on the choices of cylinders made. Moreover, $I(\gamma,\eta)$ depends only on $[\gamma] \in \pi_1(\Ham(M))$ and the free homotopy class $\al$ of $\eta.$ Thus, as above we conclude that $I(\gamma,\eta) = I(\gamma,pt_M)$ and $I(\gamma,pt_M)=0$ by \cite[Proposition 10.1]{SeidelInvertibles}.

Hence $\rP(\gamma)$ enters the diagram $(\cL_\al M, \A_H) \xrightarrow{\rP(\ga)} (\cL_\al M, \A_{H'})$ of spaces with functionals, and hence determines an isomorphism of filtered complexes that preserves grading, and hence of the corresponding graded $r2p$ persistence modules. Compare \cite[Proposition 5.3]{UsherBoundaryHofer}.
\end{rmk}


A general observation is that given a Hamiltonian diffeomorphism $g$, and a regular class $\til{\phi} \in \til{\G},$ there is a natural morphism of $r2p$ persistence modules \[[\rP(g)]: HF^{(a,b)}(\til{\phi})_\al \to HF^{(a,b)}(g \circ \til{\phi} \circ g^{-1})_\al,\] which we call the {\em push-forward map.} This morphism is built by acting by $g$ on all the objects involved in the construction. Such morphisms in the context of fixed-point Floer homology of symplectormorphisms were recently introduced and used by D. Tonkonog \cite{Tonkonog}. The basic such action is the diffeomorphism \[\rP(g): \cL_\al M \to \cL_\al M,\] \[z(t) \mapsto g(z(t)).\]

Given a Hamiltonian $H \in \cH$ that generates a representative of $\til{\phi},$ the Hamiltonian $H\circ g^{-1}$ generates a representative of $g \circ \til{\phi} \circ g^{-1}$ and the restriction $\rP(g): \cP^\al(H) \to \cP^\al(H \circ g^{-1}), x(t) \mapsto g(x(t))$ is a bijection. Moreover, since $g \in \Ham(M),$ the symplectic area of the cylinder between the reference loop $\eta_\al \in \cL_\al$ and the loop $g\circ \eta_\al \in \cL_\al,$ (which is well-defined since $M$ is $\al$-atoroidal,) is in fact zero.

Hence we have the diagram \[(\cL_\al M, \A_H) \xrightarrow{\rP(g)} (\cL_\al M, \A_{H\circ g^{-1}})\] of spaces with functionals, that is \begin{equation}\label{Equation: push preserves action}\rP(g)^*\A_{H\circ g^{-1}} = \A_{H},\end{equation} the actions being computed with respect to the same reference loop $\eta_\al.$

What remains is to observe that the restricted map on generators, given a choice of almost complex structure $J \in \cJ,$ extends naturally to an isomorphism of filtered Floer complexes \[\rP(g): (CF(H,J)_\al,\A_H) \to ((CF(H \circ g^{-1}, g_* J)_\al),\A_{H\circ g^{-1}})\] where $(g_* J)_x = Dg(g^{-1}x) J_{g^{-1}x} D(g^{-1})(x),$ is the push-forward of the almost complex structure $J$ by the diffeomorphism $g.$

\begin{df}\label{Definition: general push map}
The isomorphism $\rP(g)$ of filtered Floer complexes gives the map
\[[\rP(g)]: HF^{(a,b)}(\til{\phi})_\al \to HF^{(a,b)}(g \circ \til{\phi} \circ g^{-1})_\al\]
of $r2p$ persistence modules, which we call the push-forward map.
\end{df}

The push-forward map is an example of an operator on Floer homology coming from actions of Hamiltonian loops on the loop space of $M.$ Consider a contractible loop $\gamma$ based at $g \in \Ham(M).$ It acts by \[\rP(\gamma):\cL_\al M \to \cL_\al M,\] \[z(t) \mapsto \gamma_t(z(t))\] on the loop space of $M$ in component $\al.$ Given a Hamiltoninan $H \in \cH$ there exists a natural Hamiltonian $H' = \rP(\gamma)_*H$ such that \[\rP(\gamma)^*\A_{H'} = \A_H.\] It is given by \[H'(t,x) = H_t \circ \gamma_t^{-1} + G_t,\] where $G_t \in \cH$ is the Hamiltonian generating $\gamma.$ Note that if $H$ generates the Hamiltonian isotopy $\{\phi_t\}$ then $H'$ generates the Hamiltonian isotopy $\{\gamma_t\phi_t g^{-1}\}.$ It is therefore clear that $H'$ is $\al$-non-degenerate if and only if $H$ is, and that $\rP(\gamma)$ establishes an action-preserving bijection $\cP^\al(H) \to \cP^\al(H'),$ which moreover extends to an isomorphism of filtered Floer complexes \[\rP(\gamma): (CF(H,J)_\al,\A_H) \to ((CF(H', \rP(\gamma)_* J)_\al),\A_{H'}),\] where $(\rP(\gamma)_* J)_{t,x} = D\gamma_t(\gamma_t^{-1}x) J_{t,\gamma_t^{-1}x} D(\gamma_t^{-1})(x).$


\bs
\begin{rmk}
If $\gamma$ is not contractible then \[\rP(\gamma)^*\A_{H'} = \A_H - A(\gamma\cdot g^{-1},\eta),\]
where $A$ is the value discussed in Remark \ref{Remark: FH descends to Ham} and was shown to vanish in our specific setting. Hence in our setting \[\rP(\gamma)^*\A_{H'} = \A_H,\] and $\rP(\gamma)$ gives an automorphism of Floer homology in action windows precisely as discussed for the case of contractible loops in $\Ham(M).$
\end{rmk}

We need the following simple observation on the map $\rP(\gamma)$ and continuation maps of almost complex structures.

\bs
\begin{lma}\label{Lemma: Push gamma general and continuations}
Given regular Floer continuation data $\overrightarrow{J}=\{(H,J_s)\}_{s \in \R}$ with $J_s \equiv J_0$ for $s \ll 0$ and $J_s \equiv J_1$ for $s \gg 0,$ there is a commutative diagram of filtered Floer complexes \begin{equation}\label{Diagram: Push gamma general and continuations}
  \displaystyle    \begin{array}{lll}
        CF(H,J_0)_\al  & \xrightarrow{\rP(\gamma)} & CF(H', \rP(\gamma)_* J_0)_\al\\
         \scriptstyle{C(\overrightarrow{J})}{\downarrow} &   &  \scriptstyle{C(\rP(\gamma)_*\overrightarrow{J})}{\downarrow}\\
       CF(H,J_1)_\al   & \xrightarrow{\rP(\gamma)} & CF(H', \rP(\gamma)_* J_1)_\al \\
      \end{array}
   \end{equation} where $P(\gamma)_* \overrightarrow{J}$ is the Floer continuation data $\{(H,P(\gamma)_*(J_s))\}_{s \in \R},$ and $C(\overrightarrow{J}),$ $C(\rP(\gamma)_*\overrightarrow{J})$ are Floer continuation maps.

\end{lma}

This lemma is immediate once we change variables in the Floer continuation equation using the diffeomorphism $\rP(\gamma),$ namely there is a bijection between solutions $\{u(s,t)\}$ of the continuation equation for the operator $C(\overrightarrow{J})$ and solutions $\{v(s,t)\}$ of the continuation equation for $C(\rP(\gamma)_*\overrightarrow{J})$ given by $u(s,t) \mapsto v(s,t)=\gamma_t(u(s,t)).$

\section{Loop-rotation operators}\label{Section: loop rotation}
A time-periodic Hamiltonian $H \in \cH:= \sm{S^1 \times M, \R}$ is called {\it mean-normalized}
if $\int_M H_t \;\omega^n =0$ for all $t$, where $\dim M = 2n$. We write $\cH$ for the space of all
mean-normalized Hamiltonians in $\sm{S^1 \times M, \R}$. Consider $H \in \cH$  generating the Hamiltonian diffeomorphism $\phi_H$. Take the new Hamiltonian function $\Hk(t,x):= k H(tk,x).$
It generates $\phi_{\Hk} = \phi_{H}^k.$

We note that $\phi_{H} \phi_{\Hk} \phi_H^{-1} = \phi_{\Hk},$ hence  $\phi_H$ acts on the Floer homology of $\phi_{\Hk}.$ Assuming that $H$ is such that $\Hk$ is $\al$-non-degenerate, and denoting $\til{\phi}_{H}^k$ the class in $\til{\G}$ that $\Hk$ generates we have the morphism
\begin{equation}\label{Definition:Push k}
[\rP_k]:=[\rP(\phi_H)]: HF^{(a,b)}(\til{\phi}_{H}^k)_\al \to HF^{(a,b)}(\til{\phi}_{H}^k)_\al
\end{equation}
of filtered Floer homology  understood in the sense of the limit (see Definition \ref{Definition: $r2p$ persistence module of a Hamiltonian diffeomorphism} above).

On the other hand, since $\Hk(t,x) \equiv \Hk(t+\frac{1}{k},x),$ it is easy to see that the loop-rotation diffeomorphism \[\rR_k:\cL_\al M \to \cL_\al M,\;\;\; z(t) \mapsto z(t+\frac{1}{k}),\] satisfies \begin{equation}\label{Equation: rotation preserves Hk action}
(\rR_k)^*\A_{\Hk} = \A_{\Hk}.
\end{equation}
Hence $\rR_k$ restricts to an action-preserving bijection $\cP^\al(\Hk) \to \cP^\al(\Hk)$ and therefore for a generic almost complex structure $J \in \cJ,$ defines an isomorphism of filtered Floer complexes
\begin{equation}\label{eq-upsilon}
(CF(\Hk,J)_\al,\A_{\Hk}) \to (CF(\Hk, (\rR_k)_* J)_\al,\A_{\Hk})\;,
\end{equation}
where $((\rR_k)_* J)_t= J_{t+\frac{1}{k}}$.
Consider the induced morphism in filtered Floer homology, which is again understood in the sense
of the limit:

\bs
\begin{equation}\label{Definition: Rotation k}
[\rR_k]: HF^{(a,b)}(\til{\phi}_{H}^k)_\al \to HF^{(a,b)}(\til{\phi}_{H}^k)_\al\;.
\end{equation}
Several of the arguments that follow benefit from knowing that in fact the maps $[\rP_k]$ and $[\rR_k]$ on Floer homology in fact coincide.

\bs
\begin{lma}\label{Lemma: two descriptions}
The maps $[\rR_k]$ and $[\rP_k]$ coincide, and hence are simply different descriptions of the same map \[\rT_k : HF^{(a,b)}(\til{\phi}_{H}^k)_\al \to HF^{(a,b)}(\til{\phi}_{H}^k)_\al\] of $r2p$ persistence modules. \end{lma}

\begin{rmk}\label{Remark: representation}
It is evident from the definition (cf. Lemma \ref{Lemma: Push gamma general and continuations}) that $(\rT_k)^k = \id,$ and hence $\rT_k$ defines a $\Z_k = \Z/(k)$-representation on $HF^{(a,b)}(\til{\phi}_{H}^k)_\al.$
\end{rmk}

\medskip
\noindent
{\sc Morse theoretical digression:}  The proof of Lemma \ref{Lemma: two descriptions} rests on the following picture in Morse theory. Consider a Morse function $f$ on a closed manifold $X$ and an isotopy $\{\psi_r\}_{r \in [0,1]}$ such that $\psi_r^* f = f$ for all $r \in [0,1].$ Then $\psi = \psi_1$ acts on the Morse homology $HM^{(a,b)}(f)$ in any window $(a,b),$ and moreover it acts by $\id.$ This can be seen immediately by considering the (relative) singular homology of sublevel sets $(\{f < b\},\{f < a\})$ of $f,$ and constructing a chain homotopy of the map induced by $\psi$ to $\id$ by considering the cylinders of the singular cycles traced by the isotopy $\{\psi_r\}.$

Let us sketch the Morse homological argument proving the above-mentioned statement. It will be important in the sequel as it readily extends to the Floer theoretical context.  Fix a generic Riemannian metric $\rho$ on $X,$ and denote $\rho_r:=(\psi_r)_*\rho.$  Observe that the map $\psi_r$ canonically identifies the filtered
Morse complexes $(CM_*(f, \rho_r),d_r)$ for all $r$ with $(W_*,\del):= (CM_*(f,\rho), d)$.

Consider the family of metrics $\rho_{r,s}$, $r \in [0,1]$, $s \in \R$ such that for $s \ll -1$ $\rho_{r,s} = \rho_r = (\psi_r)_*(\rho)$ and for $s \gg 1$ $\rho_{r,s} = \rho$. Look at the gradient flow equation
\begin{equation}
\label{eq-vsp-gradflow}
\frac{du}{ds}= -\nabla_{\rho_{r,s}}f(u(s))\;,
\end{equation}
where both $r$ and $u$ is considered as a variable. Look at the isolated solutions $(r,u)$ of this equation. In light of the identification above they define a map $\cS: W_* \to W_{*+1}$ which does not increase the filtration induced by $f$.

With this identification, the action of the diffeomorphism $\psi$ on the homology of $W$ is given
(on the chain level) by the continuation map $C$ induced by the path of metrics $\rho_{1,s}$.
We claim that $C$ is chain homotopic to the identity, i.e.,
\begin{equation}\label{eq-vsp-breaking}
C-\id = \del\cS-\cS\del\;.
\end{equation}

To see this, let us analyze the space of solutions of \eqref{eq-vsp-gradflow} connecting critical points
with equal Morse indices. Given regularity, it can be compactified to a manifold with boundary of dimension $1.$ Boundary contributions appear when either $r=0$ or $r=1,$ or when there is breaking of trajectories. For $r=0$ they correspond to $\id$, for $r=1,$ the solutions satisfy the continuation equation for $C$, while the breaking of trajectories gives us $\del \cS - \cS \del$. This proves \eqref{eq-vsp-breaking} and completes our digression.

\medskip
\noindent
\begin{proof} (of Lemma \ref{Lemma: two descriptions}) $\;$

\noindent
{\sc Step 1:} Take $r \in [0,1]$. Consider two Hamiltonians depending on the parameter $r$,
$E_r= \Hk \circ (\phi^r_H)^{-1}$ and $F_r= \Hk_{t+r/k}$.
The former generates the Hamiltonian path
$$\alpha_r(t):= \phi^r_H \phi^t_{\Hk}(\phi^r_H)^{-1}= \phi^r_H \phi^{kt}_H (\phi^r_H)^{-1}\;.$$
The latter Hamiltonian generates the path
$$\beta_r(t) := \phi_{\Hk}^{t+r/k}(\phi_{\Hk}^{r/k})^{-1} = \phi_{H}^{tk+r}(\phi^r_H)^{-1}\;.$$
Observe that both paths have the same endpoints,
$$\alpha_r(0) =\beta_r(0) = \id,\;\ \alpha_r(1)=\beta_r(1) = \phi^r_H \phi^{k}_H (\phi^r_H)^{-1}\;,$$
and, as one readily checks, they are homotopic with fixed endpoints. In other words,
$\beta_r= \gamma_r\alpha_r$, where $\gamma_r$ is a contractible loop.
It follows that $\rP(\gamma_r)^*\A_{F_r}= \A_{E_r}$. Furthermore, since the conjugation by $\phi_H^r$ takes the path $\phi^t_{\Hk}$ to $\phi^t_{E_r}$ we have $\rP(\phi_H^{r})^*\A_{E_r}= \A_{\Hk}$.
We conclude that
\begin{equation}\label{eq-1-vsp-lemma}
\rP(\phi_H^{r})^*\rP(\gamma_r)^*\A_{F_r}= \A_{\Hk}\;.
\end{equation}

\medskip
\noindent
{\sc Step 2:} Next, for $r \in [0,1],$ let $\rR_{r,k}$ be the diffeomorphism \[\rR_{r,k}:\cL_\al M \to \cL_\al M,\;\;z(t) \mapsto z(t+r/k).\] It satisfies
\begin{equation}\label{Equation: Rotation time r action preserving}
\rR_{r,k}^* \A_{F_r} = \A_{\Hk}.
\end{equation}
Put $\rQ_r:= \rR_{r,k}^{-1}\rP(\gamma_r)\rP(\phi_H^{r})$. Combining \eqref{Equation: Rotation time r action preserving} with \eqref{eq-1-vsp-lemma} we get that
\begin{equation}\label{eq-2-vsp-lemma}
\rQ_r^*\A_{\Hk}= \A_{\Hk}\;.
\end{equation}
It follows that $\rQ_r$ is a path of diffeomorphisms of loop spaces preserving $\A_{H_k}$.
Thus, for every $r$,  $\rQ_r$ induces the same morphism of filtered Floer homologies of $\Hk$.
The proof of this literally imitates the Morse theoretical argument presented above.
The action functional  $\A_{H_k}$ on the loop space $\cL_\al M$ plays the role of the Morse function $f$ on $X$, the diffeomorphisms $Q_r$ correspond to $\psi_r$ and loops of almost complex structures on
$M$ (which remained behind the scenes in our exposition)  determine Riemannian metrics on $\cL_\al M$.

Therefore, diffeomorphisms $\rQ_r$ induce the same
morphism $V \to V$, where we abbreviate  $V:= HF^{(a,b)}(\til{\phi}_{H}^k)_\al$  and understand
this space in the sense of the limit according to Definition \ref{Definition: $r2p$ persistence module of a Hamiltonian diffeomorphism} above. Since $[\rQ_0] =\id$, we have that $$[\rQ_1]=  [\rR_{1,k}^{-1}]\circ [\rP(\gamma_1)] \circ [\rP(\phi_H))]=\id\;.$$ It remains to notice that each of the factors $[\rR_{1,k}^{-1}]$, $[\rP(\gamma_1)]$
and $[\rP(\phi_H))]$ in the previous equation is an automorphism of $V$, and moreover
$\;[\rR_{1,k}] = [\rR_k]$, $[\rP(\gamma_1)] = \id$ since $\gamma_1$ is a contractible
loop, and $[\rP(\phi_H))]=[\rP_k]$. Therefore, $[\rR_k]=[\rP_k]$ as required.
\end{proof}

\bs
The fact that $\rT_k$ is a morphism of $r2p$ persistence modules in particular implies the following. Consider the natural comparison map \[j_d:=j_{(a,b),(a+d,b+d)}: HF^{(a,b)}(\Hk)_\al \to HF^{(a+d,b+d)}(\Hk)_\al\] between Floer homology groups in two windows $(a,b)$ and $(a+d,b+d)$ for $d \geq 0.$ Then the following diagram commutes.

\begin{equation}\label{Diagram: rotation commutes with comparison}
  \displaystyle    \begin{array}{clc}
         HF^{(a,b)}(\Hk)_\al & \xrightarrow{j_d} & HF^{(a+d,b+d)}(\Hk)_\al\\
         \scriptstyle{\rT_k}{\downarrow} &   &  \scriptstyle{\rT_k}{\downarrow}\\
         HF^{(a,b)}(\Hk)_\al & \xrightarrow{j_d} & HF^{(a+d,b+d)}(\Hk)_\al\\
      \end{array}
   \end{equation}

Similarly, the map $\rT_k$ commutes with Floer continuation maps $HF(\Fk)_\al \to HF(\Gk)_\al$ induced from continuation maps $HF(F)_\al$ to $HF(G)_\al$ as follows: a family $\{H_r\}_{r \in \R}$ of Hamiltonians with $H_r \equiv F$ for $r\ll0,$ and $H_r \equiv G$ for $r\gg 0$ induces the family $\{\Hk_r\}$ between $\Fk, \Gk.$ We note that an interpolation homotopy between $F$ and $G$ induces in this way an interpolation homotopy between $\Fk$ and $\Gk.$

\bs
\begin{lma}\label{Lemma: continuation representation maps}
Consider an interpolation continuation map \[C(\Fk,\Gk): HF^{(a,b)}(\Fk)_\al \to HF^{(a,b)+k\Max(G-F)}(\Gk)_\al.\] Then the following diagram commutes.

\begin{equation}\label{Diagram: rotation commutes with continuation}
  \displaystyle    \begin{array}{clc}
         HF^{(a,b)}(\Fk)_\al & \xrightarrow{C(\Fk,\Gk)} & HF^{(a,b)+k\Max(G-F)}(\Gk)_\al\\
         \scriptstyle{\rT_k}{\downarrow} &   &  \scriptstyle{\rT_k}{\downarrow}\\
         HF^{(a,b)}(\Fk)_\al & \xrightarrow{C(\Fk,\Gk)} & HF^{(a,b)+k\Max(G-F)}(\Gk)_\al\\
      \end{array}
   \end{equation}
\end{lma}

\begin{proof} Consider an interpolation \[C_{s}(t,x)=C(s,t,x)=(1-\rho(s)) \Fk(x,t) + \rho(s)\Gk(x,t)\] between $\Fk,$ and $\Gk.$ This interpolation, upon the choice of a generic $(s,t)$-dependent compatible almost complex structure with $J(s,t) \equiv J_-(t),$ for $s \ll 0$ and $J(s,t) \equiv J_+(t)$ for $s \gg 0$ regular for $\Fk$ and $\Gk$ respectively, by standard action estimates, gives rise to the Floer continuation map \[ HF^{(a,b)}(\Fk)_\al  \xrightarrow{C(\Fk,\Gk)}  HF^{(a,b)+k\Max(G-F)}(\Gk)_\al\] the matrix coefficient between $x \in \cP^\al(\Fk)$ and $y \in \cP^\al(\Gk)$ with $i(x)=i(y)$ of whose chain level operator is given by counting solutions to the equation \[ \del_s u + J(s,t,u) (\del_t u - X_{C_{s}}(t,u)) = 0, \] with asymptotic boundary conditions  $u(s,t) \xrightarrow{s \to -\infty} x,$ $u(s,t) \xrightarrow{s \to \infty} y.$ Note that the operator on the level of homology does not depend on the choice of $J(s,t).$ Similarly to Lemma \ref{Lemma: Push gamma general and continuations}, as $C(s,t+\frac{1}{k}) \equiv C(s,t),$ the transformation $u(s,t) \mapsto v(s,t):= \rR_k u(s,t) = u(s,t+\frac{1}{k})$ establishes a one-to-one correspondence between the solutions $u(s,t)$ of the above continuation equation and the solutions $v(s,t)$ of the equation \[ \del_s v + J(s,t+\frac{1}{k}s,v) (\del_t v - X_{C_{s}}(t,v)) = 0\] with asymptotic boundary conditions $v(s,t) \xrightarrow{s \to -\infty} \rR_k(x),$ $v(s,t) \xrightarrow{s \to \infty} \rR_k(y).$ Since the continuation map on the homology level does not depend on the choice of almost complex structures, the lemma now follows immediately.

\end{proof}

%
%

\section{Invariants}\label{Section: invariant}

\subsection{The $\Z_k$ spectral spread}\label{Subsection: spread}
In the spirit of persistent homology (see \cite{Weinberger, Ghrist,CarlssonTopologyData} for surveys),
a rapidly developing area lying on the borderline of algebraic topology and topological data analysis, we shall use windows where $\rT_k$ does not act trivially to separate $\phi_H$ from autonomous Hamiltonian diffeomorphisms.

For $k \in \Z, \; k\geq 1,$ and a Hamiltonian $H \in \cH$ we make the following definition. Put \[\rS_k:=\rT_k -\id.\]

Composing $\rS_k$ with the comparison map $j_d: HF^{(a,b)}(\Hk)_\al \to HF^{(a+d,b+d)}(\Hk)_\al,$ we obtain a map \[j_d \circ \rS_k = \rS_k \circ j_d: HF^{(a,b)}(\Hk)_\al \to HF^{(a+d,b+d)}(\Hk)_\al.\]

This brings us to the definition of the main invariant of this paper.

\medskip
\begin{df}\label{Definition: the invariant}(The invariant)
\[w_{k,\al}(H):= \sup \{d \geq 0: \; j_d \circ \rS_k \neq 0 \; \text{for some window} \; (a,b)\}\]
\end{df}

The following proposition shows that this indeed defines an invariant and lists its properties relevant to subsequent discussion.
\medskip
\begin{prop}\label{Proposition: properties of the invariant}
\begin{enumerate}[label=(\roman{*}), ref=(\roman{*})] The assignment $H \mapsto w_{k,\al}(H)$ satisfies the following properties:
\item \label{itm:property 1} The invariant $w_{k,\al}(H) \geq 0$ is a finite number. In fact \[w_{k,\al}(H) \leq k \cdot (\Max(H) - \Min(H)).\]
\item \label{itm:property 2} $w_{k,\al}(H)$ depends only on the diffeomorphism $\phi_H,$ hence we write $w_{k,\al}(\phi_H).$
\item \label{itm:property 3} $w_{k,\al}: \G \to [0,\infty)$ is Lipschitz in Hofer's metric on $\G.$ In terms of Hamiltonians \[|w_{k,\al}(F) - w_{k,\al}(G)| \leq k \cdot (\Max(F-G) - \Min(F-G)).\] In particular $w_{k,\al}$ extends to arbitrary Hamiltonian diffeomorphisms.
\item \label{itm:property 4} $w_{k,\al}(\phi) = 0$ for every autonomous Hamiltonian diffeomorphism $\phi.$

\item \label{itm:property 6} The invariant $w_{k,\al}$ is equivariant with respect to the action of $\Symp(M)$ on $\Ham(M)$ by conjugation and the natural  action on $\pi_0 (\cL).$ That is given $\psi \in \Symp(M),$ and $\phi \in \Ham(M),$ we have \[w_{k,\psi \cdot \al}(\psi \phi \psi^{-1}) = w_{k,\al}(\phi).\]
\item \label{itm:property 7} The invariant $w_{k,\al}$ does not change under stabilization. Given a closed connected symplectically aspherical manifold $(N,\om_N),$ and denoting by $pt_N \in \pi_0 (\cL M)$ the class of contractible loops and $\id_N \in \Ham(N)$ the identity transformation, we have for all $\phi \in \Ham(M)$ \[w_{k,\al \times pt_N}(\phi \times \id_N) = w_{k,\al}(\phi),\] where $\al \times pt_N \in \pi_0(\cL(M \times N))$ comes from the natural set isomorphism \[\pi_0(\cL(M \times N)) \cong \pi_0(\cL M) \times \pi_0(\cL N).\]

\end{enumerate}
\end{prop}

The following two propositions deal with lower bounds on the invariant $w_{k,\al}$ in certain specific situations.

\medskip
\begin{prop} \label{itm:property 5}
Assume that $\Hk$ is non-degenerate. If $\al$ is a primitive class, and all pairs of generators of $CF(\Hk)_\al$ of index difference $1$ have action difference (in absolute value) at least $D>0,$ then $w_{k,\al}(\phi_H)\geq D.$
\end{prop}

\begin{prop}\label{Propsotion: Conley}
The subset of $\phi \in \Ham(M)$ for which there exists $k \in \Z_{>0}$ with $w_{k,pt_M}(\phi) > 0$ contains a $C^\infty$ dense subset of $Ham(M).$
\end{prop}

\medskip

Since by Proposition \ref{Proposition: properties of the invariant} the subset \[\mathcal{U} = \displaystyle \bigcup_{k \in \Z_{>0}}\{\phi \in \Ham(M)| \;w_{k,pt_M}(\phi) > 0\}\] of $\Ham(M)$ is open in the metric topology of the Hofer metric, Proposition \ref{Propsotion: Conley} implies Theorem \ref{thm-generic} (see Section \ref{Section:Conley} for more details).

\medskip

In the remainder of this section we prove Propositions \ref{Proposition: properties of the invariant}, \ref{itm:property 5}, and \ref{Propsotion: Conley}.

\bs
\begin{proof}(Proposition \ref{Proposition: properties of the invariant})

Property \ref{itm:property 1} is an immediate consequence of properties \ref{itm:property 3} and \ref{itm:property 4}.


\bs

For \ref{itm:property 2} we note that if $\{ \phi_H^t \}$ and $\{\phi_{H'}^t\}$ have the same endpoint $\phi,$ then the action $\rP(\gamma^{(k)})$ of the $k$-iterated loop $\gamma^{(k)}_t:=\gamma_{kt}$ of the difference loop $\{\gamma_t=\phi^t_{H'}(\phi^t_{H})^{-1}\}_{t \in [0,1]}$ on the loop space $\cL_\al M$ by Remark \ref{Remark: FH descends to Ham} provides an identification the $r2p$ persistence modules \[[\rP(\gamma^{(k)})]: HF^{(a,b)}([\{ \phi_{\Hk}^t \}])_\al \to HF^{(a,b)}([\{ \phi_{{H'}^{(k)}}^t \}])_\al.\] What remains to show is that this identification commutes with $\rT_k.$ To this end note that the diffeomorphism $\rR_k: \cL_\al M \to \cL_\al M$ which induces $\rT_k$ on the level of $r2p$ persistence modules commutes with the diffeomorphism $\rP(\gamma^{(k)}).$ The proof is now immediate.

\bs
\begin{rmk}
Note that in fact, since we only consider shifts of intervals in the definition of $w_{k,\al},$ for the purposes of this proof, it is not actually necessary to show that the shift $I(\gamma,\eta)$ (cf. Remark \ref{Remark: FH descends to Ham}) of action functionals vanishes.
\end{rmk}



\bs

For \ref{itm:property 3} we proceed to show the statement for Hamiltonians as follows.
First assume that $k\Osc(F-G) < w_{k,\al}(G).$ We will prove

\[|w_{k,\al}(G) -w_{k,\al}(F)| \leq k \cdot \Osc(F-G). \]

Put \[u(t)=\max_M(F(t,\cdot)-G(t,\cdot)),\; v(t) = \min_M(F(t,\cdot)-G(t,\cdot)),\] \[u= \Max(F-G) = \int_0^1 u(t)\,dt, v= \Min(F-G)=\int_0^1 v(t) \,dt,\]\[ \Delta = w_{k,\al}(G) - k \cdot \Osc(F-G).\] Note that $\Delta > 0$ by assumption.
Below we'll use $u-v = \Osc(F-G)$ and \[G(t,\cdot)+u(t) \geq F(t,\cdot) \geq G(t,\cdot)+v(t).\]

By Diagram \eqref{Diagram: rotation commutes with comparison} and Lemma \ref{Lemma: continuation representation maps} we obtain the following commutative diagram, where $j_\Delta= j_{I+ku,I+ku+\Delta},$ $C_+ = C(\Gk,\Fk),$ $C_- = C(\Fk,\Gk),$ $w=w_{k,\al}(G).$
\begin{equation*}\label{Diagram: Lipschitz proof}
  \displaystyle    \begin{array}{clclclc}
         HF^{I}(\Gk)_\al & \xrightarrow{C_+} & HF^{I+ku}(\Fk)_\al & \xrightarrow{j_{\Delta}} & HF^{I+ku+\Delta}(\Fk)_\al & \xrightarrow{C_-} & HF^{I+w}(\Gk)_\al\\
         \scriptstyle{\rS_k}{\downarrow} &   &  \scriptstyle{\rS_k}{\downarrow} &  & \scriptstyle{\rS_k}{\downarrow} &  & \scriptstyle{\rS_k}{\downarrow} \\
        HF^{I}(\Gk)_\al & \xrightarrow{C_+} & HF^{I+ku}(\Fk)_\al & \xrightarrow{j_{\Delta}} & HF^{I+ku+\Delta}(\Fk)_\al & \xrightarrow{C_-} & HF^{I+w}(\Gk)_\al\\
      \end{array}
   \end{equation*}


Note that by definition of $w_{k,\al}(G),$ decreasing the translation $\Delta>0$ to $\Delta-\delta >0$ by a small $\delta > 0$ if necessary, the composition $\rS_k \circ C(\Fk,\Gk) \circ j_{\Delta} \circ C(\Gk,\Fk)$ of the top row with the rightmost vertical arrow does not vanish. By the commutativity of the diagram, this implies that $C(\Fk,\Gk) \circ j_\Delta \circ \rS_k \circ C(\Gk,\Fk) \neq 0,$ and hence \[0 \neq j_\Delta \circ \rS_k: HF^{I+ku}(\Fk) \to HF^{I+ku+\Delta}(\Fk).\]

Consequently $w_{k,\al}(F) \geq \Delta.$ Hence by definition of $\Delta,$ $w_{k,\al}(G) - w_{k,\al}(F) \leq k \cdot \Osc(F-G).$

If $w_{k,\al}(G) \geq w_{k,\al}(F),$ this proves the statement. If not, then $w_{k,\al}(F) \geq w_{k,\al}(G) > k\Osc(F-G).$ Hence the above argument having $F$ and $G$ switch places, we get $w_{k,\al}(F) - w_{k,\al}(G) \leq k \cdot \Osc(F-G).$

Hence under the assumption that $k\Osc(F-G) \leq w_{k,\al}(G),$ we have
\[|w_{k,\al}(G) -w_{k,\al}(F)| \leq k \cdot \Osc(F-G).\]

It remains to remove the assumption from the statement. We argue as follows. By the above statement, if either $w_{k,\al}(G)> k \cdot \Osc(F-G)$ or $w_{k,\al}(F)> k \cdot \Osc(F-G)$ then $|w_{k,\al}(F) - w_{k,\al}(G)| \leq k \cdot \Osc(F-G).$

In the remaining case we have both $w_{k,\al}(F) \leq k \cdot \Osc(F-G)$ and $w_{k,\al}(G) \leq k \cdot \Osc(F-G)$. Assume without loss of generality that $w_{k,\al}(F) \geq w_{k,\al}(G).$ Then since $w_{k,\al}(G) \geq 0,$

\[0 \leq w_{k,\al}(F) - w_{k,\al}(G) \leq w_{k,\al}(F) \leq k \cdot \Osc(F-G).\]

So the required inequality holds.
\bs

For \ref{itm:property 4} consider an autonomous Hamiltonian diffeomorphism $\phi$ generated by an autonomous Hamiltonian $h \in \sm{M,\R}_0.$ Consider a non-degenerate Hofer $\epsilon$-small perturbation of $\phi$ generated by time-periodic Hamiltonian $H \in \cH$ with $\Osc(H-h) < \epsilon.$ We use Definition \ref{Definition: Rotation k} of the loop-rotation operator. Note that the family of Hamiltonians $\{\Hk_r\}_{r\in [0,1]}$ satisfies $\Osc(\Hk_r-kh) < k\cdot \epsilon,$ for all $r \in [0,1].$ Since $(\rR_{r,k}^{-1})^*\A_{\Hk} = \A_{\Hk_r},$ this means that for all $x \in \cL_\al M,$ \[(\rR_{r,k}^{-1})^*\A_{\Hk} - \A_{\Hk} < 2k\cdot \epsilon.\]

\begin{rmk}
At this point a short intuitive remark is in order. Given a Morse function on a closed manifold $X,$ and an isotopy $\{\psi_r\}_{r \in [0,1]}$ with endpoint $\psi_1=\psi$ of $X,$ such that $(\psi_r)^* f-f < \delta$ for all $r \in [0,1]$ and some $\delta >0$ we can see that for a relative singular cycle $c$ in the singular chain complex of the pair $(\{f < b\}, \{f \leq a\})$ the image of the cycle $\psi_*(c)-c$ in the singular chain complex of the pair $(\{f < b+\delta\}, \{f \leq a + \delta\})$ is a boundary, by considering the trace of $c$ under the isotopy $\{\psi_r\}.$

We add that while in this proof we perturb to achieve regularity, we may have considered the $r2p$ persistence module of the degenerate Hamiltonian $h$ itself, and then relied on the above intuitive picture with $\delta=0.$
\end{rmk}

We construct a null-homotopy $\cN: CF(\Hk,J)_\al \to CF(\Hk,J)[1]_\al$ of \[\rS'_k=C_{1}\circ \rR_k - C_0\] on $CF(\Hk,J)_\al,$ for a generic almost complex structure $J \in \cJ,$ where \[C_0: CF(\Hk,J)_\al \to CF(\Hk,J)_\al\] is a self-continuation map, and \[C_1: CF(\Hk,(\rR_k)_*J)_\al \to CF(\Hk,J)_\al\] is a continuation map. The matrix element $\langle \cN(x),y\rangle$ between two generators $x,y$ of indices $i,i+1$ is defined by counting solutions to the following parametrized Floer equation, with the boundary conditions  $\{y_r = y\}_{r \in [0,1]}$ and $\{ x_r = \rR_{r,k}(x) \in \cP^\al(\Hk_r)\}_{r \in [0,1]}$ corresponding to $x$ and $y.$

%
%
%
%
%
%
%
%



For $r \in [0,1],$ choose an interpolation $\widehat{H}^{(k)}_s,$ $s \in \R$ between the function $\widehat{H}_0^{(k)}(r,t,x) =\Hk(t,x)$ and the function $\widehat{H}_1^{(k)}(s,t,x) =\Hk_r(t,x).$ For a fixed $r,s$ we obtain a Hamiltonian $\widehat{H}^{(k)}_{r,s} \in \cH.$

Choose a generic family of almost complex structures $$\{ J(r,s) \in \cJ \}_{(r,s) \in [0,1] \times \R}$$ such that $J(r,s,t) \equiv (\rR_{r,k})_* J(t)$ for $s \ll 0$ and $J(r,s,t) \equiv J(t)$ for $s \gg 0.$ Then we count solutions $(r,u_r)$ of the parametric Floer equation \[ \del_s u_r + J(r,s,t,u_r)(\del_t u_r - X_{\widehat{H}^{(k)}_{r,s}}(t,u_r)) = 0,\] with boundary conditions $u_r(s,t) \xrightarrow{s \to -\infty} x_r$ and $u_r(s,t) \xrightarrow {s \to +\infty} y_r.$

Considering the breaking of solutions in one-parameter families we see that $\cN$ is a null-homotopy of $\rS'_k,$ that is

\[\rS'_k=C_{1}\circ \rR_k - C_0= \del \cN - \cN \del.\]

We analyze the effect of this operator on the action filtration. The usual action estimates in Floer theory show that this operator does not raise the action filtration more than \[\max_{r \in [0,1]} \Osc(\Hk - \Hk_r) \leq 2k \cdot \epsilon,\] since $0  \leq \Osc(\Hk - \Hk_r) = \Osc(\Hk -k\cdot h + k\cdot h - \Hk_r) \leq \Osc(\Hk -k\cdot h) + \Osc(k\cdot h - \Hk_r) \leq k\cdot \eps + k \cdot \eps.$

Hence for the perturbation $H \in \cH$ of $h$ we have $w_{k,\al}(H) \leq 2k\cdot \eps.$ By property \ref{itm:property 3}, taking $\eps \to 0,$ we obtain $w_{k,\al}(\phi) = 0.$

\bs
For \ref{itm:property 6} consider the diffeomorphism of loop spaces \[\rP(\psi): \cL_\al M \to \cL_{\psi \cdot \al} M,\] \[z(t) \mapsto \psi(z(t)).\]

This diffeomorphism satisfies \[\rP(\psi)^* \A_{\Hk \circ \psi^{-1}} = \A_{\Hk} - I(\psi,\alpha),\]

for a constant $I(\psi,\alpha)$ equal to the (well-defined) symplectic area of a cylinder between $\eta_{\psi\cdot \alpha}$ and $\psi\circ \eta_\al$ (recall that $\eta_\al \in \cL_\al M$ is a fixed reference loop). This means that $\rP(\psi)$ defines an isomorphism \[[\rP(\psi)]: HF^{(a,b)}(\phi^k)_\al \to HF^{(a,b)+I(\psi,\al)}(\psi \phi^k \psi^{-1})_{\psi\cdot \al}\]

between the $r2p$ persistence module of $\phi^k$ in class $\al,$ and the $r2p$ persistence module of $\psi \phi^k \psi^{-1}$ in class $\psi \cdot \al$ shifted by $I(\psi,\al).$ Since the invariant is independent on shifts of $r2p$ persistence modules, it remains to see that $[\rP(\psi)]$ commutes with $\rT_k,$ which immediately follows from the fact that the diffeomorphisms $\rP(\psi): \cL M \to \cL M,$ covering $\al \to \psi\cdot \al$ on $\pi_0$ and $\rR_k: \cL M \to \cL M$ covering $\id$ on $\pi_0,$ commute.

\bs
For \ref{itm:property 7} we use the following easy lemma.

\begin{lma}\label{Lemma: 4 intervals}
Let $I_1,I_2,I_3,I_4,$ with $I_q =(a_q,b_q),\; 1 \leq q \leq 4$ be intervals such that $a_1 \leq a_2 \leq a_3 \leq a_4$ and $b_1 \leq b_2 \leq b_3 \leq b_4.$ If $j_{(a_1,b_1),(a_4,b_4)} \circ \rS_k \neq 0$ then $j_{(a_2,b_2),(a_3,b_3)} \circ \rS_k \neq 0.$
\end{lma}

\begin{proof}
The lemma follows immediately from the identity \[j_{I_1,I_4} \circ \rS_k= j_{I_3,I_4} \circ j_{I_2,I_3} \circ j_{I_1,I_2} \circ \rS_k = j_{I_3,I_4} \circ (j_{I_2,I_3} \circ \rS_k) \circ j_{I_1,I_2}.\]
\end{proof}

Put $w=w_{k,\al}(\phi),$ and $w'=w_{k,\al \times pt_N}(\phi \times \id_N).$ We note that it is enough to show that

\begin{enumerate}
\item \label{proof of itm:property 7 direction 1} if $w > 0$ then $w' \geq w,$ and
\item \label{proof of itm:property 7 direction 2} if $w' > 0$ then $w \geq w'.$
\end{enumerate}
Indeed, this would imply equality even in the case when one of the values $w,w'$ vanishes.

In the proof of both directions, we note that by property \ref{itm:property 3} we can assume that $\phi^k$ is generated by a non-degenerate Hamiltonian $\Hk \in \cH$ for the class $\alpha,$ with minimal gap $\eps >0$ in the spectrum $Spec_{k,\al}(H):=\A_{\Hk}(\cP^\al(\Hk)).$ 

Then we replace $\id_N$ by the time-one map $\phi^1_f$ of a Morse function $f,$ considered as an autonomous Hamiltonian, such that $kf$ is $C^2$-small for the purposes of Floer theory, and satisfies the estimate $|kf|_{C^0} < \delta < \frac{1}{10} \eps.$ Denote by $Spec_{k,\al}(H)+(-\delta,\delta)$ the $\delta$-neighbourhood of $Spec_{k,\al}(H)$ in $\R.$

By the Kunneth theorem in Hamiltonian Floer homology combined with the calculation of Hamiltonian Floer homology for $C^2$-small Morse Hamiltonians, and the observation that it is clearly natural with respect to the operator $\rT_k,$ we see that for each window $(a,b)$ with \[a,b \notin Spec_{k,\al}(H)+(-\delta,\delta)\] we have the isomorphism \begin{equation} \label{Equation: Kunneth stabilization} HF^{(a,b)}(\Hk+kf)_{\al \times pt_N} \cong HF^{(a,b)}(\Hk)_\al \otimes HM(kf),\end{equation}

where $HM(kf)$ denotes the Morse homology of $kf,$ which is isomorphic to the singular homology of $N,$ and the operator $\rT_k$ acts as \[\rT_k^{\Hk+kf} \cong \rT_k^{\Hk} \otimes \id_{HM(kf)},\] with respect to the isomorphism  (\ref{Equation: Kunneth stabilization}) above, where for a normalized $1$-periodic Hamiltonian $G,$ \[\rT_k^{\Gk}: HF^{(a,b)}(\Gk)_\al \to HF^{(a,b)}(\Gk)_\al\] is the loop rotation operator.


To prove direction \ref{proof of itm:property 7 direction 1} assume in addition that $\delta < \frac{1}{10} w,$ and choose, by definition of $w,$ a number $d > w - \delta $ and a window $(a,b)$ such that \[j_{(a,b),(a+d,b+d)}^{\Hk} \circ \rS_k^{\Hk}\neq 0.\]
By having chosen $\delta$ sufficiently small, there exist windows $(a_1,b_1)$ and $(a_4,b_4)$ such that \[a \leq a_1 < a + 2\delta,\]\[ b \leq b_1<b + 2\delta,\] \[a+d - 2\delta < a_4 \leq a + d,\]\[ b+d - 2\delta < b_4 \leq b +d,\] and \[a_1,b_1,a_4,b_4 \notin Spec_{k,\al}(H)+(-\delta,\delta).\] In particular we have \[j_{(a_1,b_1),(a_4,b_4)}^{\Hk} \circ \rS_k^{\Hk} \neq 0\] by Lemma \ref{Lemma: 4 intervals}. Moreover, since $a_1,b_1,a_4,b_4 \notin Spec_{k,\al}(H)+(-\delta,\delta),$ we conclude that \[j_{(a_1,b_1),(a_4,b_4)}^{\Hk+kf} \circ \rS_k^{\Hk+kf} = (j_{(a_1,b_1),(a_4,b_4)}^{\Hk} \circ \rS_k^{\Hk}) \otimes \id_{HM(kf)} \neq 0.\] It is easy to see that \[a_3 := a_1 + d - 4 \delta,\] \[b_3 := b_1 + d - 4\delta\] satisfy \[a_1<a_3<a_4\] \[b_1<b_3<b_4,\] and hence applying Lemma \ref{Lemma: 4 intervals} with \[(a_2,b_2) = (a_1,b_1)\] we have \[j_{(a_2,b_2),(a_3,b_3)}^{\Hk+kf} \circ \rS_k^{\Hk+kf} \neq 0\] and hence \[w_{k,\al \times pt_N}(\phi \times \phi^1_f)\geq w_{k,\al}(\phi)-4\delta.\] Since by property \ref{itm:property 3} \[w_{k,\al \times pt_N}(\phi \times \id_N) \geq w_{k,\al \times pt_N}(\phi \times \phi^1_f) - 2\delta,\] we obtain \[w_{k,\al \times pt_N}(\phi \times \id_N) \geq  w_{k,\al}(\phi) - 6 \delta.\]
Since this inequality holds for all sufficiently small $\delta,$ this finishes the proof of direction \ref{proof of itm:property 7 direction 1}.
The proof of direction \ref{proof of itm:property 7 direction 2} is very similar to that of direction \ref{proof of itm:property 7 direction 1} and hence we omit its details.

\bs
This finishes the proof of the proposition.
\end{proof}

\begin{proof}\label{Proof: itm:property 5}(Proposition \ref{itm:property 5})

Consider a generator $x$ of index $i$ of the Floer complex $CF(\Hk)_\al.$ Put $\A$ for its critical value. By assumption on the action difference, we see that $x$ defines a non-trivial cohomology class in $HF^{(\A-Q, \A+R)}(\Hk)_\al,$ where $0 < Q, R < D.$ So do the generators $(\phi_H)^e(x)$ for $0 \leq e \leq k-1,$ which are all different since $\al$ is primitive. Moreover, these classes persist under comparison maps $HF^{(\A-Q, \A+R)}(\Hk)_\al \to HF^{(\A-Q', \A+R')}(\Hk)_\al,$ $0 < Q, Q, R, R' < D, \; Q\geq Q', R \leq R'$ between Floer homology groups of different windows of this type. By  \eqref{Definition:Push k} and Proposition \ref{Lemma: two descriptions}, we see that $\rS_k$ is non zero on $HF^{(\A-Q, \A+R)}(\Hk)_\al.$ By Diagram \ref{Diagram: rotation commutes with comparison} with $(a,b)=(\A-D+\eps,\A+\eps),$ and $d = D-2\eps,$ for a small $\eps > 0,$ and naturality of the comparison maps, we see that $j_{D-2\eps}\circ \rS_k: HF^{(\A-D+\eps,\A+\eps)}(\Hk)_\al \to HF^{(\A-\eps,\A + D-\eps)}(\Hk)_\al$ does not vanish. Hence $w_{k,\al}(H) \geq D - 2 \eps,$ for each small $\eps>0,$ and therefore $w_{k,\al}(H) \geq D.$

\end{proof}

\begin{proof}\label{Proof: Proposition Conley}(Proposition \ref{Propsotion: Conley})

Consider the $C^\infty$-dense subset of $\Ham(M)$ consisting of Hamiltonian diffeomorphisms $\phi$ such that for all $k,$ the iteration $\phi^k$ is non-degenerate - that is its graph intersects the diagonal transversely (in fact, since we are working in the class $pt_M \in \pi_0(\cL M)$ of contractible loops, it is sufficient to require only that the intersections corresponding to contractible orbits be transverse). These diffeomorphisms are called {\em strongly non-degenerate}. For each such diffeomorphism $\phi$ of a symplectically aspherical manifold the Conley conjecture holds (cf. \cite{GG,Ginzburg,SalamonZehnder92}), which implies in particular that given $\phi$ there is an infinite subset $\mathrm{P}_\phi \subset \Z_{>0}$ of prime numbers such that for all $p \in \mathrm{P}_\phi$ the diffeomorphism $\phi^p$ has a contractible fixed point $x$ which is not a fixed point of $\phi,$ or in other words $\phi$ has a contractible primitive $p$-periodic orbit $x$. Pick any such $p$ and $x,$ and let $\A$ denote its action $\A_{H^{(p)}}(x).$ Then for an $\epsilon > 0$ that is smaller than the minimal action gap in $Spec_\al(H^{(p)}),$ which is a finite set, $[x],[\phi(x)],...,[\phi^{p-1}x]$ define a base of cycles for a subspace $V \subset HF^{(\A-\eps,\A+\eps)}.$ Moreover, this subspace $V$ is invariant with respect to the action of $\rT_p,$ which acts on $V$ by $\rT_p([\phi^j(x)]) = [\phi^{j+1}(x)].$ Clearly this persists for all windows $(\A-\eps_1,\A+\eps_2)$ where $0 < \eps_1,\eps_2 \leq \eps,$ and hence $w_{p,pt_M}(\phi) \geq \eps > 0.$
\end{proof}

\subsection{A primer on persistence modules}\label{Subsection: enriched persistence modules}

We start by noting that there is a weaker version $\widehat{w}_{k,\al} \leq w_{k,\al}$ of the $\Z_k$ spectral spread where all intervals are unbonded from below - that is when $a=-\infty$ everywhere in the definition. This version still satisfies Properties \ref{itm:property 1}-\ref{itm:property 7} of Proposition \ref{Proposition: properties of the invariant}. Moreover, it can be reformulated in the language of {\em barcodes of one-parametric persistence modules} (see Section \ref{Subsection: spread and persistence modules}) which captures additional information about the $\Z_k$-action of $\rT_k$ on filtered Floer homology, which we subsequently use to prove Theorem \ref{thm-main-squares}. In this section we collect preliminaries
on persistence modules, see \cite{BauerLesnick,CarlssonTopologyData,CarlssonZomorodianComputing1d,CarlssonEtAlBarcodesShapes,Crawley,Ghrist}.
Let us mention that by default, a persistence module is assumed to be one-parametric, and that
we work in the simplest setting suitable for our purposes.

\subsubsection*{Persistence modules}

Let $\K$ be a field. A (one-parametric) {\it persistence module} is given by
a pair $(V,\pi),$
where
\begin{itemize}
\item[{(i)}] $V_t$, $t \in \R$ are finite dimensional $\K$-vector spaces;
\item[{(ii)}]\label{item: compactly supported persistence modules} {\bf (compact support)} $V_s=0$ for all $|s|$ sufficiently large;
\item[{(iii)}]{\bf (persistence)} $\pi_{st}: V_s \to V_t$, $s < t$ are morphisms satisfying $\pi_{sr} = \pi_{tr}\circ\pi_{st}$ for all $s,t,r \in \R$.
\item[{(iv)}] {\bf (semicontinuity)} For every $r \in \R$ there exists $\epsilon>0$ such that
$\pi_{st}$ is an isomorphism for all $r-\epsilon < s < t \leq r$;
\item[{(v)}] {\bf (finite spectrum)} For all but a finite number of points $r \in \R$, there exists a neighborhood $U$ of $r$ such that $\pi_{st}$ is an isomorphism for all $s <t$ with $s,t \in U$.
    The exceptional points form {\it the spectrum} of the persistence module.
\end{itemize}
For the sake of brevity, we often denote the persistence module $(V,\pi)$ or simply $V$ and
write $\cS(V)$ for its spectrum.

It is instructive to mention that for every pair $a < b$ of consecutive points of the spectrum,
the morphism $\pi_{st}$ is an isomorphism for all $s,t $ with $a< s< t \leq b$.
This readily follows from axioms (iii) and (iv) and the compactness of $[s,t]$.

\begin{example}
Filtered Floer homology of an $\al$-atoroidal symplectically aspherical symplectic manifold for a non-degenerate Hamiltonian diffeomorphism as defined in Section \ref{Section:HF} gives an example of a persistence module. We refer to it as the {\em Floer persistence module}.
\end{example}

\subsubsection*{Morphisms} A morphism $A$ between persistence modules $(V,\pi)$ and
$(V',\pi')$ is a family of linear maps $A_t: V_t \to V'_t$
which respect the persistence morphisms:
$$A_t\pi_{st} =\pi'_{st}A_s\;$$
for all $s<t$.

Persistence modules and their morphisms form a category. Thus we can speak about isomorphic
persistence modules. One readily checks that isomorphic modules have the same spectra.

\begin{example}\label{Example: morphism via continuation map}
The continuation map $HF^{(-\infty,a)}(F)_\al \to HF^{(-\infty,a+\Max(G-F))}(G)_\al$ between Floer homologies of two non-degenerate Hamiltonians $F,G$ gives a morphism between the corresponding persistence modules. Note that the second persistence module is shifted by $\Max(G-F).$
\end{example}

\subsubsection*{The structure theorem } Let us formulate the main structure theorem for (compactly supported semi-continuous) persistence modules, see e.g. \cite{Crawley}.

Given two  persistence modules $(V,\pi)$ and
$(V',\pi')$, we define their direct sum as
$$(\{V_t \oplus V'_t\},\{\pi_{st} \oplus \pi'_{st}\})\;.$$

Let $I=(a,b]$ be an interval, where $a,b \in \R$.
Introduce a persistence module  $Q(I) = (\{Q(I)_t\}, \theta)$ given by $(Q(I))_t = \K$ for $t \in I$ and $(Q(I))_t = 0$ otherwise, while the morphisms $\theta$ are the identity maps within $I$ and zeroes otherwise.

\begin{thm}[The structure theorem for persistence modules]\label{Theorem: structure of persistence modules}

For every persistence module $V$ there exists a unique collection of pair-wise distinct intervals $I_j = (a_j,b_j]$, $a_j,b_j \in \cS(V)$, $j=1,...,N$
and multiplicities $m_1,...,m_N \in \N$ so that
\begin{equation}\label{eq-structure}
V \cong \oplus_{j=1}^N (Q(I_j))^{m_j}\;,
\end{equation}
where  $$(Q(I_j))^{m_j}= Q(I_j) \oplus \dots \oplus Q(I_j)$$ $m_j$ times.
\end{thm}

\medskip
\noindent
The collection of intervals $\{I_j,m_j\}$ with multiplicities appearing
in the normal form of the persistence module $V$ is called the barcode of $V$ and is denoted
by $\cB(V)$.

\subsubsection*{Interleaving distance} For a persistence module $V=(\{V_t\},\{\pi_{st}\})$
and $\delta \in \R$ denote by $V^\delta$ the shifted module $$V^\delta=(\{V_{t+\delta}\},\{\pi_{s+\delta,t+\delta}\})\;.$$ Note that $\cS(V^\delta) = \cS(V)+\delta.$ Observe that for $\delta \in \R_{>0},$ the persistence modules $V$ and $V^\delta$ are related by a canonical shift morphism $\phi_V(\delta)$ given by $\pi_{t,t+\delta}: V_t \to V_{t+\delta}$.

For a morphism $f: V \to W$ write $f(\delta)$ for the induced morphism $V^\delta \to W^\delta$.

We say that two persistence modules $V$ and $W$ are $\delta$-interleaved ($\delta >0$) if there exists
morphisms $f: V \to W^\delta$ and  $g: W \to V^\delta$ so that
$$f(\delta)g= \phi_W(2\delta), \;g(\delta)f= \phi_V(2\delta)\;.$$

The interleaving distance $d_{inter}$ between two isomorphism classes of persistence
modules equals the infimum of $\delta$ such that these modules are $\delta$-interleaved.

Observe that since any $V$ is compactly supported, $\phi_V(\delta)=0$ for all sufficiently large $\delta$,
so the interleaving distance between any two modules is finite: take such a $\delta$ and $f=g=0$.

\subsubsection*{Bottleneck distance between barcodes}
Recall that a {\it matching} between finite sets $X$ and $Y$ is a
bijection $\mu: X' \to Y',$  where $X' \subset X$ and $Y' \subset Y$.
We denote $X'= \text{coim}(\mu)$ and $Y'=\text{im}(\mu)$.

Given a barcode $\cB=\{I_i,m_i\}$, $i=1,...,M$, consider a  set
$$\langle \cB \rangle:= \{I_i^{(k_i)},\; k_i=1,...,m_i,\;\;i=1,...,M\}\;.$$
Let $\langle \cB \rangle_\epsilon \subset \langle \cB \rangle$ be a subset consisting
of all intervals of length $> \epsilon$.

For an interval $I=(a,b]$ put $I^{-\delta} = (a-\delta,b+\delta]$.

For $\delta >0$, define a $\delta$-matching between barcodes $\cB$ and $\cC$
as a matching $\mu$ between $\langle \cB \rangle$ and $\langle \cC \rangle$ such that
\begin{itemize}
\item $\langle \cB \rangle_{2\delta} \subset \text{coim}(\mu)$;
\item $\langle \cC \rangle_{2\delta} \subset \text{im}(\mu)$;
\item $\mu(I) = J \Longrightarrow I \subset J^{-\delta}, J \subset I^{-\delta}$.
\end{itemize}
The bottleneck distance $d_{bottle}(\cB,\cC)$ is defined as the infimum of $\delta$ such that the barcodes
$\cB$ and $\cC$ admit a $\delta$-matching.

\subsubsection*{Isometry theorem}

The interleaving distance between two persistence modules and the bottleneck distance between their barcodes are related by the Isometry Theorem - cf. Bauer-Lesnick \cite{BauerLesnick} and references therein.


\begin{thm}[The Isometry Theorem for persistence modules and barcodes]\label{Theorem: Isometry theorem}
$d_{inter}(V,W) = d_{bottle}(\cB(V),\cB(W)).$
\end{thm}

\subsubsection*{Multiplicity function}
Let $\cB$ be a barcode (recall that in this section all barcodes consist of finite intervals
only since they correspond to compactly supported persistence modules).

For an interval $I \subset \R$ denote the by $m(\cB,I)$ the number of bars in $\cB$ (with multiplicities!) containing $I$. For an interval $I=(a,b]$ of length $> 2c$ put $I^c=(a+c,b-c]$.

\medskip
\noindent
\begin{prop}\label{prop-dbottle}
Assume that $d_{bottle}(\cB,\cC) < c$ and for an interval $I$ of length $> 4c$
\begin{equation}\label{eq-bottle}
m(\cB,I) = m(\cB,I^{2c})= l.
\end{equation}
Then $m(\cC,I^c) = l$.
\end{prop}

\begin{proof}
First, \eqref{eq-bottle} yields that if a bar $E$ of $\cB$ contains $I^{2c}$,
it necessarily contains $I$.

Next, by definition of the bottleneck distance, there exists a $\delta$-matching $\nu$
between $\cB$ and $\cC$ with $\delta < c$. Note that the set of bars of $\cB$ containing $I$ lies in $\text{coim}(\nu)$ and the set of bars of $\cC$ containing $I^c$ lies in $\text{im}(\nu).$ We claim that $\nu$ establishes a bijection between these two sets.

Let $E \in \langle \cB \rangle$ be a bar containing $I$. Then $E \subset \nu(E)^{-\delta},$ implying \[I^c \subset E^\delta \subset \nu(E).\] In the other direction, let $J \in \langle \cC \rangle$ be a bar containing $I^c.$ Then $J=\nu(E)$ for some $E \in \langle \cB \rangle$ and by the same argument as before \[I^{2c} \subset J^\delta \subset E.\] Hence by our first observation $E$ contains $I.$ This finishes the proof.
\end{proof}

\subsection{Persistence modules with a $\Z_k$-action}\label{subsec-persist-invol}
The main result of the present section, Theorem \ref{thm-pmi-main} below, is an algebraic counterpart
of the geometric ``distance to $k$-th powers" problem appearing in Theorem \ref{thm-main-squares}.

Let  $(V,\pi)$ be a {\em {$\Z_k$} persistence module}, which is a persistence module equipped with a $\Z_k$-action, given by an automorphism $A: V \to V$ of persistence modules (this means that $A$ is a morphism $V \to V$ with $A^k = \id$). We denote this data by $(V,\pi;A)$. For simplicity of exposition we assume that $k=p$ is a prime, and fix it.

\medskip
\noindent
{\bf Assumption:} We shall assume that the ground field $\K$ has characteristic $\neq p,$ contains all $p$-th roots of unity\footnote{That is the polynomial $x^p-1 \in \K[x],$ which is separable by the assumption $\mathrm{char}(\K) \neq p,$ splits over $\K.$}, and fixing a primitive $p$-th root of unity $\zeta_p,$ the equation $x^p -(\zeta_p)^q = 0,$ for any integer $q$ not divisible by $p,$ has no solutions in $\K.$ The following lemma gives an example of a field satisfying this condition.

\begin{lma}\label{Lemma: equation has no solution}
Let $p \geq 2$ be a prime number, and let $\Q_p$ be the cyclotomic field obtained from $\Q$ by adjoining a primitive $p$-th root of unity $\zeta_p.$ Then the equation $x^p = \zeta_p^q,$ where $q \in \Z$ with $gcd(p,q)=1,$ has no solution $x$ in $\Q_p.$
\end{lma}

\begin{proof}
Assume on the contrary that $x$ is a solution. Choose a primitive $p^2$-th root of unity $\zeta_{p^2}$ such that $\zeta_{p^2}^p = \zeta_p.$ In the algebraic closure $\overline{\Q}$ of $\Q,$ for some $\tau \in \N,$  $x = \zeta_{p^2}^q \cdot (\zeta_p)^\tau \in \Q_p,$ thus $\zeta_{p^2}^q$ lies in $\Q_p.$ Also, $\zeta_{p^2}^p = \zeta_p$ lies in $\Q_p.$ Taking $u,v \in \Z$ with $uq+vp=1,$ we get that $\zeta_{p^2}= (\zeta_{p^2}^q)^u \cdot (\zeta_{p^2}^p)^v \in \Q_p.$ Thus $\Q_{p^2}= \Q_p,$ contradicting basic theory of cyclotomic fields (cf. \cite[Corollary 7.8]{Morandi}).
\end{proof}

\begin{rmk}
For $p=2$ a {$\Z_p$ persistence module} is a persistence module $V$ with involution $A,$ the primitive root of unity is $\zeta_2 = -1,$ the condition on $\K$ is that $\mathrm{char}(K) \neq 2$ and $x^2+1 = 0$ has no solution in $\K,$ and an example of such a field $\K$ is simply $\Q_2 = \Q.$ Another example is $\K = \Z_3.$
\end{rmk}

The following lemma serves to relate the assumption on the ground field $\K$ to the multiplicities of barcodes of persistence modules.
\medskip

\begin{lma}\label{Lemma: dimension divisible by m}
Let $p$ be a prime number, $V$ - vector space over $\K$ with an operator $B$
satisfying $B^p = \zeta_p \cdot \id.$ Then $p$ divides $\dim V.$
\end{lma}

\begin{proof}
As $(\det B)^p = (\zeta_p)^{\dim V},$ the lemma follows immediately from the property of $\K.$
\end{proof}

\medskip
An {\em equivariant interleaving} between {$\Z_p$ persistence modules} $V$ and $W$ is defined as an interleaving respecting
the $\Z_p$-actions. Given two {$\Z_p$ persistence modules}, define the distance $\widehat{d}_{inter}$ between them
as the infimum of $\delta$ such that they admit an equivariant $\delta$-interleaving.
Clearly, $\widehat{d}_{inter} \geq d_{inter}$.

A {$\Z_p$ persistence module} $(V,\pi;A)$ is called {\it a full $p$-th power} if $A=B^p$ for some morphism $B: V \to V$.
Note that $B^{p^2} = \id$ and $B$ automatically commutes with $A$.

\medskip
\noindent
\begin{df} For a {$\Z_p$ persistence module} $V:=(V,\pi;A)$ set \[\kappa(V)=\inf_W \widehat{d}_{inter}(V,W),\]
where the infimum is taken over all full $p$-th powers $W$.
\end{df}

\medskip
\noindent Our objective is to give a lower bound for $\kappa(V)$ through a barcode
of an auxiliary persistence module $(L_\zeta)_t = \ker(A_t - \zeta\cdot\id),$ for $\zeta \neq 1$ a $p$-th root of unity. Clearly $A_t$ descends to the morphism $A_t = \zeta\cdot \id :(L_\zeta)_t \to (L_\zeta)_t.$ We call $L_\zeta$ the $\zeta$-eigenspace of the {$\Z_p$ persistence module} $V.$

\begin{rmk}
We remark that the persistence module $(L)_t = V_t/\text{Fix}(A_t),$ which appears in Section \ref{Subsection: spread and persistence modules} below, in this situation is isomorphic to the direct sum of eigenspaces with eigenvalue different from $1:$ \[L \cong \displaystyle\bigoplus_{\zeta \neq 1,\, \zeta^p=1} L_\zeta.\] Indeed it is an easy exercise to check that, fixing a primitive $p$-th root of unity $\zeta,$ the operators $\{\pi_k:V \to V\}_{0\leq k \leq p-1}$ given by $\pi_k(v) = \frac{1}{p}\sum \zeta^{-jk} A^j(v)$ give a splitting \[V \cong \displaystyle\bigoplus_{0\leq k \leq p-1} L_{\zeta^k}.\] Indeed $\mathrm{Image}(\pi_k) \subset L_{\zeta^k}$ and $v = \sum_{k=0}^{p-1} \pi_k(v)$ for all $v \in V.$
\end{rmk}

Note that every equivariant interleaving between the $\Z_p$ persistence modules $V$ and $W$ descends to an interleaving
between the $\zeta$-eigenspaces, and hence, by the isometry theorem,
\begin{equation}\label{eq-dist-inter-shad}
\widehat{d}_{inter}(V,W) \geq d_{inter}(L_\zeta,K_\zeta) = d_{bottle}(\cB(L_\zeta),\cB(K_\zeta))\;.
\end{equation}

\medskip
\noindent

\begin{prop}\label{prop-multipl}
Let $L_\zeta$ be the $\zeta$-eigenspace of a full $p$-th power {$\Z_p$ persistence module} for a primitive $p$-th root of unity $\zeta.$ Then for every interval $I \subset \R$ the multiplicity
$m(\cB(L_\zeta),I)$ is divisible by $p.$
\end{prop}

\begin{proof}
Assume that $A_t=B_t^p$. Since $B_t$ commutes with $A_t$
it preserves $(L_\zeta)_t.$ Let $B'_t: (L_\zeta)_t \to (L_\zeta)_t$ denote the restriction $B'_t = (B_t)|_{(L_\zeta)_t}$ to $(L_\zeta)_t.$ Note that $B'_t$ satisfies $(B'_t)^p = \zeta$. Furthermore, denoting by $\theta_{st}: (L_\zeta)_s \to (L_\zeta)_t$ the persistence morphisms and using that $B'_t\theta_{st}=\theta_{st}B'_s$, we get that $\text{Image}(\theta_{st})$ is invariant under $B'_t$.

By Lemma \ref{Lemma: dimension divisible by m} one concludes that the dimension of $\text{Image}(\theta_{st})$ {\bf is divisible by $p$} for all $s <t$. By looking at the normal form of $\{(L_\zeta)_t\},$
we readily conclude the desired statement, since for an interval $I=(a,b]$, every bar containing $I$ contributes $1$
to $\dim \text{Image}(\theta_{a_+,b_-})$, while all other bars contribute $0$.
\end{proof}

\medskip
\noindent
\begin{df}\label{df-sens-spread}
For a primitive $p$-th root of unity $\zeta$ define $\mu_{p,\zeta}(V,A)$ as the supremum of those $c \geq 0$ for which there exists an interval $I$ of length $>4c$ and  such that $m(\cB(L_\zeta),I) = m(\cB(L_\zeta),I^{2c})=l$ with $l \neq 0 \mod{p},$ where $L_\zeta$ is the $\zeta$-eigenspace of $V$. The {\it multiplicity sensitive spread} $\mu_p(V,A)$ of a {$\Z_p$ persistence module} $V$ is defined as \[\mu_{p}(V,A) = \max_{\zeta \neq 1:\;\zeta^p=1}\mu_{p,\zeta}(V,A).\]

\end{df}

\medskip
\noindent
\begin{prop}\label{mu2-Lip} $|\mu_p(V)-\mu_p(W)| \leq \widehat{d}_{inter}(V,W)$
for every pair of {$\Z_p$ persistence modules} $V$ and $W$.
\end{prop}

\begin{proof}
We first claim that it is enough to show that \[|\mu_{p,\zeta}(V)-\mu_{p,\zeta}(W)| \leq \widehat{d}_{inter}(V,W)\] for all $p$-th roots of unity $\zeta \neq 1.$ By the symmetry between $V$ and $W$ we can assume that $\mu_p(V) \geq \mu_p(W) \geq 0.$ Let $\zeta_0$ be such that $\mu_p(V) = \max_{p,\zeta_0}(V),$ that is the maximum in Definition \ref{df-sens-spread} is attained at $\zeta_0.$ Then \[|\mu_p(V) - \mu_p(W)| = \mu_{p,\zeta_0}(V) - \mu_p(W) \leq \mu_{p,\zeta_0}(V) - \mu_{p,\zeta_0}(W),\] whence the claim is immediate.

Now fix a $p$-th root of unity $\zeta \neq 1.$ Let $L_\zeta$ and $K_\zeta$ be the $\zeta$-eigenspaces of $V$ and $W$, respectively. By \eqref{eq-dist-inter-shad} it is enough to show \begin{equation}\label{eq-mu2-Lip-bottle}
|\mu_{p,\zeta}(V)-\mu_{p,\zeta}(W)| \leq d_{bottle}(\cB(L_\zeta),\cB(K_\zeta)).
\end{equation} Let $d_{bottle}(\cB(L_\zeta),\cB(K_\zeta)) =\epsilon$. By symmetry between $V$ and $W$ it suffices to show that
\begin{equation}\label{eq-vsp-mult-shad}
\mu_{p,\zeta}(V) -\mu_{p,\zeta}(W) \leq \epsilon\;.
\end{equation}
Let $\mu_{p,\zeta}(V) =c > 0$. Assume without loss of generality that $\epsilon < c$.
Take an interval $I$ such that $m(\cB(L_\zeta),I)= m(\cB(L_\zeta),I^{2c})= l$ with $l \neq 0 \mod p$.
This yields
$$m(\cB(L_\zeta),I^{2c})= m(\cB(L_\zeta),I^{2c-2\epsilon})=l,\;\; m(\cB(L_\zeta),I) = m(\cB(L_\zeta), I^{2\epsilon})=l\;,$$
and hence by Proposition \ref{prop-dbottle}
$$m(\cB(K_\zeta),I^{2c-\epsilon})= m(\cB(K_\zeta),I^{\epsilon})=l\;.$$
It follows that
$\mu_{p,\zeta}(W) \geq c-\epsilon$ which yields \eqref{eq-vsp-mult-shad}.
\end{proof}

\medskip
\noindent
\begin{thm}
\label{thm-pmi-main} $\kappa(V) \geq \mu_p(V)$ for every {$\Z_p$ persistence module} $V$.
\end{thm}

\begin{proof}
By Proposition \ref{prop-multipl} $\mu_{p}(W)= 0$ for every full $p$-th power {$\Z_p$ persistence module}
$W$. Thus by Proposition \ref{mu2-Lip}
$$\widehat{d}_{inter}(V,W) \geq \mu_p(V)\;,$$
and hence $\kappa (V) \geq \mu_p(V)$.
\end{proof}

\subsection{A new invariant of Hamiltonian diffeomorphisms} \label{Example: pmi from a square}
Fix a prime number $p.$ Write $\G_p$ for the set of Hamiltonian diffeomorphisms whose primitive $p$-periodic
orbits are non-degenerate. Take any $\phi \in \G_p$ and fix a primitive free homotopy class $\al$ of loops.
Then for each degree $r \in \Z$
$$
V_r(\phi) := (HF^{(-\infty,\cdot)}_r(\phi^p)_\al,[\rR_p])
$$
is a {$\Z_p$ persistence module}. Here $\rR_p$ is the loop rotation operator induced by $t \mapsto t+1/p$, or, equivalently,
induced by the action of $\phi$ on the filtered Floer homology of $\phi^p$. Clearly, it induces a $\Z_p$-action
on the persistence module $HF^{(-\infty,\cdot)}_r(\phi^p)_\al$. Fixing a primitive $p$-th root of unity $\zeta \neq 1,$ we write $L_r(\phi)=L_r(\phi)_\zeta$ for the $\zeta$-eigenspace of $V_r(\phi)$, and $\cB_r(\phi)=\cB_r(\phi)_\zeta$ for the barcode of $L_r(\phi)$.

By Lemma \ref{Lemma: continuation representation maps}, for every $\phi,\psi \in \G_p$ the
{$\Z_p$ persistence modules} $V_r(\phi)$ and $V_r(\psi)$ are equivariantly $\delta$-interleaved
with $\delta= p \cdot d(\phi,\psi)$, where $d$ stands for Hofer's metric. Therefore \eqref{eq-dist-inter-shad}
implies that
\begin{equation}\label{eq-dist-inter-shad-1}
 d_{bottle}(\cB_r(\phi),\cB_r(\psi)) \leq \widehat{d}_{inter}(V_r(\phi),V_r(\psi)) \leq p\cdot d(\phi,\psi)\;\;\forall \phi,\psi \in \G_p\;.
\end{equation}
In particular, the barcode mapping
$$\phi \mapsto \cB_r(\phi)$$
is Lipschitz with respect to Hofer's metric $d$ on $\Ham(M)$ and the bottleneck distance
on the space of barcodes.

Put $\mu_{p,\zeta}(r,\phi) = \mu_{p,\zeta}(V_r(\phi),A_r(\phi)),$ and \[\mu_p(r,\phi)  = \max_{\zeta \neq 1:\;\zeta^p=1}\mu_{p,\zeta}(r,\phi) = \mu_p(V_r(\phi)).\] Define {\it the multiplicity sensitive spread} $\mu_p(\phi)$ of a Hamiltonian diffeomorphism as \[\mu_p(\phi) = \max_{r \in \Z}\mu_p(r,\phi).\] By  \eqref{eq-dist-inter-shad-1} and \eqref{eq-mu2-Lip-bottle} we get that
\begin{equation}\label{eq-ineq-mu2-hofer}
|\mu_p(\phi) -\mu_p(\psi)| \leq p\cdot d(\phi,\psi)\;.
\end{equation}
Since $\G_p$ is Hofer-dense in $\Ham(M)$, it follows that the invariant $\mu_p$ can be extended
by continuity to the whole $\Ham(M)$, and the extension is still $p$-Lipschitz in Hofer's norm.

The multiplicity sensitive spread yields the following estimate for the distance to $\Pow_p$:

\medskip
\noindent
\begin{thm}\label{thm-mu2-squares}
$d(\phi, \Pow_p) \geq \frac{1}{p}\cdot \mu_p(\phi)$
for all $\phi \in \Ham(M)$.
\end{thm}

\begin{proof} Let $\theta =\psi^p$ be a $p$-th power of a Hamiltonian diffeomorphism.
Then $V_r(\theta)$ is a full $p$-th power {$\Z_p$ persistence module} for each $r \in \Z.$ Indeed, given a $1$-periodic Hamiltonian $F(t,x)$ generating $\psi,$ we can choose the $1$-periodic Hamiltonian $H(t,x)$ generating $\theta$ as $H(t,x) = F^{(p)}(t,x) = pF(pt,x).$ Then $H^{(p)}(t,x) = p^2 F(p^2 t,x) = F^{(p^2)}(t,x).$ Therefore the operator \[[\rR_{p^2}(F)]:HF^{(-\infty,a)}_r(H^{(p)})_\al \to HF^{(-\infty,a)}_r(H^{(p)})_\al\] is well defined, and moreover the identity \[t+\underbrace{\frac{1}{p^2}+ \ldots+\frac{1}{p^2}}_{p\; \text{times}} = t + \frac{1}{p}\] in $S^1$ shows that \[[\rR_{p^2}(F)]^p = [\rR_p(H)].\]

By \eqref{eq-dist-inter-shad-1} and Theorem \ref{thm-pmi-main}
$$p \cdot d(\phi,\Pow_p) \geq \max_{r,\zeta} \kappa(L_r(\phi)_\zeta) \geq \mu_p(\phi)\;,$$
as required.
\end{proof}

Finally we observe that that the multiplicity sensitive spread is invariant under stabilization.

\begin{thm}\label{thm-mu2-stabilization}
For $\phi \in \Ham(M),$ $\al \in \pi_0(\cL M),$ and any closed connected symplectically aspherical manifold $N,$ consider the stabilization $\phi \times \id \in {\Ham(M \times N)}$ of $\phi.$ Then we have \[\mu_p(\phi) = \mu_p(\phi \times \id_N),\] the latter value being computed in the class $\al \times pt_N$ in $\pi_0(\cL (M \times N)).$
\end{thm}

\begin{proof} Clearly it is enough to prove that
\[\mu_{p,\zeta}(\phi) = \mu_{p,\zeta}(\phi \times \id_N),\]
for all $p$-th roots of unity $\zeta \neq 1.$

Put $c=\mu_{p,\zeta}(\phi) = \max_r \mu_{p,\zeta}(r,\phi),$ and $c' = \mu_{p,\zeta}(\phi \times \id_N).$ As carried out in the proof of Proposition \ref{Proposition: properties of the invariant}, Item \ref{itm:property 7}, we perturb $\id_N$ to the time-one map of the Hamiltonian flow of a $C^2$-small Morse function on $N,$ and argue up to a small $\delta>0.$ We omit these considerations herein.

We show first that $c \leq c'.$ Take the minimal $r$ with $\mu_{p,\zeta}(r,\phi)=c.$ Let $I \subset \R$ be an interval such that $m( \cB_r(\phi)_\zeta,I) = m( \cB_r(\phi)_\zeta,I^{2c})= l,$ where $l \neq 0 \mod p$. Consider $\mu_{p,\zeta}(\phi \times 1).$ By the Kunneth theorem in Floer homology, the barcode $\cB_r(\phi \times 1)_\zeta$ consists of $b_i=\dim H_i(N)$ copies of $\cB_{r-i}(\phi)_\zeta,$ $i=0,1,..., \dim N.$

Note that for $i>0$ none of $\cB_{r-i}(\phi)_\zeta$ contains both $I$ and $I^{2c}$ with equal multiplicities not divisible by $p$, since otherwise we get a contradiction to $r$ being the minimal degree where the maximum $c$ is attained.

Thus we are left with $i=0,$ and since $b_0=1,$ we have $l$ copies of $I$ and $I^{2c}$
in the corresponding piece of the barcode $\cB_r(\phi \times 1)_\zeta.$ It follows that

\[\mu_{p,\zeta}(\phi\times \id) \geq \mu_{p,\zeta}(r, \phi \times \id) \geq c = \mu_{p,\zeta}(\phi).\]

In the other direction, we show that $c \geq c'.$ Assume that $m(\cB_r(\phi \times \id)_\zeta,I) = m(\cB_r(\phi \times \id)_\zeta,I^{2c'}) = l,$ where $l \neq 0 \mod p$. Define $S = \{i\;|\; b_i \neq 0\}.$ For each $i \in S$ put $m_i,m'_i$ for the multiplicities of $I$ and $I^{2c'}$ in $\cB_{r-i}(\phi)_\zeta.$

We claim that $m_i = m'_i$ for all $i \in S.$ Indeed, by definition of the multiplicity function, $m_i \leq m'_i.$ Hence the claim follows by the identities \[l = m(\cB_r(\phi \times \id)_\zeta,I) = \sum_{i \in S} m_i b_i,\] \[l = m(\cB_r(\phi \times \id)_\zeta,I^{2c'}) = \sum_{i \in S} m'_i b_i.\]  To conclude the proof, we note that $l = \sum m_i b_i \neq 0 \mod p,$ and hence there exists $i_0$ with $m_{i_0} \neq 0 \mod p.$ Hence \[\mu_{p,\zeta}(\phi) \geq \mu_{p,\zeta}(r-i_0, \phi) \geq c' = \mu_{p,\zeta}(\phi \times \id).\]
\end{proof}

\section{Hamiltonian egg-beater and the proof of the main results}\label{Section: example}

Here we prove the following statement and show how to deduce Theorems \ref{thm-main} and \ref{thm-main-squares} from it.

\medskip
\begin{prop}\label{Proposition: example}
On a surface $\Sigma$ of genus $g \geq 4,$ there exists a family $\{\phi_\lambda\},$ for $\la$ in an unbounded increasing sequence in $\R$, of Hamiltonian diffeomorphisms of $\Sigma,$ and a family of primitive classes $\alpha_\lambda \in \pi_0(\cL \Sigma)$ such that for $\la$ large, $(\phi_\la)^p$ has exactly $2^{2p}$ $p$-tuples $\{z,\phi_\la(z),\ldots,(\phi_\la)^{p-1}(z)\}$ of (primitive) non-degenerate fixed points in class $\al_\la,$ with action differences \[|\cA(z)-\cA(z')| \geq c\cdot \lambda + O(1),\] for each $z,z'$ belonging to different $p$-tuples, as $\lambda \to \infty$ for a certain constant $c>0.$
\end{prop}

\begin{rmk}
It sounds likely that by slightly modifying our construction one can prove an analogue of Proposition \ref{Proposition: example} and hence of Theorems \ref{thm-main} and \ref{thm-main-squares}, in the case when the genus of the surface $\Sigma$ is $g \geq 2.$
\end{rmk}

\subsection{Proof of Theorems \ref{thm-main} and \ref{thm-main-squares}}
By Proposition \ref{itm:property 5}, Proposition \ref{Proposition: example} implies that \[w_{p,\alpha_\lambda}(\phi_\lambda) \geq {c}\cdot \lambda + O(1),\] as $\lambda \to \infty,$ for a constant $c>0.$ Consequently, Proposition \ref{Proposition: properties of the invariant}: \ref{itm:property 3}, \ref{itm:property 4}, and \ref{itm:property 7} yield Theorem \ref{thm-main}.

\medskip


Further, among the $2^{2p}$ $p$-tuples of fixed points of $\phi_\lambda^p$ in the class $\alpha_\la$
choose the $p$-tuple, say $\{z,\phi_\lambda(z),\ldots, (\phi_\la)^{p-1}(z)\}$ with the minimal action. Let $r$ be the index of $z.$
Fix a small $\varepsilon > 0$ and choose the interval \[I_{\varepsilon,\la} = (\cA(z)+\frac{1}{2}\varepsilon \cdot \la,\cA(z)+(c-\frac{1}{2}\varepsilon) \cdot \la]\] of length \[|I_{\varepsilon,\la}| = (c-\varepsilon)\la> 4 \cdot \frac{\la}{4} (c-2\varepsilon)\;.\]

For a primitive $p$-th root of unity $\zeta,$ we claim that the multiplicity of the interval $I_{\varepsilon,\la}$ in the barcode of the $\zeta$-eigenspace $L_r(\phi_\lambda)_\zeta$ of the {$\Z_p$ persistence module} $V_r(\phi)$ equals $1$:
\[m(\cB_r(\phi_\la),I_{\varepsilon,\la}) = 1\]
Indeed, when $\la$ is large, the $\zeta$-eigenspace $(L_r(\phi_\la)_\zeta)_t$ for all \[\cA(z)<t < \cA(z) + (c-\varepsilon) \cdot \la\] is one-dimensional, and is in fact spanned by $\sum_{j=0}^{p-1}\zeta^{-j}[(\phi_\la)^j(z)].$ Moreover also \[m(\cB_r(\phi_\la),I_{\varepsilon,\la}^{\frac{\la}{2}(c-2\varepsilon)})=1.\]

By the definition of the multiplicity-sensitive spread, we conclude that $\mu_p(\phi_\lambda) \geq \la (c-2\varepsilon )/4$, and hence by Theorem \ref{thm-mu2-squares},
$d(\phi_\lambda, \Pow_p) \geq  \la (c-2\varepsilon )/4p.$
This, together with Theorem \ref{thm-mu2-stabilization}, proves Theorem \ref{thm-main-squares}.

\qed

The rest of this section is dedicated to the proof of Proposition \ref{Proposition: example}.

\subsection{Topological set up}
We begin by describing the topological setting of the construction. To this end consider the union of two annuli $C_V$ and $C_H$ in the standard $\R^2,$ each symplectormophic to \[C_*=[-1,1] \times \R/L\Z,\]  for a number $L\geq 4,$ intersecting at two squares $A, B.$ Here the subscripts $V$ and $H$ stand for "vertical" and "horizontal", as seen from the square $A$ - compare Figure \ref{fig:two-annuli}. Then consider this configuration of annuli as being symplectically embedded in $S^2$ and glue at least one handle into each connected component. This gives the surface $\Sigma$ of genus $\geq 4.$

More precisely, consider the cylinder $C_*$ with coordinates $x,y$ and standard symplectic form $dx\wedge dy.$ Take two copies of this cylinder, denoted by $C_V$ and $C_H.$ Consider the squares $S_0=[-1,1]\times [-1,1]/L\Z$ and $S_1=[-1,1] \times [L/2-1,L/2+1]/L\Z.$ This gives us four squares $S_{V,0},S_{V,1} \subset C_V$ and $S_{H,0}, S_{H,1} \subset C_H.$ Consider the symplectomorphism $VH_{0,1}: S_{V,0} \sqcup S_{V,1} \to S_{H,0} \sqcup S_{H,1},$ given by $VH \sqcup VH',$ where $VH:S_{V,0} \to S_{H,0}$ is given by $VH(x,[y])=(-y,[x]),$ and $VH':S_{V,1} \to S_{H,1}$ is given by $VH'(x,[y]) = (y-L/2, [-x+L/2]).$
Glue the two cylinders along $VH_{0,1}$ to obtain the following surface with boundary \[C:=C_V \cup_{VH_{0,1}} C_H.\] Next we consider a symplectic embedding $\iota_0: C \hookrightarrow S^2$ of $C$ into $S^2$ (this can be seen by first embdedding it symplectically into $\R^2$ with the standard symplectic structure inside a large ball and then composing with an embedding of that ball into $S^2$, or alternatively embed $C$ as the union of tubular neighbourhoods of two orthogonal great circles in $S^2$), and identify $C$ with the image of the embedding.

Note that $S^2 \setminus C$ has $4$ connected components. We glue at least one handle into each of these components, to obtain a surface $\Sigma$ of genus $g \geq 4,$ with a symplectic embedding \[\iota: C \hookrightarrow \Sigma.\]

\begin{figure}[htb]
\centering
\begin{minipage}{.5\textwidth}
  \centering
  \includegraphics[height = 2in, width=.67\linewidth]{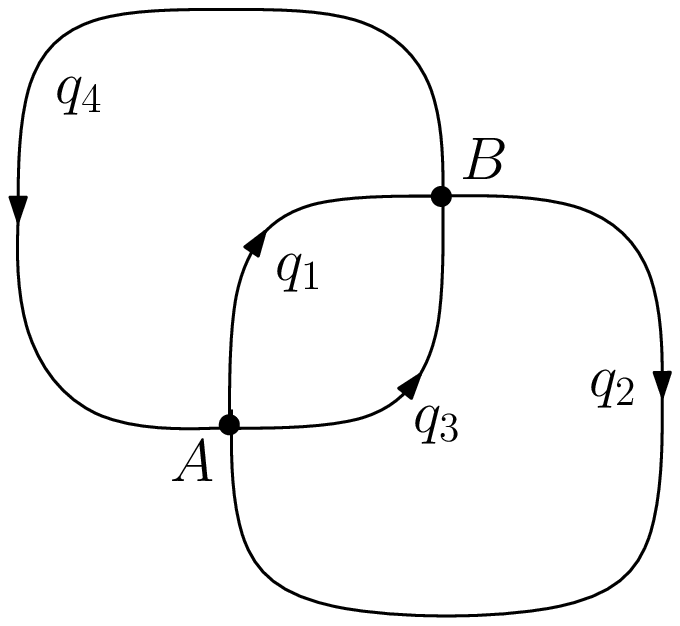}
  \caption{The graph $\Gamma$}
  \label{fig:the-graph}
\end{minipage}%
\begin{minipage}{.5\textwidth}
  \centering
  \includegraphics[height = 2in, width=.67\linewidth]{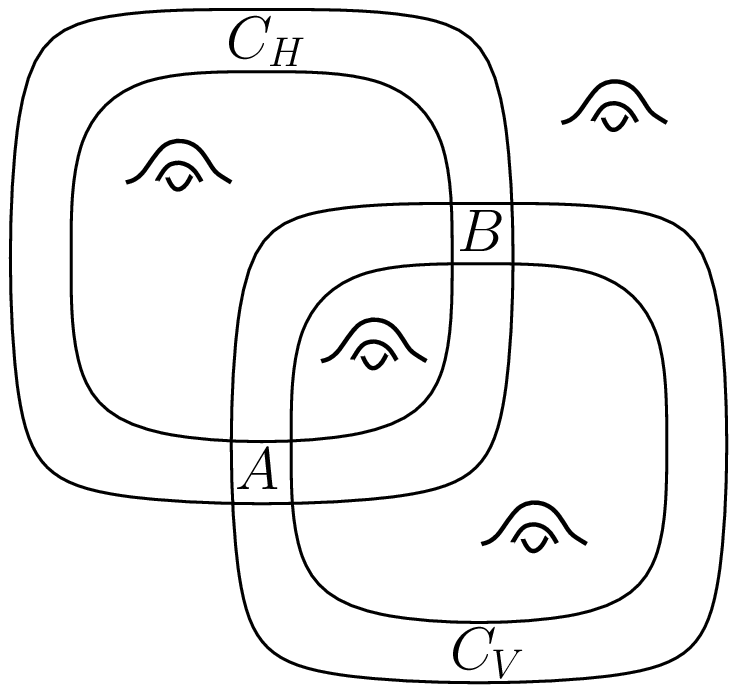}
  \caption{$C$ embedded in $\Sigma$}
  \label{fig:two-annuli}
\end{minipage}
\end{figure}

This embedding is {\em incompressible}, that is it induces an injective map on $\pi_1.$ Moreover, by considering geodesic representatives of free homotopy classes of loops in $\Sigma$ with respect to a well-chosen hyperbolic metric, which renders the components of the boundary of $\iota(C)$ geodesic, one sees that this embedding also induces an injection $\pi_0(\cL C) \to \pi_0(\cL \Sigma).$

We continue describe the general form of the class $\al_\la$ in the image of the injection $\pi_0(\cL C) \to \pi_0(\cL \Sigma).$ It is sufficient to describe its preimage in $\pi_0(\cL C) \cong \pi_1(C)/\conjugation.$

Note that $\pi_1(C) \cong \pi_1(\Gamma),$ for a graph $\Gamma$ underlying the oriented graph consisting of two vertices $A,B$ corresponding to the squares along which the images of $C_V$ and $C_H$ in $C$ intersect (we denote theses squares by $A, B$ as well), and four edges: there are two edges $q_1,q_3$ from $A$ to $B,$ and two edges $q_2,q_4$ from $B$ to $A$ - see Figure \ref{fig:the-graph}. We find it useful to orient these edges so that the concatenation $\oa=q_1 \# q_2$ corresponds to a generator of $\pi_1(C_V)$ based at a point in $A,$ and $\ob=q_3 \# q_4$ corresponds to a generator of $\pi_1(C_H)$ based at a point in $A.$ Specifically, we represent these two generators by the loops \[\gamma(a)= \{(0,[t]) \in C_V\} _{t \in \R/ L\Z}\] and \[\gamma(b)= \{(0,[t]) \in C_H\} _{t \in \R/ L\Z}\] based at the center \[0_A:=(0,[0]) \in A\] of $A.$ Put also $\oc= q_3 \# q_2.$ Since contracting the edge $q_2$ establishes a homotopy equivalence of $\Gamma$ with a bouquet $V_3$ of $3$ circles, with $\oa,\ob,\oc$ mapping to three generators of $\pi_1(V_3,[d])$ under the quotient map, we see that \[\pi_1(\Gamma,A) \cong \mathrm{Free}\langle \oa,\ob,\oc \rangle,\] the free group on $3$ generators.

The general form of the class we consider is \[\alpha = \oa^{m_1} \ob^{n_1} \oa^{m_2} \ob^{n_2} \cdot \ldots \cdot \oa^{m_p} \ob^{n_p} /\conjugation,\] for $m_1,n_1,m_2,n_2,\ldots, m_p, n_p \in \Z_{>0},$ to be specified later, such that for each $i<j,$ $(m_i,n_i) \neq (m_j,n_j),$ as elements in $\Z_{>0} \times \Z_{>0}.$ Put $\til{\alpha} = \oa^{m_1} \ob^{n_1} \oa^{m_2} \ob^{n_2} \cdot \ldots \cdot \oa^{m_p} \ob^{n_p}.$

\subsection{Constructing the egg-beater map}
We proceed to describe the dynamical system of this example. Fix a large parameter $\la \gg 1.$ We later constrain it to an unbounded increasing subsequence of $\R.$ We first construct the diffeomorphism $\phi_\lambda \in \Ham(\Sigma,\sigma),$ and based on its properties we choose an appropriate class $\al_\la$ in the image of the injection $\pi_0(\cL C) \to \pi_0(\cL \Sigma),$ and the constraints on $\la.$ 

\medskip

{\bf Notation:} For brevity we denote by $\varepsilon_s:=\sign(s) \in \{\pm 1\}$ the sign of $s \neq 0.$ For $s=0$ we define $\varepsilon_0 := 0.$
\medskip

We construct the Hamiltonian diffeomorphisms $\phi_\lambda$ as a small smoothing of a piecewise smooth homeomorphism. Consider the function $u_0:[-1,1] \to \R,$ given by \[u(s) = 1 - |s| = 1 - \varepsilon_s \cdot s.\]

\begin{figure}[htb]
\centerline{\includegraphics[height=1.5in]{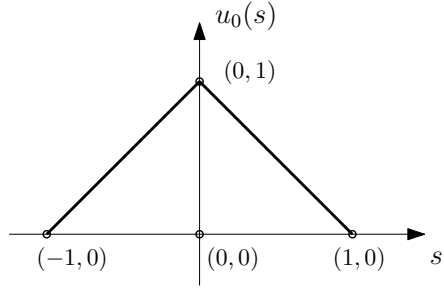}}
\caption{\label{fig:braid-closure1} Graph of $u_0$.}\label{Figure: graph of u_0}
\end{figure}

On the cylinder $C_*=[-1,1] \times \R/L\Z,$ consider the homeomorphism \[f_0=f_{0,\la}:C_* \to C_*,\]
\[f_0(x,[y]) = (x,[y+\la u_0(x)]).\]

For a smoothing $u \in C^\infty ([-1,1],\R)$ of $u_0$ with $\mathrm{supp}(u) \subset (-1,1),$ that is sufficiently $C^0$-close to $u_0$ and coincides with $u_0$ outside a sufficiently small neighbourhood of $U$ of $\{-1,0,1\},$ define the Hamiltonian diffeomorphism \[f=f_{\la}:C_* \to C_*,\]
\[f(x,[y]) = (x,[y+\la u(x)]).\]

It is easy to see by the continuous dependence of solutions to ODE in a fixed time interval that for purposes of computing fixed points of the diffeomorphisms constructed from $f$ below in the free homotopy classes of loops described below, one can effectively substitute $u$ by $u_0.$

It is easy to achieve, in addition, that $u$ is non-negative, even, and moreover such that \[\int_{-1}^{s}(u(\sigma)-u_0(\sigma)) d\sigma\] is supported in the small neighbourhood $U$ of $\{-1,0,1\}$ outside which $u$ coincides with $u_0$ - see e.g. Figure \ref{fig:graph-u-near-0-1}.

\begin{figure}[htb]
\centerline{\includegraphics[height=1.5in]{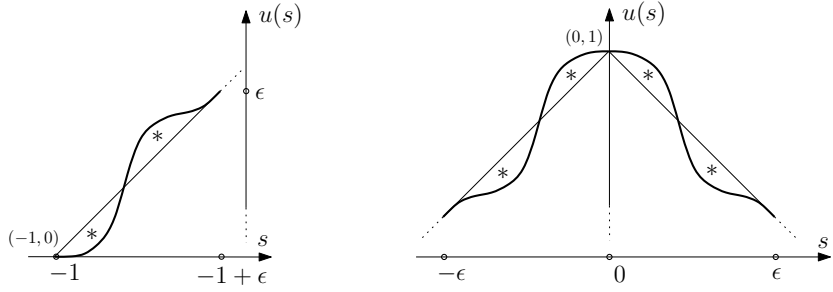}}
\caption{\label{fig:graph-u-near-0-1} Graph of $u$ near $-1$ (left) and $0$ (right). The areas marked $*$ are equal.}
\end{figure}

Note that $f^{-1}: C_* \to C_*$ is given by $f^{-1}(x,[y]) = (x,[y-\la u(x)]),$ and that $f$ is supported away from the boundary of $C_*.$ Moreover $f$ is a time-one map of an autonomous Hamiltonian flow $\{f^t\}_{t \in [0,1]}$ with Hamiltonian given by \[h(x,[y]) = h(x) = -\frac{1}{2} + \int_{-1}^{x}u(\sigma) d\sigma.\] We note that by our choice of smoothing $u,$ the function $h:[-1,1] \to \R$ is odd, and for all $s \in [-1,1] \setminus U,$ we have  \[h(s)=h_0(s),\] where \[h_0(s) = -\frac{1}{2} + \int_{-1}^{s}u_0(\sigma) d\sigma = s - \varepsilon_s \frac{s^2}{2}\] would be the Hamiltonian were we to consider the piecewise smooth homeomorphism $f_0$ as a Hamiltonian transformation. Note that $h_0$ is an odd function, and hence $h_0(s) = \vareps{s} h_0(|s|).$ 

The Hamiltonian symplectomorphism $f$ defines two Hamiltonian symplectomorphisms of $C,$ one supported on $C_V$ and one supported on $C_H,$ and both supported away from the boundary of $C,$ and hence by extension by $\id$ - two Hamiltonian symplectomorphisms of $\Sigma$ supported on $C,$ $f_V$ and $f_H$. Let us emphasize also that the Hamiltonian $h$ takes different values on the boundary components
of the annulus $C$. Nevertheless, it extends from $C_V$ and from $C_H$ to a smooth Hamiltonian on $\Sigma$
since these annuli separate $\Sigma$. In particular, $f_V$ and $f_H$ are genuine Hamiltonian diffeomorphisms of $\Sigma$.


For our fixed $\lambda \in \R,$ we define \[\phi_\lambda := f_{\la,V}\circ f_{\la,H}.\] Let $\{f^t_V\},\{f^t_H\}$ denote the Hamiltonian isotopies of $\Sigma$ (or interchangeably of $C$) with endpoints $f_V, f_H$ generated by the corresponding normalized autonomous Hamiltonians. Hence these are simply two copies of $\{f^t\}.$ We consider the Hamiltonian isotopy \[\{f^t_V\} \# \{f^t_H f_V \}\] generating $\phi_\la$ and the Hamiltonian isotopy \[\{f^t_V\} \# \{f^t_H f_V \} \# \ldots \# \{f^t_V (f_H f_V)^{p-1} (z)\} \#  \{f^t_H f_V (f_H f_V)^{p-1} \}\] generating \[(\phi_\la)^p = (f_H f_V)^p.\]

\subsection{Detecting periodic points}
We study the fixed points $z$ of $(\phi_\la)^p$ in the class $\al \in \pi_0(\cL \Sigma).$ In particular this means that \[ (\phi_\la)^p z = z\] and the free homotopy class of the orbit \[\gamma(z):=\{f^t_V(z)\} \# \{f^t_H f_V (z)\} \# \ldots \#  \{f^t_V (f_H f_V)^{p-1} (z)\} \# \{f^t_H f_V (f_H f_V)^{p-1} (z)\}\] satisfies \[[\gamma(z)] = \al.\]

\medskip

{\bf Terminology:} In what follows we refer to \[\{f^t_V(z)\}, \{f^t_H f_V (z)\},\ldots,  \{f^t_V (f_H f_V)^{p-1} (z)\}, \{f^t_H f_V (f_H f_V)^{p-1} (z)\}\] as the {\em intermediate paths} of the orbit $\gamma(z),$ and to their endpoints as the {\em intermediate points} of this orbit.
\medskip

We start solving the system of equations \begin{align}\label{Equation: fixed points in alpha 1}
&(\phi_\la)^p z =z, \\ \label{Equation: fixed points in alpha 2} &[\gamma(z)]_{\pi_0(\cL \Sigma)} = \al
\end{align}

by a topological analysis. Note that as all our Hamiltonian diffeomorphisms and flows are supported in $C$ and the embedding $\iota: C \to \Sigma$ is incompressible on both $\pi_1$ and free homotopy classes of loops, it is sufficient to restrict our consideration to $C.$ As $\pi_0(\cL C)= \pi_1(C)/\conjugation$ and as explained above $\pi_1(C) \cong \mathrm{Free}\langle \oa,\ob, \oc \rangle$ is a free group on three generators, we can reduce the second equation from $\pi_0(\cL C)$ to $\pi_1(C)$ using the following notion from combinatorial group theory.

A word in a free group on a set $X$ is {\em cyclically reduced} if all its cyclic conjugations are reduced. Equivalently, this word is reduced, and moreover if it is written cyclically, then the resulting cyclic word is also reduced.

It is a well-known result in combinatorial group theory that conjugacy classes of words in a free group $\Foabc$ are classified by cyclically reduced words in $\oa,\ob,\oc$ up to cyclic permutations (see e.g. \cite{CombinatorialGroupTheory}).

The word $\til{\alpha} = \oa^{m_1} \ob^{n_1} \oa^{m_2} \ob^{n_2} \cdot \ldots \cdot \oa^{m_p} \ob^{n_p}$ is clearly cyclically reduced, and any other word $v$ of the form $v=\oa^{k_1} \oc^{e_1} \ob^{l_1} \oc^{e_2} \ldots \oa^{k_p} \oc^{e_{2p-1}} \ob^{l_p} \oc^{e_{2p}}$ can be obtained from $\til{\al}$ by a cyclic permutation if and only if \[v = \widetilde{\al}^{(j)} = \oa^{m_{p-j+1}} \ob^{n_{p-j+1}} \ldots \oa^{m_p} \ob^{n_p} \oa^{m_1} \ob^{n_1} \ldots  \oa^{m_{p-j}} \ob^{n_{p-j}}\] for some $1 \leq j \leq p.$ Note that $\til{\al}^{(p)} = \til{\al}.$ Moreover for $k_1,l_1,\ldots, k_p,l_p \in \Z\setminus\{0\}$ the word $v$ is cyclically reduced, and is hence conjugate to $\til{\al}$ in exactly these $p$ cases.
Below we prove the following.

\begin{lma}\label{Lemma: fixed points in A}
All fixed points $z$ of $(\phi_\la) ^p$ in class $\al$ satisfy \[(\phi_\la)^j z \in A\] for all $0 \leq j \leq p.$ Moreover all the intermediate points of $\gamma(z)$ lie in $A.$
\end{lma}

Therefore the system of Equations (\ref{Equation: fixed points in alpha 1}),(\ref{Equation: fixed points in alpha 2}) reduces to the $p$ systems of equations indexed by $1 \leq j \leq p.$

\begin{align}\label{Equation: fixed points in fund class alpha 1}
&(\phi_\la)^p z =z, \\ \label{Equation: fixed points in fund class alpha 2} &[\gamma(z)]_{\pi_1(C,0_A)} = \til{\al}^{(j)},
\end{align}

The proof of Lemma \ref{Lemma: fixed points in A} also shows the following statement. \begin{lma}\label{Lemma: bijection between systems of equations}
The map $z \mapsto \phi_\la z$ establishes a bijection from the set of solutions of the system of Equations (\ref{Equation: fixed points in fund class alpha 1}),(\ref{Equation: fixed points in fund class alpha 2}) for $j$ to the set of solutions of the system of Equations (\ref{Equation: fixed points in fund class alpha 1}),(\ref{Equation: fixed points in fund class alpha 2}) for $j+1.$ The inverse bijection is given by the map $z \mapsto (\phi_\la)^{p-1} z = (\phi_\la)^{-1} z.$
\end{lma}

Therefore the set of solutions of the system of Equations (\ref{Equation: fixed points in alpha 1}),(\ref{Equation: fixed points in alpha 2}) consists of $p$-tuples \[\{z,\phi_\la z,\ldots, (\phi_\la)^{p-1}z \}\] where $z$ is a fixed point of $(\phi_\la)$ with $[\gamma(z)]_{\pi_1(C,0_A)} = \til{\al}.$ We remark that the fixed points $\{ (\phi_\la)^{j} z \}_{1 \leq j \leq p}$ have the same action and index values.

\begin{proof}[Proof of Lemma \ref{Lemma: fixed points in A}]
We note that if $[\gamma(z)]= \alpha$ then $z \in A \cup B.$ Indeed, if $z \in C_V \setminus A \cup B,$ then the path $\{f^t_V(z)\}$ is constant, in contradiction to the form of $[\gamma(z)].$ Similarly, if $z \in C_H \setminus A \cup B,$ then the path $\{f^t_H f_V (f_H f_V)^{p-1} (z)= f^{t-1}_H (f_H f_V)^{p} (z) = f^{t-1}_H (z)\}$ is constant, in contradiction to the form of $[\gamma(z)].$ Similarly, the intermediate points $$f_V(z), f_H f_V(z), f_V f_H f_V (z),\ldots, (f_H f_V)^{p-1}(z), f_V (f_H f_V)^{p-1}(z)$$ of the loop $\gamma(z)$ lie in $A \cup B.$

The following table summarizes the possible types of trajectories of $f^t_V$ and $f^t_H,$ that join $A \to A, A \to B, B \to A, B \to B.$ These correspond to morphisms in the fundamental groupoid of the graph $\Gamma.$ Note that we adopt a non-standard notation for composition of morphisms in the fundamental groupoid - for example we write $a= [q_1 \# q_2] = q_1 q_2$ instead of $a = q_2 q_1$ - to make the relation to geometry more intuitive. Let $m \in \Z_{\geq 0}$ be a generic notation for a non-negative integer, and $n \in \Z_{\geq 1}$ be a generic notation for a positive integer. The conditions $m \geq 0$ and $n-1 \geq 0$ in this table are consequences of the fact that $u$ is a non-negative function.

\renewcommand{\arraystretch}{1.5}

\begin{center}
    \begin{tabular}{| c | c | c | c |c|}
    \hline
     & $A \to A$ & $A \to B$ & $B \to A$ & $B \to B$ \\ \hline
    $f_V$ & $\oa^m$ & $\oa^{n-1} q_1$ & $q_2\oa^{n-1}$ & $q_2\oa^{n-1} q_1$ \\ \hline
    $f_H$ &  $\ob^m$ & $\ob^{n-1} q_3$ & $q_4\ob^{n-1}$ & $q_4 \ob^{n-1} q_3$\\ \hline

    \end{tabular}
\end{center}

This table together with relations defining $\oa,\ob,\oc$ together with $q_1 q_4= \oa \oc^{-1} \ob$ and the fact that $[\gamma(z)]$ is represented by a {\em composable} path of morphisms in the fundamental groupoid of the graph $\Gamma,$ shows that $[\gamma(z)]_{\pi_0(\cL C)}$ is of the form \[v=\oa^{k_1} \oc^{e_1} \ob^{l_1} \oc^{e_2} \oa^{k_2} \oc^{e_3} \ob^{l_2} \ldots \oa^{k_p} \oc^{e_{2p-1}} \ob^{l_p} /\conjugation,\] where $k_1,l_1,\ldots, k_p,l_p \in \Z_{\geq 0}.$  Hence if $[\gamma(z)]_{\pi_0(\cL C)} = \alpha,$ then in fact $k_1,l_1,\ldots, k_p,l_p \in \Z_{> 0},$ whence $v$ is cyclically reduced, $z \in A$ and by the properties of cycliclally reduced words, $[\gamma(z)]_{\pi_1(C) = \pi_1(\Gamma,A)} = \til{\al}^{(j)},$ for some $1 \leq j \leq p.$ Moreover, the intermediate points \[f_V(z), f_H f_V(z), f_V f_H f_V (z),\ldots , f_V (f_H f_V)^{p-1} (z)\] lie in $A.$
\end{proof}

It remains to study the solutions to the system of Equations (\ref{Equation: fixed points in fund class alpha 1}),(\ref{Equation: fixed points in fund class alpha 2}) for $j=p$ with the help of Lemma \ref{Lemma: fixed points in A}. Lemma \ref{Lemma: fixed points in A} and its proof imply that for any such solution $z$ the classes of the intermediate paths of $\gamma(z)$ in $\pi_1(C,A)$ are well-defined and in fact satisfy \begin{gather} [\{f^t_V(z)\}] = a^{m_1}, [\{f^t_H f_V (z)\}] = b^{n_1}, \nonumber \\  [\{f^t_V f_H f_V (z)\}] = a^{m_2}, [\{f^t_H f_V f_H f_V (z)\}] =b^{n_2}, \nonumber \\ \dots \nonumber \\ [\{f^t_V (f_H f_V)^{p-1} (z)\}] = a^{m_p}, [\{f^t_H f_V (f_H f_V)^{p-1}  (z)\}] = b^{m_p} \label{Equation: intermediate points} \end{gather}

Knowing this, we pass to the universal cover of $C_*$ and translate the sequence \[z \mapsto f_V(z) \mapsto f_H f_V (z) \mapsto f_V f_H f_V (z) \mapsto \ldots \mapsto f_V (f_H f_V)^{p-1} (z) \mapsto z\] of intermediate points in terms of explicit points in $(-1,1) \times (-1,1).$

To this end, for $m \in \Z,$ denote by \[r=r_m:(-1,1) \times \R \to (-1,1) \times \R,\] \[r(x,y) = (x, y- m\cdot L).\] the {\em reduction map}. Note that the images of $(x,y)$ and $r(x,y)$ in $C_* = [-1,1] \times \R/L \Z$ are equal, and that $r$ maps $(-1,1) \times (m\cdot L-1,m\cdot L+1)$ isomorphically to $(-1,1) \times (-1,1).$ Call $VH:S_{0,V} \to S_{0,H}, (x,[y]) \mapsto (-y,[x])$ and $HV:S_{0,H} \to S_{0,V}, (x,[y]) \mapsto (y,[-x])$ the {\em flip maps}. Using these maps we decode the equation $(\phi_\la)^p z = z$ for $z = (x_0,y_0),$ in the class $\til{\al},$  as \[(x_{2p},y_{2p}) = (x_0,y_0),\] where \[(x_1,y_1) = VH \circ r_{m_1} \circ \til{f}(x_0,y_0),\] \[(x_2,y_2) =HV \circ r_{n_1} \circ \til{f}  (x_1,y_1) = \]\[= HV \circ r_{n_1} \circ \til{f} \circ VH \circ r_{m_1} \circ \til{f}(x_0,y_0),\] and in general for $0 \leq j \leq p-1$ \[(x_{2j+2},y_{2j+2}) =HV \circ r_{n_{j+1}} \circ \til{f}  (x_{2j+1},y_{2j+1}) = \]\[= HV \circ r_{n_{j+1}} \circ \til{f} \circ VH \circ r_{m_{j+1}} \circ \til{f}(x_{2j},y_{2j}).\]



See the diagram in Figure \ref{Figure: diagramp} and an illustration in Figure \ref{Figure: decoding}.  Here $\widetilde{f}$ is the lift of $f$ to the universal cover $\til{C}_*  \cong [-1,1] \times \R,$ given by the Hamiltonian flow $\{f^t\}$ of the Hamiltonian $h.$

The sequence of points \begin{equation}\label{Equation: intermediate points 2}
(x_0,y_0) \mapsto (x_1, y_1) \mapsto ... \mapsto (x_{2p-1},y_{2p-1}) \mapsto (x_{2p},y_{2p}) = (x_0, y_0)
\end{equation} corresponds the sequence of intermediate points (\ref{Equation: intermediate points}).

\begin{figure}[htb]
\centerline{\includegraphics[height=3in]{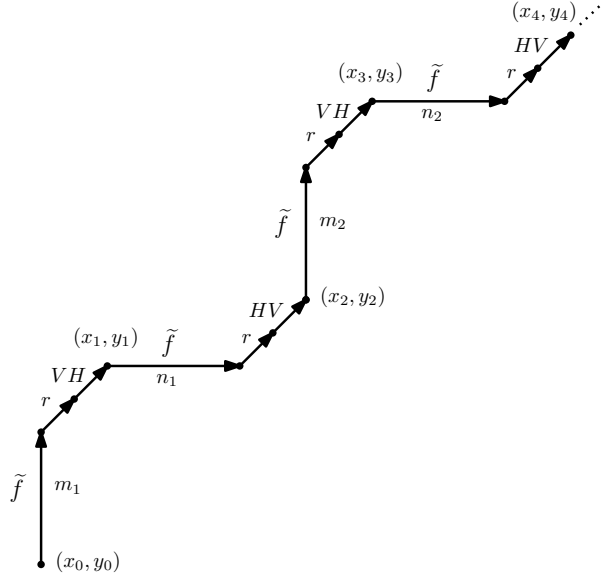}}
\caption{\label{fig:braid-closure1} Diagram for the equation $(\phi_{\la})^{p} (x_0,y_0) = (x_0,y_0)$.}\label{Figure: diagramp}
\end{figure}

\begin{figure}[htb]
\centerline{\includegraphics[height=3in]{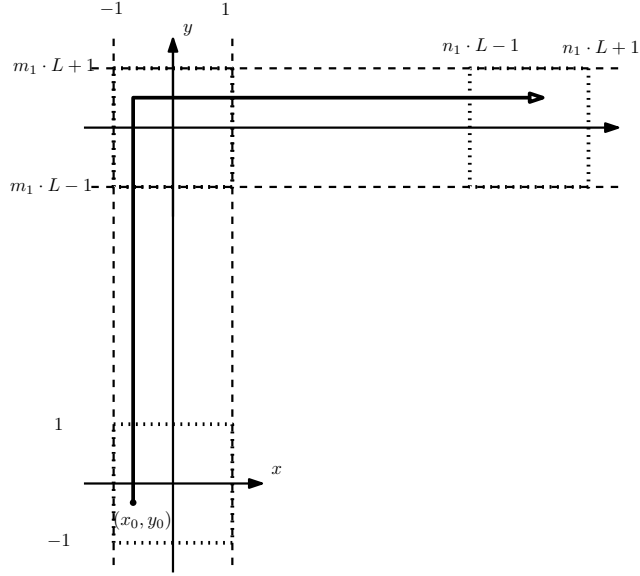}}
\caption{Decoding the equation via diagram.}\label{Figure: decoding}
\end{figure}

The equation $(x_{2p},y_{2p}) = (x_0,y_0)$ gives the following system of equations (\ref{Equation:1 for fixed point x}),(\ref{Equation:2 for fixed point x}). 

For $1 \leq j \leq p,$ set \[n_j = \frac{\nu_1}{L} \la, m_j = \frac{\mu_1}{L} \la.\]

We will choose certain $\{\nu_j,\mu_j\in (0,1)\},$ with $\frac{\kappa}{\kappa'} \in \Q$ for every $\kappa,\kappa' \in \{\nu_j,\mu_j\}$ and fix them. We consider $\la$ with the condition that $n_j = \frac{\nu_j}{L} \la, m_j = \frac{\mu_j}{L} \la \in \Z_{>0}$ for all $1\leq j \leq p.$ By construction there are arbitrarily large $\la$ with this property, and each such choice of $\la$ determines a primitive free homotopy class $\al_\la.$

Consider the composition \[HV \circ r_{\nu\la} \circ \til{f} \circ VH \circ r_{\mu\la} \circ \til{f}.\] Whenever $(x,y)$ is such that \begin{gather}
\til{f}(x,y) \in (-1,1) \times (\mu \la-1,\mu \la+1), \label{Equation: condition 1}\\
\til{f}(VH \circ r_{\mu\la} \circ \til{f}(x,y)) \in (-1,1) \times (\nu \la-1,\nu \la + 1), \label{Equation: condition 2}
\end{gather} the image $HV \circ r_{\nu\la} \circ \til{f} \circ VH \circ r_{\mu\la} \circ \til{f}(x,y) \in (-1,1)^{2}$ of $(x,y)$ under this composition is given by the formula \[\Phi^{\mu,\nu}_\la(x,y) = (x + \la u(y + \la u(x) - \mu \la) - \nu \la,\, y + \la u(x) - \mu \la).\] Moreover, whenever $\mu\la,\nu\la \in \Z,$ \[pr_{C_V}(HV \circ r_{\nu\la} \circ \til{f} \circ VH \circ r_{\mu\la} \circ \til{f}(x,y))= \phi_\la(x,y)\] where $pr_{C_V}:\til{C}_V \to C_V$ is the natural projection. Clearly $pr_{C_V}$ maps $(-1,1)^2$ isomorphically onto $A.$

Note that by Equation \eqref{Equation: intermediate points} the intermediate points of the trajectory of a solution in class $\al$ satisfy these conditions. Hence the even itermediate points satisfy the system of equations:
\begin{align}\label{Equation:1 for fixed point x}
(x_{2j+2},y_{2j+2}) &= \Phi^{\mu_{j+1},\nu_{j+1}}_\la (x_{2j},y_{2j}),\; 0\leq j \leq p-1\\
\label{Equation:2 for fixed point x} (x_{2p},y_{2p}) &= (x_0,y_0).
\end{align}

In fact $j$ can be considered modulo $p$ in these equations, emphasizing their cyclic nature. Then the equation for $(x_0,y_0)$  can be written as

\begin{equation}\label{Equation: for x0 y0}
\Phi_\la^{\mu_p,\nu_p} \circ \ldots \circ \Phi_\la^{\mu_1,\nu_1} (x_0,y_0) = (x_0,y_0),
\end{equation}

and for $(x_{2j_0},y_{2j_0})$ in general, as

\begin{equation}\label{Equation: for x0 y0 with index j_0}
\Phi_\la^{\mu_{j_0+p},\nu_{j_0+p}} \circ \ldots \circ \Phi_\la^{\mu_{j_0},\nu_{j_0}} (x_{2j_0},y_{2j_0}) = (x_{2j_0},y_{2j_0}),
\end{equation}

We remark that considering only the even intermediate points we do not lose data. Indeed any odd intermediate point $(x_{2j+1},y_{2j+1}),\; 0 \leq j \leq p-1$ satisfies \begin{equation}\label{Equation: odd via even} (x_{2j+1},y_{2j+1}) = (-y_{2j+2},x_{2j}).\end{equation}

We claim that for $\la \gg 1,$ given any $\overrightarrow{\varepsilon} = (\vareps{1},\vareps{2},\ldots,\vareps{2p}) \in \{\pm 1\}^{2p},$ there exists a unique solution \[\{(x_{2j},y_{2j})\}_{j=0}^{p-1} = (z_{\overrightarrow{\varepsilon}}, (\phi_\la) z_{\overrightarrow{\varepsilon}},\ldots,  (\phi_\la)^{p-1} z_{\overrightarrow{\varepsilon}} )\] of the system of equations (\ref{Equation:1 for fixed point x}),(\ref{Equation:2 for fixed point x}) such that \[\{\displaystyle{\varepsilon_{x_{2j}},\varepsilon_{y_{2j}}}\}_{j=0}^{p-1} = \overrightarrow{\varepsilon},\] that is for each $0 \leq j \leq p-1,$ \[(\varepsilon_{x_{2j}},\varepsilon_{y_{2j}}) = (\vareps{2j+1},\vareps{2j+2}).\]

Assume $\{(x_{2j},y_{2j})\}_{j=0}^{p-1}$ is a solution with sign vector $\overrightarrow{\varepsilon}.$ Then using the definition $u_0(s) = 1 -\vareps{s} s,$ and the remark that all the solutions in class $\til{\al}$ must have coordinates away from the subset of $(-1,1)$ where $u$ differs from $u_0,$ we simplify Equation \eqref{Equation: for x0 y0} to the following linear equation in two variables $(x_0,y_0).$

Note that \[\Phi_\la^{\mu_{j+1},\nu_{j+1}}(x_{2j},y_{2j}) = A_{\la,j+1} (x_{2j},y_{2j}) + b_{\la,j+1},\] where for all $0 \leq j \leq p-1,$ \begin{equation}\label{Equation: def of A}  A_{\la,j+1} = \begin{pmatrix} 1 + \vareps{2j+4}\vareps{2j+1} \la^2 & -\vareps{2j+4}\la \\ -\vareps{2j+1}\la & 1 \end{pmatrix}\end{equation} and \begin{equation}\label{Equation: def of b} b_{\la,j+1} = (-\vareps{2j+4}(1-\mu_{j+1})\la^2 + (1-\nu_{j+1})\la, (1-\mu_{j+1}) \la).\end{equation}

Note that $\det( A_{\la,j+1}) = 1.$ For future use in Section \ref{Section: discussion}, we remark that in fact $A_{\la,j+1}$ is a product of $2$ parabolic matrices as follows: \begin{equation} \label{Equation: 2 parabolic matrices} A_{\la,j+1} = \begin{pmatrix} 1  & -\vareps{2j+4}\la \\ 0 & 1 \end{pmatrix} \begin{pmatrix} 1  & 0 \\ -\vareps{2j+1}\la & 1 \end{pmatrix}.\end{equation} We also record that \begin{equation}\label{Equation: Ak inverse on bk} (A_{\la,j+1})^{-1} b_{\la,{j+1}} = ((1-\nu_{j+1}) \la, (1-\mu_{j+1})\la + \vareps{2j+1}(1-\nu_{j+1}) \la^2),\end{equation} and in particular its coordinates are polynomials of degree at most $2$ in $\la.$

Equation \eqref{Equation: for x0 y0} simplifies to:
\begin{align}\label{Equation: for x0 y0 specific 1}
(A_{\la,p} \circ \ldots \circ A_{\la,1} - \id) (x_0,y_0) = - v_{\la,0}, \\ \label{Equation: for x0 y0 specific 2}
v_{\la,0} = \sum_{j=0}^{p-2} A_{\la,p} \circ \ldots \circ A_{\la,j+2} (b_{\la,j+1}) + b_{\la,p}.
\end{align}

We use the convention that the last summand $b_{\la,p}$ corresponds to the index $j=p-1.$ We claim that for $\la \gg 1,$ this equation is non-singular, and hence has a unique solution $z = z_{\overrightarrow{\varepsilon}} =(x_0, y_0).$ To it there corresponds a unique solution $\phi_\la^{j_0} (z_{\overrightarrow{\varepsilon}})$ of Equation \eqref{Equation: for x0 y0 with index j_0} with index $j_0.$ We require the following observation on the product matrices, similar to the ones involved in \eqref{Equation: for x0 y0 specific 1},\eqref{Equation: for x0 y0 specific 2}. Denote for $j \geq i,$ \[\overline{A}_{\la,j,i} = A_{\la,j} \circ \ldots \circ A_{\la,i},\] and \[\overline{A}_\la = \overline{A}_{\la,p,1} = A_{\la,p} \circ \ldots \circ A_{\la,1}.\]

\vspace{1.5 pt}
\begin{lma}\label{Equation: product matrix asymptotics}
Put $\overline{\varepsilon}_{j,i} = \prod_{k=i-1}^{j-1} \varepsilon_{2k+1}\varepsilon_{2k+4},$ where $\vareps{l}$ is considered with $l$ modulo $2p.$ Note that $\overline{\varepsilon}:= \overline{\varepsilon}_{p,1} =  \prod_{k=1}^{2p}\vareps{k}.$ Then for all $1 \leq i\leq j \leq p,$ \[ \overline{A}_{\la,j,i} = \begin{pmatrix*}[l] \overline{\varepsilon}_{j,i} \la^{2(j-i+1)} + \ldots & - \overline{\varepsilon}_{j,i}\vareps{2i-1} \la^{2(j-i+1)-1} + \ldots \\ - \overline{\varepsilon}_{j,i} \vareps{2j+2}\la^{2(j-i+1)-1} + \ldots & \overline{\varepsilon}_{j,i} \vareps{2i-1}\vareps{2j+2} \la^{2(j-i+1)-2} + \ldots \end{pmatrix*}, \] where the dots denote lower order terms in $\la.$ In particular \[\overline{A}_\la = \begin{pmatrix} \overline{\varepsilon} \la^{2p} + \ldots & - \overline{\varepsilon}\vareps{1} \la^{2p-1} + \ldots \\ -\overline{\varepsilon} \vareps{2}\la^{2p-1} + \ldots &  \overline{\varepsilon} \vareps{1}\vareps{2} \la^{2p-2} + \ldots \end{pmatrix}.\]
\end{lma}

The proof of this lemma is a straightforward induction. Now, since for every $2 \times 2$ matrix $A$ with $\det A = 1,$ one has $\det(A -\id) = 2 - \trace(A),$ we compute \[\det(\overline{A}_\la - \id) = 2 - \trace(\overline{A}_\la) = -\overline{\varepsilon} \la^{2p} + \ldots,\] where the dots denote lower order terms in $\la.$ In particular for $\la \gg 1,$ \[\det(\overline{A}_\la - \id) \neq 0,\] concluding the proof of the claim.

We now claim that the combined solution $\{(x_{2j},y_{2j})\}_{j=0}^{p-1}$ can moreover be shown directly to satisfy the following asymptotics as $\la \to \infty.$ For each $0 \leq j \leq p-1,$
\begin{equation}\label{Equation: asymptotics of solutions.}
(x_{2j},y_{2j}) = (\vareps{2j+1}(1-\mu_j),\vareps{2j+2}(1-\nu_{j-1})) + O(\frac{1}{\la}).
\end{equation}
In particular, $z = (x_0,y_0)$ lies in $(-1,1)^2.$ Moreover, by these asymptotics and Equation \eqref{Equation: odd via even}, for each $0\leq j \leq p-1,$ $(x_{2j},y_{2j})$ satisfies Conditions \eqref{Equation: condition 1},\eqref{Equation: condition 2} with respect to $(\mu_{j+1},\nu_{j+1}).$ Therefore $z$ is a solution of the system \eqref{Equation: fixed points in fund class alpha 1},\eqref{Equation: fixed points in fund class alpha 2} (for $j=p$) and its intermediate points are given by $\{(x_k,y_k)\}_{k=0}^{2p-1}$ which indeed lie in $(-1,1)^2.$ This together with Lemma \ref{Lemma: bijection between systems of equations} finishes the first, existence and uniqueness, part of Proposition \ref{Proposition: example}.

We proceed with the proof of these asymptotics. Our convention is that $\mu_j,\nu_j$ are considered with $j$ modulo $p,$ and $\vareps{2j+1},\vareps{2j+2}$ are considered modulo $2p.$ Hence for $j=0$ the required asymptotics are: \[(x_{0},y_{0}) = (\vareps{1}(1-\mu_1),\vareps{2}(1-\nu_{p})) + O(\frac{1}{\la}).\]

Note that for a $2 \times 2$ matrix $A$ with $\det(A) = 1$ and with $(A-1)$ invertible, one has the identity $(A-1)^{-1} = \frac{1}{\det(A-1)} (A^{-1} - 1).$ Hence by Equation \eqref{Equation: for x0 y0} \[-z = \frac{1}{\det(\overline{A}_\la-1)} ((A_{\la,1})^{-1} \circ \ldots \circ (A_{\la,p})^{-1} - 1) v_{\la,0} .\]

For $0 < j < p-1,$ the summand \[\overline{A}_{\la, p,{j+2}} (b_{\la,{j+1}})\] in $v_{\la,0}$ contributes the following summand in the expression for $(-z):$ \[ \frac{1}{\det(\overline{A}_\la-1)} ((\overline{A}_{\la,j,1})^{-1} \circ (A_{\la,j+1})^{-1}(b_{\la,j+1}) -  \overline{A}_{\la,p,j+2} (b_{\la,j+1})).\]

Since $\det(\overline{A}_\la-1)$ is a polynomial of degree $2p$ in $\la$ with non-trivial leading coefficient, the coordinates of $(\overline{A}_{\la,j,1})^{-1} \circ (A_{\la,j+1})^{-1}(b_{\la,j+1})$ are, by \eqref{Equation: Ak inverse on bk}, polynomials in $\la$ of degree at most $2j+2 \leq 2(p-2) + 2 = 2p - 2,$ and the coordinates of $\overline{A}_{\la,p,j+2} (b_{\la,j+1})$ are polynomials in $\la$ of degree at most $2(p - (j+2) + 1) + 2 = 2p - 2j \leq 2p - 2,$ this summand contributes $O(\frac{1}{\la^2})$ to the expression for $(-z).$

We are left with two terms in the sum, corresponding to $j=0$ and $j=p-1.$ First consider $j=0.$ The corresponding summand in the expression for $(-z)$ is \[ \frac{1}{\det(\overline{A}_\la-1)} ((A_{\la,1})^{-1} (b_{\la,1}) -  \overline{A}_{\la,p,2} (b_{\la,1}))\] \[= - \frac{1}{\det(\overline{A}_\la-1)} \overline{A}_{\la,p,2} (b_{\la,1}) + O(\frac{1}{\la^2}) =\] by Lemma \ref{Equation: product matrix asymptotics} \[= - \frac{1}{-\overline{\varepsilon}\la^{2p} + O(\la^{2p-1})} \begin{pmatrix*}[l] \overline{\varepsilon}_{p,2} \la^{2p-2} + \ldots & - \overline{\varepsilon}_{p,2}\vareps{3} \la^{2p-3} + \ldots \\ - \overline{\varepsilon}_{p,2} \vareps{2}\la^{2p-3} + \ldots &  \overline{\varepsilon}_{p,2} \vareps{3}\vareps{2} \la^{2p-4} + \ldots \end{pmatrix*} (b_{\la,1}) + O(\frac{1}{\la^2}) =\] by Equation \eqref{Equation: def of b} \[ = \frac{1}{\overline{\varepsilon}\la^{2p} + O(\la^{2p-1})} (-\overline{\varepsilon}_{p,2} \vareps{4} (1-\mu_1)\la^{2p} + O(\la^{2p-1}), O(\la^{2p-1})) = - (\vareps{1}(1-\mu_1),0) + O(\frac{1}{\la}).\]

Now consider $j=p-1.$ The corresponding summand in the expression for $(-z)$ is \[\frac{1}{\det(\overline{A}_\la-1)} ((A_{\la,p-1,1})^{-1} \circ (A_{\la,p})^{-1}(b_{\la,p}) -  b_{\la,p}) = \]  \[ = \frac{1}{\det(\overline{A}_\la-1)} (A_{\la,p-1,1})^{-1} \circ (A_{\la,p})^{-1}(b_{\la,p})  + O(\frac{1}{\la^2}).\]

Note that by Lemma \ref{Equation: product matrix asymptotics} \[(\overline{A}_{\la,p-1,1})^{-1} = \begin{pmatrix}  \overline{\varepsilon}_{p-1,1} \vareps{1}\vareps{2p} \la^{2p-4} +\ldots  &   \overline{\varepsilon}_{p-1,1}\vareps{1} \la^{2p-3} + \ldots \\ \overline{\varepsilon}_{p-1,1} \vareps{2p}\la^{2p-3} + \ldots  & \overline{\varepsilon}_{p-1,1} \la^{2p-2} + \ldots \end{pmatrix},\]

and by \eqref{Equation: Ak inverse on bk} \[(A_{\la,p})^{-1} b_{\la,p} = ((1-\nu_{p}) \la, (1-\mu_{p})\la + \vareps{2p-1}(1-\nu_{p}) \la^2).\]

Hence \[ \frac{1}{\det(\overline{A}_\la-1)} (A_{\la,p-1,1})^{-1} \circ (A_{\la,p})^{-1}(b_{\la,p})  + O(\frac{1}{\la^2}) = \] \[ = \frac{1}{-\overline{\varepsilon}\la^{2p} + O(\la^{2p-1})} (O(\la^{2p-1}), \overline{\varepsilon}_{p-1,1} \vareps{2p-1}(1-\nu_p) \la^{2p} + O(\la^{2p-1})) = - (0, \vareps{2} (1 - \nu_p)) + O(\frac{1}{\la}).\]

In conclusion, we have obtained the asymptotics $z = (\vareps{1} (1 - \mu_1),\vareps{2} (1 - \nu_p)) + O(\frac{1}{\la}),$ as required. The asymptotics for $(\phi_\la)^{j_0} z$ are obtained similarly, starting with Equation \eqref{Equation: for x0 y0 with index j_0} with index $j_0.$ This finishes the proof of the claim.

It is easy to see that these fixed points of $(\phi_\la)^p$ are non-degenerate. Indeed, given a solution $z_{\overrightarrow{\varepsilon}} = (x_0,y_0),$ we compute \[(D (\phi_\la)^p)_{(x_0,y_0)} -\id = \overline{A}_\la - \id,\] which we have previously shown to be non-singular for $\la \gg 1.$

\subsection{The end of the proof of Proposition \ref{Proposition: example}}
It remains to calculate the action gaps for the periodic orbits found above. 
Choose \[\eta_{\al_\la}:=\ga(\oa)^{\# m_1}\# \ga(\ob)^{\# n_1}\# \ldots \# \ga(\oa)^{\# m_p}\# \ga(\ob)^{\# n_p}\] as the reference loop in the class $\al_\la.$ Then for $\overrightarrow{\varepsilon} \in \{\pm 1\}^{2p}$ we calculate that the action of the orbit corresponding to $z_{\overrightarrow{\varepsilon}}$ is \begin{equation}\label{Equation: total action.}
\A(z_{\overrightarrow{\varepsilon}}) = \frac{\la}{2}\sum_{j=0}^{p-1} (\vareps{2j+1} (1-\mu_{j+1})^2 - \vareps{2j+4}(1-\nu_{j+1})^2) + O(1),\end{equation} as $\la \to \infty.$

Indeed each of the $2p$ terms corresponds to an intermediate path, and for example, using the asymptotics \eqref{Equation: asymptotics of solutions.} and the definition of $h_0(s),$ the first term is \[\A_1 = \int_0^1 \la h(f^t_V(x_0,y_0))\,dt - \int_{\displaystyle{\{{ f^t_V}(x_0,y_0)\}_{t \in [0,1]}}} \,xdy= \] \[ = \la h(x_0) -\la \mu_1 x_0 +O(1) = \vareps{x_0} (\la h(|x_0|) -\la \mu_1 |x_0|) + O(1)=\] \[= \la \vareps{x_0} ((1-\mu_1) - \frac{(1-\mu_1)^2}{2} - \mu_1 (1-\mu_1)) + O(1) = \frac{\la}{2}\vareps{x_0} (1-\mu_1)^2 + O(1).\]
Recall that in our conventions $\vareps{x_0} = \vareps{1},$ hence the term we just computed corresponds to the first summand for $j=0.$
There exist $\{\mu_j,\nu_j \in (0,1)\}_{1\leq j \leq p}$ with all pairwise ratios in $\Q,$ such that the coefficients \[\{\sum_{j=0}^{p-1} (\vareps{2j+1} (1-\mu_{j+1})^2 - \vareps{2j+4}(1-\nu_{j+1})^2)\}_{\overrightarrow{\varepsilon} \in \{\pm 1\}^{2p}}\] of $\frac{\la}{2}$ in Equation \eqref{Equation: total action.} are all different, and hence the action differences between the different $p$-tuples of fixed points of $(\phi_\la)^p$ in the free homotopy class $\al_\la$ grow linearly. This finishes the proof of Proposition \ref{Proposition: example}.

\qed

\section{Alternative approach in dimension $2$}\label{sec-2D}

For Theorem \ref{thm-main} above there is an alternative proof in dimension $2,$ that is when $M$ is a point. We present it here and note that it relies very strongly on two-dimensional methods.

\begin{proof}(A $2d$ proof of Theorem \ref{thm-main} in $2d$)

As in the proof of Propostion \ref{Proposition: example} we fix a large $\la \gg 1,$ and choose an appropriate small smoothing $u$ of $u_0(s) = 1 - |s|,$ and construct $\phi_\la$ as the composition $f_H f_V$ of diffeomorphisms of two annuli embedded in $\Sigma$ as in Figure \ref{fig:two-annuli}, where in each annulus (isomorphic to $C_* = [-1,1] \times \R/ L\R,$ for $L \geq 4$) the diffeomorphism is given by $f(x,[y]) = (x, y + \la u(x)).$ Note that $f$ is the time-one map of the autonomous Hamiltonian isotopy $f^t(x,[y]) = (x, [y + t \cdot \la u(x)]).$

For a free homotopy class $\beta$ of loops in a surface $\Sigma,$ denote by \[\mathrm{si}(\beta) = \min \#\{\mathrm{double\; points\; of} \; b\},\] where the minimum is taken over all immersed loops $b: S^1 \to \Sigma$ in general position with $[b]=\beta.$ This is the geometric self-intersection number of $\beta.$ In particular, $\beta$ is represented by a simple closed curve if and only if $\mathrm{si}(\beta)=0.$ Recall that by Constraint \ref{constraint-1} from Section \ref{subsec-constraints} a Hamiltonian diffeomorphism $\phi$ is not autonomous if it has a non-constant orbit in a primitive class $\al$ that is not simple, i.e. $\mathrm{si}(\al(\phi,z)) > 0$ for a fixed point $z.$ Now consider the diffeomorphism \[\phi_\lambda \in \Ham(\Sigma)\] constructed earlier. Fix coefficients $0 < \mu, \nu < 1,$ with $\frac{\mu}{\nu} \in \Q_{>0}$ to be determined later. Consider all $\la \in \R_{>0},$ such that $m = \frac{\mu}{L} \lambda, n = \frac{\nu}{L} \lambda$ satisfy $m,n \in \Z_{>0}.$ Clearly there are arbitrarily large $\la$ with this property.

Consider the free homotopy class $\beta_\la= \oa^m \ob^n /\mathrm{conj}.$ By \cite{HassScottIntersections} \[\mathrm{si}(\beta_\la) = m\cdot n + (m-1)\cdot(n-1) > 0,\] and hence $\beta_\la$ is not simple.

\begin{clm}\label{Claim: Actions in 2d proof}
There exist $\mu,\nu \in (0,1)$ with $\frac{\mu}{\nu} \in \Q_{>0},$ such that for $\la \gg 1$ the diffeomorphism $\phi_\la$ has exactly $4$ non-degenerate critical orbits with distinct action values in the class $\beta_\la$ and minimal action gap $D_\la = c \cdot \la + O(1),$ as $\la \to \infty.$
\end{clm}

Claim \ref{Claim: Actions in 2d proof} implies that for each $0<c'<c,$ $0<\eps$ and $\la$ sufficiently large, there exists a window $(\A-c' \cdot \la - \eps,\A+c' \cdot \la + \eps),$ with $\A$ the action value of a critical orbit of $\phi_\la$ in class $\beta_\la,$ that does not contain the critical value of any other orbit in this class. Hence the comparison map \[HF^{(\A-c' \cdot \la - \eps,\A+ \eps)}(\phi_\la)_{\beta_\la} \to HF^{(\A - \eps,\A+c' \cdot \la + \eps)}(\phi_\la)_{\beta_\la},\] is an isomorphism for each $0 < \eps_1,\eps_2 \leq \eps$. We claim that this implies that \[d(\phi_\la,\theta) \geq \frac{c'}{2} \cdot \la\] for any autonomous $\theta \in \Ham(M).$ Indeed assume that $d(\phi_\la,\theta) < \frac{c'}{2} \cdot \la.$ Consider the diagram of continuation maps \[HF^{(\A-c' \cdot \la - \eps,\A+\eps)}(\phi_\la)_{\beta_\la} \to HF^{(\A-\frac{c'}{2} \cdot \la - \eps ,\A+\frac{c'}{2} \cdot \la + \eps)}(\theta)_{\beta_\la} \to HF^{(\A - \eps,\A+c' \cdot \la + \eps)}(\phi_\la)_{\beta_\la}.\]

Since by naturality of continuation maps this composition is a comparison map that is by construction an isomorphism, we see that $HF^{(\A-\frac{c'}{2}\cdot \la - \eps ,\A+\frac{c'}{2} \cdot \la + \eps)}(\theta)_{\beta_\la}$ does not vanish, and hence there exists a critical orbit of $\theta$ in the primitive non-simple homotopy class $\beta_\la,$ in contradiction to Constraint \ref{constraint-1}. Note, that here we are working with Floer homology of a degenerate Hamiltonian diffeomorphism, but the implication still holds by continuous dependence on parameters of solutions of ODE's in a fixed time interval.

\begin{pf}(Claim \ref{Claim: Actions in 2d proof})
We calculate the fixed points of $\phi_\la$ in class $\beta_\la$ and their actions. We will prove later that all fixed points $(x_0,y_0)$ of $\phi_\la$ in class $\beta_\la$ lie in the square $A,$ and in fact their class in $\pi_1(C,0_A)$ is exactly $a^m b^n.$ For now, consider all such fixed points. It turns out that the fixed point equation is
\begin{align}
\label{Equation: 2d fixed point 1} & u(x_0) = \mu\\
\label{Equation: 2d fixed point 2} & u(-y_0) = \nu.
\end{align}

Indeed since, as we show below, the intermediate paths $\{f^t_V(z)\}$ and $\{f^t_H f_V(z)\}$ have classes \[[\{f^t_V(z)\}] = a^m, \; [\{f^t_H f_V(z)\}] = b^n\] in $\pi_1(C,A).$ Hence lifting to the universal cover of $C_*$ the fixed point equation in this class translates to \[(x_0,y_0) = HV \circ r_{n} \circ \til{f} \circ VH \circ r_{m} \circ \til{f} (x_0,y_0),\] where  for $p \in \Z,$  \[r_p:(-1,1) \times \R \to (-1,1) \times \R,\] \[r_p(x,y) = (x, y- p\cdot L)\] is the {\em reduction map} that maps $(-1,1) \times (p\cdot L-1,p\cdot L+1)$ isomorphically onto $(-1,1) \times (-1,1)$, and $\til{f}$ is the lift of $f$ to $\til{C}_* = [-1,1] \times \R$ given by the isotopy $\{f^t\}.$ That is $\til{f} (x,y) = (x,y+\la u(x)).$ Let us derive in detail the fixed point equation. We compute \[(x_0,y_0) = HV \circ r_{n} \circ \til{f} \circ VH \circ r_{m} \circ \til{f} (x_0,y_0) = HV \circ r_{n} \circ \til{f} \circ VH (x_0, y_0 + \la u(x_0) - \la \mu) =\]\[ = HV \circ r_{n} \circ \til{f} (-y_0 - \la u(x_0) + \la \mu, x_0)=\]\[ = HV \circ r_{n} (-y_0 - \la u(x_0) + \la \mu, x_0 + \la u(-y_0 - \la u(x_0) + \la \mu)) = \] \[ = HV (-y_0 - \la u(x_0) + \la \mu, x_0 + \la u(-y_0 - \la u(x_0) + \la \mu) - \la \nu) = \]\[= (x_0 + \la u(-y_0 - \la u(x_0) + \la \mu) - \la \nu, y_0 + \la u(x_0) - \la \mu).\]

Hence the fixed point equation reduces to the system of equations
\begin{align}
&\la u(x_0) - \la \mu = 0, \\
&\la u(-y_0 - \la u(x_0) + \la \mu) - \la \nu = 0,
\end{align}

which is clearly equivalent to the system (\ref{Equation: 2d fixed point 1}),(\ref{Equation: 2d fixed point 2}).

A calculation shows that for a sufficiently small smoothing $u$ of $u_0,$ as described in Section \ref{Section: example}, the system (\ref{Equation: 2d fixed point 1}),(\ref{Equation: 2d fixed point 2}) has four solutions \[(x_0,y_0) = z_{\overrightarrow{\varepsilon}}=(\vareps{1}(1-\mu),\vareps{2} (1-\nu)),\] where ${\overrightarrow{\varepsilon}}= (\vareps{1},\vareps{2}) \in \{\pm 1\}^{2}.$ Note that these solutions lie in the region where $u=u_0.$.


Choose $\eta_{\beta_\la}:=\ga(\oa)^{\# m}\# \ga(\ob)^{\# n}$ as the reference loop in the class $\beta_\la.$ An easy calculation shows that the action of the fixed point $z_{\overrightarrow{\varepsilon}}$ corresponding to $\overrightarrow{\varepsilon} = (\vareps{1},\vareps{2}) \in \{\pm 1\}^2$  is \[\A_{H}(z_{\overrightarrow{\varepsilon}}) = \frac{\la}{2}(\vareps{1} (1-\mu)^2 - \vareps{2}(1-\nu)^2).\]
Hence for a proof of Claim \ref{Claim: Actions in 2d proof} it is sufficient to choose $\mu,\nu$ outside the line $\{\mu - \nu = 0\}.$

We leave the rather easy proof of non-degeneracy to the interested reader.

It remains to show that any fixed point $(x_0,y_0)$ in class $\beta_\la$ lies in $A.$ First of all, clearly $(x_0,y_0) \in A \cup B,$ since otherwise, by consideration of the supports of $f_H,f_V,$ one would have had to have $m=0$ or $n=0.$ There are $4$ types of intermediate paths $\{f_V^t(x_0,y_0)\} \# \{f_H^t(x_0,y_0)\},$ $A \to A \to A,$ $A \to B \to A,$ $B \to A \to B,$ $B \to B \to B.$ These orbits have representatives of type $\oa^k \ob^l, \oa^k \oc \ob^l, \oc^{-1}\oa^k \ob^l, \oc \oa^k \oc^{-1} \ob^l$ in $\pi_1(\Sigma)$ respectively (here $k,l$ are generic notation for a pair integers). We see that since we can consider the problem in $\cL (C)$ instead of $\cL \Sigma,$ by a hyperbolic geometry argument, $\oa^m \ob^n,$ being a cyclically reduced word in the free group $Free\langle\oa,\ob,\oc\rangle$ is conjugate to one of these representatives if and only if it is of type $A \to A \to A$ and $k=m,l=n.$ Hence $(x_0,y_0) \in A.$
\end{pf}

This proves Claim \ref{Claim: Actions in 2d proof} and hence finishes the alternative proof in dimension $2.$

\end{proof}

\section{Interaction with the Conley conjecture}\label{Section:Conley}

The goal of this section is to prove Theorem \ref{thm-generic}.

\begin{pf}(Theorem \ref{thm-generic})
For the purposes of this section all $k$-periodic orbits will be assumed contractible. Put $\G = \Ham,$ and for $\phi \in \G$ denote by $\cP^k(\phi)$ the set of all its primitive $k$-periodic orbits. Define \[\G_k := \{\phi \in \G|\; \text{all} \; x \in \cP^k(\phi) \; \text{are non-degenerate} \},\]
\[\G'_k := \{\phi \in \G_k| \; \cP^k(\phi) \; \text{is non-empty}\}.\]

Note that $\G_k$ is an open dense subset of $\G$ in $C^\infty$-topology, and hence \[\G_\infty := \bigcap_{k \geq 1} \G_k\] is a residual subset of $\G.$ Moreover, $\G'_k$ is an open subset of $\G$ and by Constraint \ref{constraint-2} for $k \geq 2,$ \begin{equation}\label{Equation: Conley subset}
\G'_k \subset \Ham \setminus \Aut.
\end{equation} The non-degenerate case of the Conley conjecture \cite{SalamonZehnder92} implies that if $\phi \in \G_1$ then $\cP^k(\phi)$ is non-empty for some $k \geq 2.$ Therefore if $\phi \in \G_\infty$ then, being in particular in $\G_1,$ it lies in $\G'_k$ for some $k \geq 2.$ Thus \[\G_\infty \subset \bigcup_{k \geq 2} \G'_k,\] whence, as $\G_\infty,$ being residual, is dense in the $C^\infty$-topology,   $\bigcup_{k \geq 2} \G'_k$ is dense in the $C^\infty$-topology. Hence by (\ref{Equation: Conley subset}), we conclude that $\Ham \setminus \Aut$ contains an open dense set in the $C^\infty$-topology.

Proposition \ref{Propsotion: Conley} implies that this $C^\infty$ open and dense subset is contained in a Hofer-open subset $\bigcup_{k \geq 2} \{\phi \in \G|\;w_{k,pt_M}(\phi)>0\},$ which by Proposition \ref{Proposition: properties of the invariant} Item \ref{itm:property 4} is in turn contained in $\Ham \setminus \Aut,$ which completes the proof of Theorem \ref{thm-generic}.
\end{pf}

\section{Discussion}\label{Section: discussion}

\subsection{Extension to monotone manifolds}

Consider a spherically monotone symplectic manifold $(M,\om)$ - that is \[\langle [\om],B \rangle = \kappa \cdot \langle c_1(TM), B \rangle\] for all \[B \in \text{Image}(\pi_2(M) \to H_2(M,\Z))\] for a constant $\kappa > 0.$ Assume that $(M,\om)$ is also $\al$-toroidally-monotone for a class $\al \in \pi_0 (\cL M).$ This means that \[\langle [\om],B \rangle = \kappa \cdot \langle c_1(TM), B \rangle\] (the same $\kappa$!) for all \[B \in \text{Image}(\pi_1(\cL_\al M) \to H_2(M,\Z)),\] where the map sends a loop in $\cL_\al M$ to the fundamental class of the associated $T^2$-cycle in $M.$ We make a choice of a base-point $\eta_\al \in \cL_\al M$ and of a trivialization of $\eta_\al^* (TM,S^1)$ as a symplectic vector bundle over $S^1.$

Let $k \geq 2$ be an integer. The Hamiltonian Floer homology $HF_*(\til{\phi}^k)_\al$ in class $\al$ with coefficients in a base field $\K$ can be  defined as a graded module over the Novikov ring (i.e., the ring of semi-infinite Laurent series) $\K[[q^{-1},q],$ where $\deg(q) = 2 N_M,$ twice the minimal Chern number of $(M,\om).$ The filtered Floer homology $HF^{(-\infty,b)}_*(\til{\phi}^k)_\al$ will be a module over  the ring  $\K[[q^{-1}]]$
of formal power series in $q^{-1}$. Noting that since multiplication by $q$ shifts the degree by $2 N_M$ the homology groups $HF^{(-\infty,b)}_r(\til{\phi}^k)_\al$ in a given degree $r$ will be finite dimensional vector spaces over the base field $\K.$ A similar remark holds for degree $r$ Floer homology groups $HF^{(a,b)}_r(\til{\phi}^k)_\al$ in finite action windows.

Thus it is straightforward to see that the considerations of Section \ref{Section: invariant} are fully applicable in this case and give us an invariant $w_{k,\al,r}(\til{\phi})$ and a $\Z_k$ persistence module \[(V_r(k,\al,\til{\phi}),A_r(k,\al,\til{\phi})) = (HF^{(-\infty,\cdot)}_r(\til{\phi}^k)_\al,[\rR_k(\til{\phi})])\] with $A_r(k,\al,\til{\phi})^k = \id.$ Here the $\Z_k$ action  $A_r(k,\al,\til{\phi})=[\rR_k]$ is induced by the loop rotation operator.

Put $\text{Area}(q) = \Om:= \kappa \cdot N_M,$ for the minimal positive symplectic area of a sphere in $M.$ Now we note that multiplication by $q$ on the chain level gives isomorphisms \[(q\cdot): V_r(k,\al,\til{\phi}) \to V_{r+2N_M}(k,\al,\til{\phi})^\Om\] and \[(q\cdot): HF^{(a,b)}_r(\til{\phi})_\al \to HF^{(a,b)+\Om}_{r+2N_M}(\til{\phi})_\al.\] This means, since the definitions of our invariants depend only on action-differences, that $\max_{r \in \Z} (w_{k,\al,r}) = \max_{r \in \Z/(2N_M)} (w_{k,\al,r}).$ Considering the actions of loops in $\Ham$ on Floer homology, which result in isomorphisms up to shifts in grading and in the action-filtration, we see that \[w_{k,\al}(\til{\phi}):= \max_{r \in \Z/(2N_M)} w_{k,\al,r}(\til{\phi})\] depends only on the projection $\phi$ of $\til{\phi}$ to $\Ham.$ This invariant $w_{k,\al}(\phi)$ still satisfies all the necessary properties.

Similarly, considering the actions of loops in $\Ham$ on Floer homology, we see that if $\phi \in \G_p$ has a root $\psi$ of order $p,$ then for any degree $r$ and $k=p,$ prime, the $\Z_p$ persistence module $(V_r(p,\til{\phi},\al),A_r(p,\til{\phi},\al))$ given by the filtered Floer homology of $\til{\phi}^p$ in class $\al$ for any lift $\til{\phi}$ of $\phi$ to $\til{\Ham},$ and the loop rotation operator on it, is a full $p$-th power {$\Z_p$ persistence module}. Indeed, taking a lift $\til{\psi}$ of $\psi$ to $\til{\Ham},$ we obtain a special lift $\til{\psi}^p$ of $\phi,$ whose persistence module $(V_r(p,\til{\psi}^p,\al),A_r(p,\til{\psi}^p,\al))
$ is a full $p$-th power {$\Z_p$ persistence module} (with full root of order $p$ furnished by $A_r(p^2,\til{\psi},\al)$). Moreover, $(V_r(p,\til{\phi},\al),A_r(p,\til{\phi},\al))$ is isomorphic, by the action of a corresponding loop in $\Ham$ to $(V_{r'}(p,\til{\psi}^p,\al)^c,A_{r'}(p,\til{\psi}^p,\al)^c)$ for some $r'$ with a certain shift $c \in \R,$ and is therefore a full $p$-th power {$\Z_p$ persistence module}.

In fact, this property is invariant under shifts of the persistence module, and up to shifts in degree, as explained above, and hence we can consider only a finite number ($2N_M$) of degrees $r.$ Moreover, we see that the multiplicity-sensitive spread $\mu_p(V,A)$ of a $\Z_p$ persistence module $(V,A)$ from Definition \ref{df-sens-spread} is clearly independent of shifts. Hence we conclude that for Hamiltonian diffeomorphism $\phi \in \Ham$ the multiplicity-sensitive spread $\mu_p(\phi)$ is well-defined, and satisfies Theorem \ref{thm-mu2-squares}, giving a lower bound on the Hofer distance between $\phi$ and any $\theta \in \Pow_p.$


It seems likely that this construction applies to negatively monotone symplectic manifolds, for which by a result of Ginzburg-G\"{u}rel \cite{GGNegMon} the Conley conjecture holds. Hence one expects to generalize Theorem \ref{thm-generic} to this case.

As for generalization of Theorems \ref{thm-main} and \ref{thm-main-squares} we consider the following adaptations of our example.

\begin{example}\label{Example: blow-up}
Consider the surface $\Sigma$ of genus at least $4,$ and a symplectically aspherical manifold $N$ of $\dim N = 2n-2.$ Consider the diffeomorphism $\phi_\la$ of $\Sigma$ and its stabilization $\phi_\la \times \id_N.$ Take the one-point size $\Omega$ blow-up \[M=(\mathrm{Bl}_{(x,y)}X,\om_M)\] of \[X=\Sigma \times N\] at a point $(x,y)$ with $x$ outside the union of annuli where $\phi_\la$ is supported. Denote by $E \subset M$ the exceptional divisor. Let $\overline{\al}_\la \in \pi_0(\cL M)$ be the lift of the class $\al \times pt_N \in {\pi_0(\cL(\Sigma \times N)).}$
As $\dim M = 2n,$ the evaluation of $[\om_M]$ on $\text{Image}(\pi_2(M) \to H_2(M,\Z))$ and on $\text{Image}(\pi_1(\cL_{\overline{\al}_\la} M) \to H_2(M,\Z))$ is generated by its values on $\C P_1 \subset E,$ and moreover $\langle [\om_M], \C P^1\rangle = \Omega,$ $\langle c_1(TM), [\C P^1] \rangle = n-1.$

Look at the fixed points of $\phi_\lambda^p$ in the class $\alpha_\la.$ Let their actions be $a_1,...,a_l$ so that $a_i -a_j \sim \lambda,$
and the indices, calculated with respect to a well-chosen trivialization of $T\Sigma$ along the basic cycles $\oa,\ob$ on $\Sigma$, are $I_1,...,I_l.$

The trivialization we use: on one of the two annuli (identified with $T^*S^1$) the trivialization of $T\Sigma$ is given by the vertical Lagrangian distribution, and on the other by the Lagrangian distribution tangent to the zero section (clearly, these two trivializations are homotopic) - it is easy to see that these two trivializations glue well together to a trivialization along $\al_\la.$

Since by the proof of Proposition \ref{Proposition: example}, Equation \eqref{Equation: 2 parabolic matrices}, the linearization of $\phi_\la^p$ at each fixed point is the product of $2p$ parabolic $2\times 2$ matrices, and since multiplication with $\id_N$ contributes a finite correction to Conley-Zehnder indices, there exists $C > 0$ such that $|I_j| \leq C$ for all $j.$ The crucial point
here is that $C$ is independent of $\lambda.$

Recall that $\Omega$ stands for the area of $\C P^1$ in the exceptional divisor.

Now fix degree $r =I_1,$ for instance. Observe that the actions with this index are
$a_1,$ as well as $a_k + \Omega \cdot q_k,$ where $q_k$ is the solution of $I_1= I_k + 2(n-1)q_k,$ $k \neq 1.$
Observe that $q_k$ are bounded and thus the differences $|a_1 - (a_k + \Omega \cdot q_k)|$ are of order $\lambda.$ It follows that in the degree $r:= I_1$ we have that $w_{p,\alpha,r}(\til{\phi}) \sim \lambda,$ and similarly, $\mu_p(\phi) \sim \la,$ as required.
\end{example}

\begin{example}
Using the same argument as in Example \ref{Example: blow-up}, we can show analogues of Theorems \ref{thm-main}, \ref{thm-main-squares} for the manifold $\Sigma \times N,$ where now $N$ is any spherically monotone symplectic manifold. The diffeomorphisms we take are stabilizations $\phi_\la \times \id_N$ of $\phi_\la$ from Section \ref{Section: example}, and the class is $\al_\la \times pt_N.$
\end{example}

It would be interesting to see how the methods of this paper extend further to non-monotone symplectic manifolds.

\subsection{Spectral spread and persistence modules}\label{Subsection: spread and persistence modules}
We start with a modified and simplified version of the spectral spread $w_{k,\alpha}$.
The modification affects the choice of the windows in Definition \ref{Definition: the invariant}: instead of arbitrary windows $(a,b)$ we work with semi-infinite windows
$(-\infty, b)$. Let us spell this out.

For $k \in \Z, \; k\geq 1,$ and a Hamiltonian $H \in \cH$ we make the following definition. Put \[\rS_k:=\rT_k -\id \;,\] where $T_k$ is the loop rotation operator. Composing $\rS_k$ with the comparison map $j_d: HF^{(-\infty ,b)}(\Hk)_\al \to HF^{(-\infty ,b+d)}(\Hk)_\al,$ we obtain a map \[j_d \circ \rS_k = \rS_k \circ j_d: HF^{(-\infty,b)}(\Hk)_\al \to HF^{(-\infty,b+d)}(\Hk)_\al.\]

\begin{df}\label{Definition: the invariant-1}(Modified spread)
\[\hat{w}_{k,\al}(H):= \sup \{d \geq 0: \; j_d \circ \rS_k \neq 0 \; \text{for some window} \; (-\infty,b)\}\]
\end{df}

Note that $\hat{w}_{k,\al} \leq w_{k,\al}$. Nevertheless our proof of Theorem \ref{thm-main} remains
unchanged if one uses $\hat{w}$ instead of $w$.

The main point of this section is that the modified spectral spread $\hat{w}_{k,\alpha}$ admits a concise description in terms of persistent homology (see \ref{Subsection: enriched persistence modules}) In what follows we use Floer homology with coefficients
in any field $\K$.

In the context of Floer homology and loop rotation operators, we have to deal with persistence modules
equipped with an extra data, the family of morphisms $A_t: V_t \to V_t$. Furthermore, we assume that

\begin{equation}\label{eq-As-commute}
A_t \circ \theta_{st}= \theta_{st} \circ A_s \;\; \forall s < t\;.
\end{equation}

We denote the enriched persistence module by  $(V,\theta;A)$.

\medskip
\noindent \begin{df}\label{Definition: algebraic spread}
The spread $\hat{w}(V,\theta;A)$ is defined
as the supremum of $d>0$ such that for some $s \in \R$
the map $\theta_{s,s+d}(A_s-\id) \neq 0$.
\end{df}

It turns out that the spread can be expressed in terms of the barcode of an auxiliary
persistence module which is defined as follows. Put $W_t = \{v \in V_t \;:\; A_tv_t =v_t\}$ and $L_t = V_t/W_t$.
We write $p_t: V_t \to L_t$ for the natural projection.
By \eqref{eq-As-commute} $\theta_{st}(W_s) \subset W_t$,
and hence the morphism $\theta_{st}$ descends to a morphism $\pi_{st}: L_s \to L_t$.
We get a new persistence module $(L,\pi)$.

\medskip

Put $\beta(L,\pi)$ for the longest finite bar in the bar-code of the persistence module $(L,\pi).$ This notion is related to Usher's notion of boundary depth (cf. \cite{UsherBoundaryHofer}).

\medskip
\noindent
\begin{thm}\label{thm-barcode} $\hat{w}(V,\theta;A) = \beta(L,\pi).$
\end{thm}

\begin{proof} Note that for $s < t$ and $v \in V_s$
$$\theta_{st}(A_s-\id)v = (A_t -\id)\theta_{st}v  \neq 0 \Leftrightarrow $$
$$A_t\theta_{st}v \neq \theta_{st}v \Leftrightarrow \theta_{st}v \notin W_t  \Leftrightarrow p_t(\theta_{st}(v))= \pi_{st}(p_s(v)) \neq 0\;.$$
This yields
$$\theta_{st}(A_s-\id) = (A_t -\id)\theta_{st}  \neq 0 \Leftrightarrow \pi_{st} \neq 0\;.$$ Note that the reverse implication holds since $p_s$ is onto. Hence $\hat{w}(V,\theta;A) = \beta(L,\pi).$
\end{proof}

%

\medskip
\begin{problem}
It would be interesting to extract the invariant $w_{k,\al}$ in a similar way from two-parametric persistence modules.
\end{problem}

\medskip

It is easy to see that for two persistence modules $K,L,$ \[|\beta(K) - \beta(L)| \leq 2 d_{bottle}(\cB(K),\cB(L)),\] which shows (cf. Equation \eqref{eq-dist-inter-shad-1}) that for an integer $k \geq 2,$ $\al \in \pi_0(\cL M),$ and $\phi \in \cG_k,$ the map $\phi \mapsto \hat{w}_{k,\al}(\phi) = \hat{w}(V(k,\al,\phi),A(k,\al,\phi))$ is Lipschitz with constant $2k$ with respect to the Hofer metric. However, by considering continuation maps directly and separating the positive and the negative parts of the Hofer metric, we show in Proposition \ref{Proposition: properties of the invariant}, Item \ref{itm:property 3} that in fact it is Lipschitz with constant $k.$ It would be interesting to modify the bottleneck distance on the space of barcodes to account for this discrepancy.

\bs
\noindent{\bf Acknowledgments.} We are grateful to Paul Seidel for his generous help with this
paper (see Section \ref{subsec-constraints} above). We thank Mohammed Abouzaid, Paul Biran, Strom Borman, Fr\'{e}d\'{e}ric Bourgeois, Octav Cornea, Maia Fraser, Michael Khanevsky, Michael Usher, Shmuel Weinberger and Jun Zhang for useful discussions.

Our understanding of persistent homology profited from lectures by Yaniv Ganor, Asaf Kislev and Daniel Rosen at a guided reading course delivered by L.P. at Tel Aviv University.
We thank all of them.

Parts of this paper have been written during L.P.'s stay at the University of Chicago and ETH-Z\"{u}rich, and E.S.'s stay at the Hebrew University of Jerusalem and CRM, University of Montreal. We thank these institutions for their warm hospitality.

Preliminary results of this paper have been presented at symplectic workshops in the Lorentz Center, Leiden (Summer, 2014) and in the Clay Institute, Oxford (Fall, 2014). We are indebted to the organizers, Hansj\"{o}rg Geiges, Viktor Ginzburg, Federica Pasquotto and Dominic Joyce, Alexander Ritter and Ivan Smith, respectively, for this opportunity.

\medskip

\bibliographystyle{alpha}

\bigskip

\noindent
\begin{tabular}{ll}
Leonid Polterovich & Egor Shelukhin\\
School of Mathematical Sciences & CRM\\
Tel Aviv University & University of Montreal\\
polterov@post.tau.ac.il & egorshel@gmail.com \\
\end{tabular}

\end{document}